\documentclass[a4paper]{amsart}
\pdfoutput=1
\usepackage{mathmode}
\usepackage{bm}
\usepackage{yhmath}
\usepackage{tikz,tikz-cd}
\usepackage[margin=1.3in]{geometry}
\usepackage{graphicx,array,subfigure}
\usepackage{enumitem}
\usepackage{stackengine,scalerel}
\usepackage{amsthm}
 \usepackage{multirow}
\graphicspath{ {./Xu_figures/} }

\def\HSK{hyperbolic-type satellite knot} %%%%%%%
\def\HSKs{hyperbolic-type satellite knots}
\def\HT{hyperbolic-type}

\def\SCS{unoriented characterising slope}%%%%%%%%%%%%%%%%%%%%
\def\SCSs{unoriented characterising slopes}
\def\SetSCS{unoriented characterising set of slopes}
\def\SC{unoriented-characterising}

\def\Z{\mathbb{Z}}
\def\abs#1{\left |#1 \right |}

\def\SK{\mathcal{N}} %I-bundle over Klein bottle
\def\EM{Eudave-Mu\~noz }

\def\pisi{\overline{\psi}}     %induce on slope
\def\[#1{[\![#1]\!]}

\def\O{\mathcal{O}}
\def\slope{\text{slp}}
\def\Qinf{\mathbb{Q}\cup \{\infty\}}
\def\Q{\mathbb{Q}}
\def\slp{\text{slp}^\ast}

\def\cpl#1{S^3\setminus \mathring{N}(#1)}

\def\ttt{\stackrel{\cong}{\longrightarrow}}

\def\quotient#1#2{%
    \raise1ex\hbox{$#1$}\Big/\lower1ex\hbox{$#2$}%
}
\def\quo#1#2{%
    \raise0.6ex\hbox{$#1$}\big/\lower0.6ex\hbox{$#2$}%
}
\def\Hom{\mathrm{Hom}}

\def\F{f^\dagger}

\def\limi{\varprojlim}

\def\Zx{\widehat{\Z}^{\times}}

\def\SL{\mathrm{SL}}

\def\Jn{\sigma_2^{-1}\sigma_1(\sigma_1\sigma_2)^{3n}}
\def\tensor{\widehat{\mathbb{Z}}\otimes_{\mathbb{Z}}}

\newcommand\ttimes{\mathbin{\ThisStyle{\ensurestackMath{%
  \stackengine{-1\LMpt}{\SavedStyle\times}
  {\SavedStyle_{\hstretch{.9}{\mkern1mu\sim}}}{O}{c}{F}{T}{S}}}}}
\tikzset{%
  symbol/.style={
    draw=none,
    every to/.append style={
      edge node={node [sloped, allow upside down, auto=false]{$#1$}}
    },
  },
}

\newtheorem*{THMB}{\autoref{THMB}}
\newtheorem*{CORC}{\autoref{CORC}}
\newtheorem*{THMA}{\autoref{inthm: Dehn filling}}
\date{\today}

\title[Profinite rigidity and Dehn filling of cusped hyperbolic 3-manifolds]{Profinite rigidity witnessed by Dehn fillings of cusped hyperbolic 3-manifolds}
\author{Xiaoyu Xu}
\address{Beijing International Center for Mathematical Research\\
Peking University\\
 Beijing 100871, P.R. China }
\email{xuxiaoyu@stu.pku.edu.cn}

\begin{document}
\begin{sloppypar}
\maketitle
%\pagestyle{empty}
%\maketitle

\begin{abstract}
Any profinite isomorphism between two cusped finite-volume hyperbolic 3-manifolds carries profinite isomorphisms between their  Dehn fillings. With this observation, we prove that some cusped finite-volume hyperbolic 3-manifolds are profinitely rigid among all compact, orientable 3-manifolds, through detecting their exceptional Dehn fillings. In addition, we improved a criteria for profinite rigidity of a hyperbolic knot complement or a {\HSK} complement among compact, orientable 3-manifolds, through examining its characterising slopes. 

We obtain infinitely many profinitely rigid examples, including:  the complement of the Whitehead link, Whitehead sister link, $\frac{3}{10}$ two-bridge link; specific surgeries on one component of these links; the complement of (full) twist knots $\mathcal{K}_n$,  Eudave-Mu\~noz knots $K(3,1,n,0)$, Pretzel knots $P(-3,3,2n+1)$,  $5_2$ knot\iffalse, Whitehead double of the right-handed trefoil with twisting number $2$\fi; the Berge manifold, and many more.
\iffalse
We prove that some cusped finite-volume hyperbolic 3-manifolds are profinitely rigid among all compact, orientable 3-manifolds, through detecting their exceptional Dehn fillings. The profinite rigid examples include: the Berge manifold; the complement of the Whitehead link, Whitehead sister link, $\frac{3}{10}$-two bridge link, twist knots $\mathcal{K}_n$, and Eudave-Mu\~noz knots $K(3,1,n,0)$;  $\frac{1}{n}$ and $n-\frac{1}{2}$-surgery on the $\frac{3}{10}$-bridge link; $\frac{5}{2}$-surgery on the Whitehead link; and the manifold $M_{14}$ introduced in \cite{GW} etc. 

\fi%\com{To be revised}
\end{abstract}

\setcounter{tocdepth}{1}
\tableofcontents %tabel of contents

\section{Introduction}\label{SEC: Intro}

The collection of all finite quotient groups of a finitely generated group $G$, denoted by $\mathcal{C}(G)$, reflects information of the original group $G$. When $G=\pi_1M^3$ is a 3-manifold group, $\mathcal{C}(G)$ corresponds to the lattice of finite regular coverings of $M$, which is in fact the collection of all finite deck transformation groups. 

It is known that among finitely generated groups, the profinite completion encodes full data of finite quotients.

%It is a useful idea to distinguish properties of a (finitely generated) group $G$ through its finite quotients $\mathcal{C}(G)=\left\{[Q]\mid G\text{ surjects }Q,\,\abs{Q}<+\infty\right\}$, where $[Q]$ denotes the isomorphism class of $Q$. When $G=\pi_1M^3$ is a 3-manifold group, $\mathcal{C}(G)$ corresponds to the lattice of finite regular coverings of $M$, ie the collection of all finite deck transformation groups. It is known that among finitely generated groups, the profinite completion encodes full data of finite quotients.
\begin{definition}\label{DEF: Profinite completion}
Let $G$ be a group, the \textit{profinite completion} of $G$ is defined as $$\widehat{G}=\limi\limits_{N\lhd_f G} \quo{G}{N}\;,$$ where $N$ ranges over all finite-index normal subgroups of $G$ (the notation $\lhd_f$ denotes finite-index normal subgroup).
\end{definition}
\begin{fact}[{\cite{DFPR82}}]\label{Finite quotient}
    For two finitely generated groups $G_1$ and $G_2$, $\mathcal{C}(G_1)=\mathcal{C}(G_2)$ if and only if $\widehat{G_1}\cong \widehat{G_2}$.
\end{fact}

%The geometrization theorem implies that a 3-manifold is largely determined by its fundamental group. For instance, the fundamental group of a closed, orientable, prime 3-manifold determines its homeomorphism type except for the special case of lens spaces \cite[Theorem 2.2]{AFW15}. 
%As we are studying the fundamental group $\pi_1M^3$ through its finite quotients $\mathcal{C}(\pi_1M)$, it is natural to ask whether $\mathcal{C}(\pi_1M)$ (or equivalently, $\widehat{\pi_1M}$) determines the homeomorphism type of $M$. Or generally, which properties of $M$ can be distinguished from $\widehat{\pi_1M}$.
Since 3-manifolds are largely determined by their fundamental groups, it is natural to ask whether $\mathcal{C}(\pi_1M)$ (or equivalently, $\widehat{\pi_1M}$) determines the homeomorphism type of $M$.

\begin{definition}\label{indef: almost rigidity}
Let $\mathscr{M}$ be a class of compact  3-manifolds. We say that a manifold $M\in\mathscr{M}$ is  \textit{profinitely rigid} in $\mathscr{M}$, if for any $N\in \mathscr{M}$, $\widehat{\pi_1N}\cong \widehat{\pi_1M}$ implies $N\cong M$.
\iffalse
For any $M\in \mathscr{M}$, we denote $\Delta_{\mathscr{M}}(M)=\{ N\in \mathscr{M}\mid \widehat{\pi_1N}\cong \widehat{\pi_1M}\}$. 
We say that a manifold $M$ is \textit{profinitely rigid} in $\mathscr{M}$ if $\Delta_{\mathscr{M}}(M)=\{M\}$, and $M$ is \textit{profinitely almost rigid} (or \textit{finite genus}) in $\mathscr{M}$ if $\Delta_{\mathscr{M}}(M)$ is finite.  
The class $\mathscr{M}$ is called \textit{profinitely (almost) rigid} if every $M\in \mathscr{M}$ is profinitely (almost)  rigid. 
\fi
\end{definition}
%In practice, $\mathscr{M}$ is usually considered as the collection of all compact, orientable (and boundary incompressible) 3-manifolds.

\iffalse
To avoid the trivial ambiguities, we include the following convention.
\begin{convention}\label{conv1}
In our context, we always assume that a compact 3-manifold has no  boundary spheres. Indeed, the boundary spheres can be capped off by 3-balls while preserving the fundamental group.
\end{convention}
\fi

\begin{convention}
To avoid the trivial ambiguities, 
in our context, we always assume that a compact 3-manifold has no  boundary spheres. Indeed, the boundary spheres can be capped off by 3-balls while preserving the fundamental group.
\end{convention}

The profinite completion of a 3-manifold group does reflect substantial information of the 3-manifold itself. In \cite{Xu}, the author showed that among compact orientable 3-manifolds with empty or toral boundary, the profinite completion of fundamental group determines the homeomorphism type of a 3-manifold up to finitely many possibilities.

In fact, profinite rigidity in 3-manifolds is closely related with geometrization. In a series of works, Wilton-Zalesskii \cite{WZ17,WZ17b,WZ19} showed  that among compact, orientable, irreducible 3-manifolds with empty or incompressible toral boundary, the profinite completion determines whether the 3-manifold is geometric; it further determines the geometry type within the geometric case, and determines the JSJ-decomposition in the non-geometric case. In both cases, it is crucial to study the profinite rigidity of the geometric pieces, either with or without boundary. 

The questions of profinite rigidity in seven of the eight geometries proposed by Thurston have been profoundly understood. For instance, $Sol$-manifolds are profinitely ``almost'' rigid \cite{GPS80}, but not profinitely rigid by \cite{Stebe,Fun13}; in fact, the profinite rigidity in $Sol$-manifolds can be reduced to a pure number-theoretic problem. \cite{Wil17}, together with \cite[Corollary 8.3]{Xu} for the bounded case, provides a complete profinite classification for orientable Seifert fibered spaces. Thus, the most challenging part in profinite rigidity of 3-manifolds is the hyperbolic case.

Liu \cite{Liu23} first showed that the profinite completion of the fundamental group of a finite-volume hyperbolic 3-manifold (including both closed and cusped hyperbolic manifolds) determines its homeomorphism type up to finitely many possibilities. Yet, whether profinite rigidity holds generally in the hyperbolic case is still unknown. 
There are a few examples of profinitely rigid closed hyperbolic 3-manifolds, for instance the Weeks manifold  \cite{BMRS20}, the $\mathrm{Vol(3)}$ manifold \cite{BR22}, and $0$-surgery of the knots ${6_2}$ and ${6_3}$ in the Rolfsen table by \cite{Cw22}. \revised{The results for these manifolds are even stronger.  In fact, the fundamental groups of these manifolds are   profinitely rigid among all finitely generated residually finite groups.} %Some families of fibered hyperbolic 3-manifolds with cusps are also known to be profinitely rigid among all compact 3-manifolds. These examples include hyperbolic once-punctured torus bundles over $S^1$ by \cite{BRW17}, including the figure-eight knot complement  \cite{BR20}; and hyperbolic four-punctured sphere bundles over $S^1$ by \cite{Cw24}, including the ``magic'' manifold. The proof of profinite rigidity for these families are actually based on the congruence omnipotence property of the corresponding mapping class groups. 
\revised{As for cusped hyperbolic 3-manifolds, the figure-eight knot complement is proven to be profinitely rigid among all compact 3-manifolds by Bridson-Reid \cite{BR20}.}

\subsection{Profinite rigidity from Dehn filling}
In this paper, we mainly focus on cusped hyperbolic 3-manifolds. %Some of our results can also be generalized to mixed manifolds with hyperbolic pieces carrying boundary components. 
The main advantage of the cusped case is that we can utilize the key technique, so called peripheral $\Zx$-regularity, introduced in \cite{Xu}. Based on this technique, we prove that any two profinitely isomorphic  cusped hyperbolic 3-manifolds have profinitely isomorphic Dehn fillings. %This enables us to infer useful information from the finite quotients of a Kleinian group through Dehn fillings. 

Let $M$ be a compact, orientable 3-manifold  with boundary components consists of tori $\partial_1M,\cdots, \partial_n M$. For each  $1\le i \le n$, suppose $c_i$ is a slope on $\partial_iM\cong T^2$, and $c_i$ is allowed to be an ``empty'' slope. We denote by $M_{(c_i)}=M_{c_1,\cdots,c_n}$ the corresponding {\em Dehn filling} along these slopes, where $\partial_iM$ is skipped over if $c_i$ is an empty slope.

\begin{mainthm}\label{inthm: Dehn filling}
Suppose $M$ and $N$ are orientable cusped finite-volume hyperbolic 3-manifolds, and $\widehat{\pi_1M}\cong \widehat{\pi_1N}$. Then, there exists a homeomorphism $\Psi: \partial M\to \partial N$ such that for any boundary slopes $(c_i)$ on $\partial M$ (allowing empty slopes),  $\widehat{\pi_1M_{(c_i)}}\cong \widehat{\pi_1N_{(\Psi(c_i))}}$; moreover,  this   isomorphism respects the peripheral structure (see \autoref{DEF: Peripehral structure}).
\end{mainthm}

Combining \autoref{inthm: Dehn filling} with the aforementioned results on profinite detection of geometrization, we obtain the following corollary. 

\begin{corollary}\label{incor: exceptional}
 In the context of \autoref{inthm: Dehn filling}, the boundary homeomorphism $\Psi$ sends exceptional slopes bijectively to exceptional slopes. 
\end{corollary}

We remark that \autoref{inthm: Dehn filling} can also be generalized to mixed manifolds with hyperbolic pieces carrying boundary components, see \autoref{PROP: Mixed Dehn filling}.

As an application of \autoref{incor: exceptional}, we found infinitely many examples of profinitely rigid cusped hyperbolic 3-manifolds, including non-fibered ones, through  detecting their exceptional Dehn fillings. In fact, the following results are based on a series of profound studies on exceptional Dehn fillings, including \cite{Ber91,Gor98,GW,Lee06,Lee07}, which characterise certain cusped hyperbolic 3-manifolds by their extreme patterns of exceptional Dehn fillings. This enables us to narrow down the manifolds profinitely isomorphic to the specific ones to some limited candidates that can be easily distinguished from each other.

\begin{mainthm}\label{MAIN}
The following cusped finite-volume hyperbolic 3-manifolds are profinitely rigid among all compact, orientable 3-manifolds.
\begin{enumerate}[label=(\arabic*), leftmargin=23.5pt]
\item\label{m.W} The Whitehead link complement $\cpl{\mathcal{W}}$ (\autoref{Fig: W}).
\item\label{m.WS} The Whitehead sister link complement $\cpl{\mathcal{WS}}$ (\autoref{Fig: WS}).
\item\label{m.L} The complement of the two-bridge link with Schubert normal form $\frac{3}{10}$, denoted by $\cpl{\mathcal{L}}$ (\autoref{Fig: L}).
\item\label{m.Kn} The complement of the (full) twist knots $\mathcal{K}_n$, $n\in \Z\setminus\{0,1\}$ (\autoref{Fig: Kn}).
\item\label{m.Jn} The complement of the braid knots $\mathcal{J}_n$ presented by $\Jn$, where $n\in \Z\setminus\{-1,0,1\}$ (\autoref{Fig: Jn}).
\item\label{m.EM} The complement of the \EM knots $K(3,1,n,0)$, $n\in \Z\setminus\{ 0\}$ (\autoref{Fig: EM}); in particular, $K(3,1,1,0)$ is the Pretzel knot $P(-2,3,7)$ (\autoref{Fig: Pr}).%, including the complement of the $(-2,3,7)$-Pretzel knot.
\item\label{m.LA} The Berge manifold, ie the complement of the link $L_A$ in $S^3$ (\autoref{Fig: LA}). 
\item\label{m.52} $5,-1,-2,-4$, and $\frac{5}{2}$-Dehn surgery on one component of the Whitehead link shown in \autoref{Fig: W}.
\item\label{m.Lss} $n-\frac{1}{2}$-Dehn surgery on one component of the $\frac{3}{10}$-two bridge link $\mathcal{L}$ shown in \autoref{Fig: L}, where $n\in \Z$. 
\item\label{m.LB} $0$-surgery on the $L_B^{(1)}$ component of the link $L_B$ shown in \autoref{Fig: LB}.
\item\label{m.LBr} $0$-surgery on the $L_B^{(1)}$ component and $r$-surgery on the $L_B^{(2)}$ component of the link $L_B$ (\autoref{Fig: LB}), where $r\in \mathbb{Q}\setminus\{0,4\}$. %\com{Check fiber or not}
\item\label{m.M14} The manifold $M_{14}$ introduced in \cite{GW}, which is the double branched cover of $S^2\times I$ along the tangle $\mathcal{Q}$ shown in \autoref{Fig: M14}.%\com{Check fiber or not or anything.}
\end{enumerate}
\end{mainthm}

\begin{figure}[h]
\centering
\subfigure[The Whitehead link $\mathcal{W}$]{
\includegraphics[width=4.1cm]{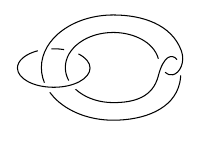}
\label{Fig: W}
}
\subfigure[The Whitehead sister link $\mathcal{WS}$]{
\includegraphics[width=4.2cm]{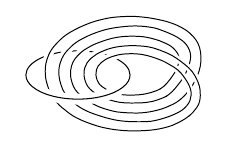}
\label{Fig: WS}
}
\subfigure[The $\frac{3}{10}$-two bridge link $\mathcal{L}$]{
\includegraphics[width=4.2cm]{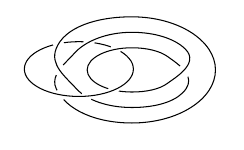}
\label{Fig: L}
}
\caption{Link diagrams for $\mathcal{W}$, $\mathcal{WS}$, $\mathcal{L}$}
\end{figure}

\begin{figure}[h]
\centering
\subfigure[The twist knot $\mathcal{K}_n$ ($n\ge 0$)]{
\includegraphics[width=4cm]{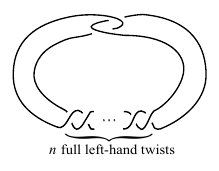}
}
\hspace{4mm}
\subfigure[The twist knot $\mathcal{K}_n$ ($n<0$)]{
\includegraphics[width=4cm]{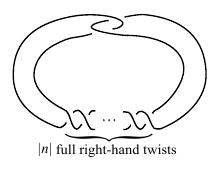}
}
\caption{Knot diagrams for $\mathcal{K}_n$}
\label{Fig: Kn}
\end{figure}

\begin{figure}[h]
\centering
\subfigure[The braid knot $\mathcal{J}_n$]{
\includegraphics[width=3cm]{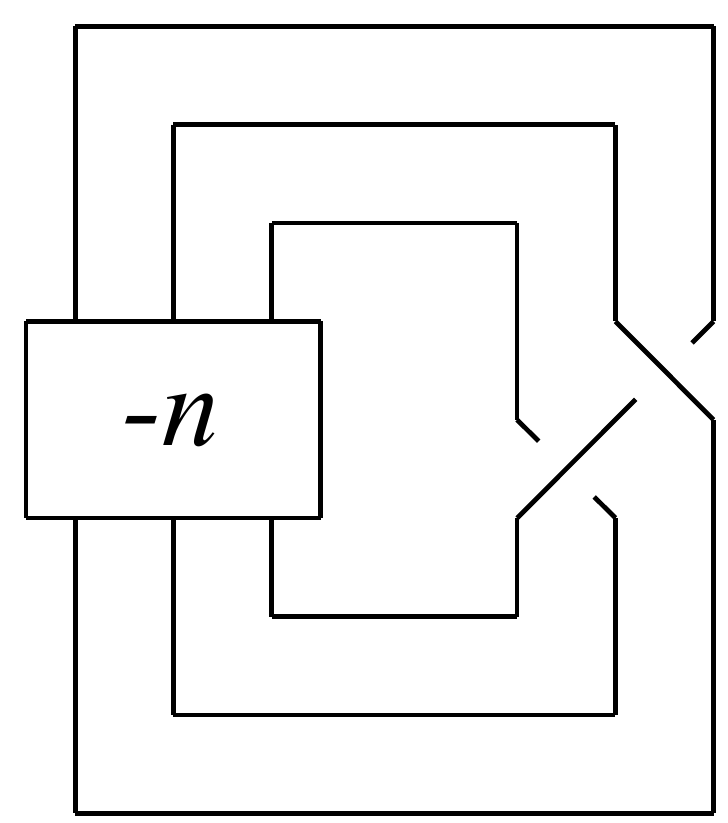}
}
\hspace{4mm}
\subfigure[Notation for full twists]{
\includegraphics[width=5.75cm]{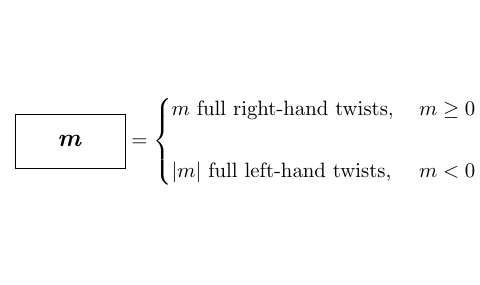}
}
\caption{Knot diagram for $\mathcal{J}_n$}
\label{Fig: Jn}
\end{figure}

\begin{figure}[h]
\centering
\subfigure[The \EM knot $K(3,1,n,0)$]{
\includegraphics[width=5.65cm]{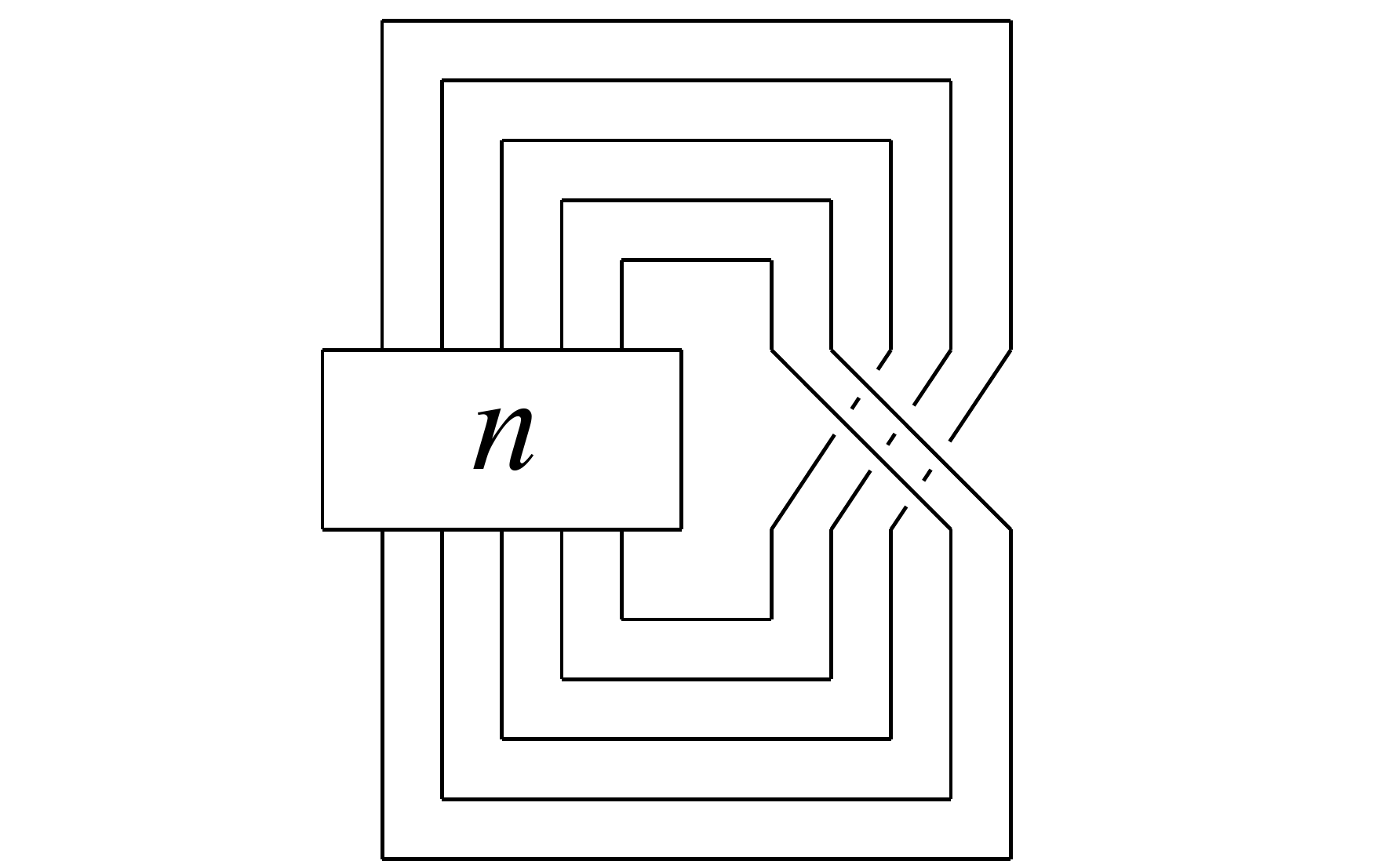}
\label{Fig: EM}
}
\hspace{3mm}
\subfigure[$K(3,1,1,0)$ is the $(-2,3,7)$-Pretzel knot]{
\includegraphics[width=5.45cm]{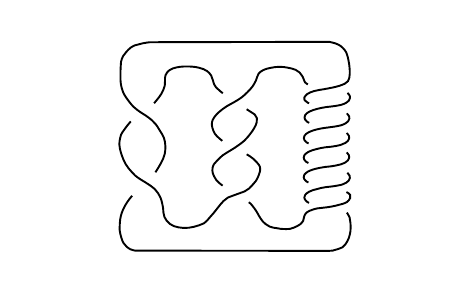}
\label{Fig: Pr}
}
\caption{Knot diagram for $K(3,1,n,0)$}
\end{figure}

\begin{figure}[h]
\centering
\subfigure[The Berge manifold $\cpl{L_A}$]{
\includegraphics[width=4.6cm]{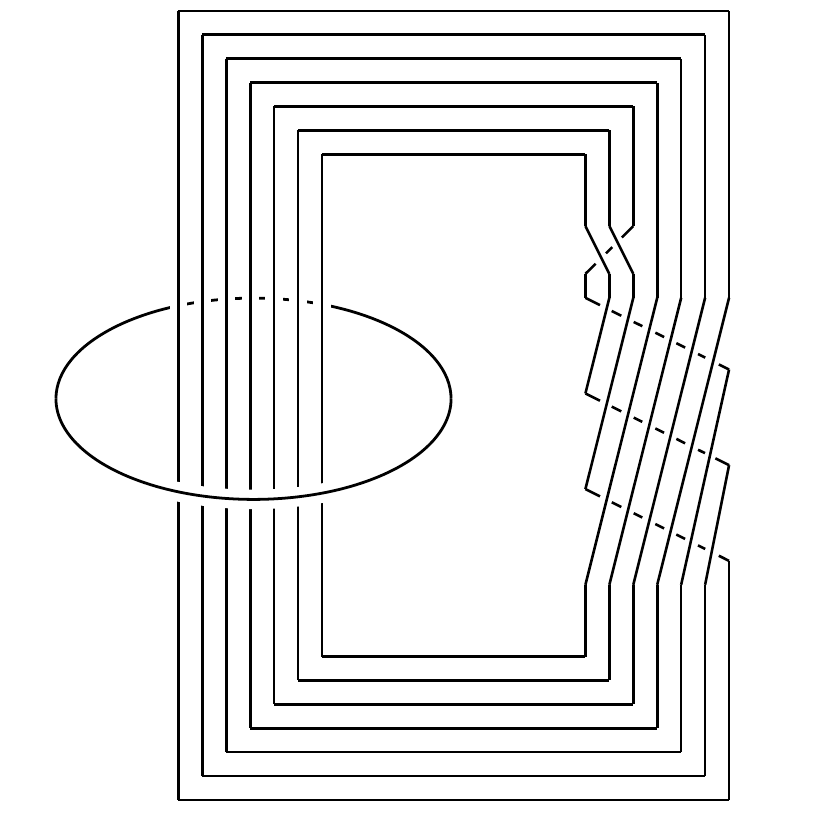}
\label{Fig: LA}
}
\hspace{15mm}
\subfigure[The link $L_B$]{
\includegraphics[width=3.9cm]{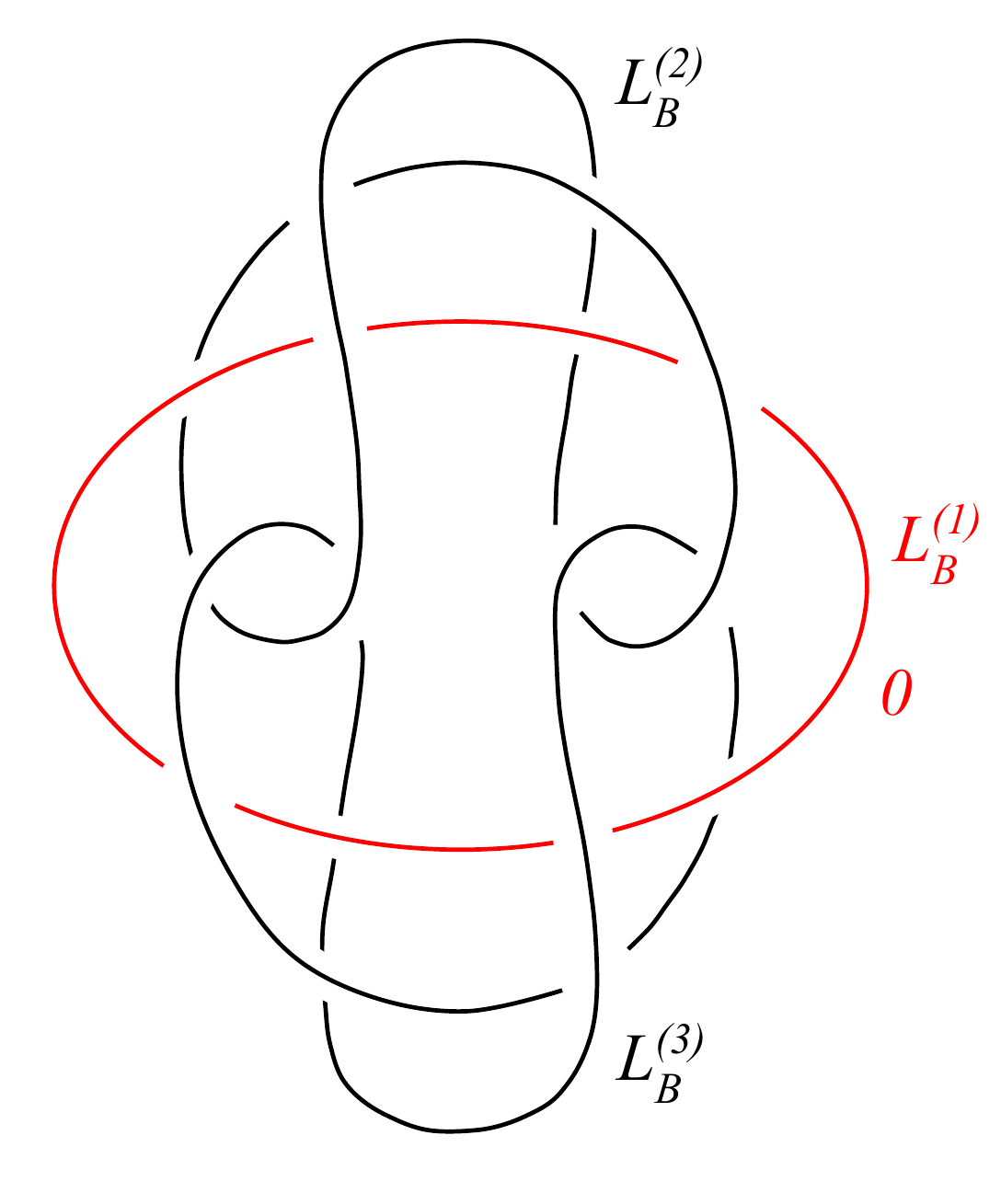}
\label{Fig: LB}
}
\caption{Link diagrams for $L_A$ and $L_B$}
\end{figure}

\begin{figure}[h]
\centering
\includegraphics[width=4cm]{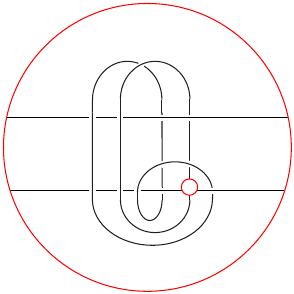}
\caption{The tangle $\mathcal{Q}$ in $S^2\times I$, \cite[Figure 22.14]{GW}}
\label{Fig: M14}
\end{figure}

\iffalse
\revise{
\begin{remark}
\iffalse
\begin{enumerate}[label=(\arabic*), leftmargin=*]
\item The twist knot $\mathcal{K}_{-1}$ is in fact the figure-eight knot, and the profinite rigidity of the figure-eight knot complement was originally proven by Bridson and Reid \cite{BR20}.
\item The $L_B^{(2)}$ and $L_B^{(3)}$ components of the link $L_B$ can be viewed as a $4$-braid  in the complement of the $L_B^{(1)}$ component, which is  homeomorphic to $D^2\times S^1$. Thus, the $0$-surgery on the $L_B^{(1)}$ component is in fact a four-punctured sphere bundle over $S^1$. The profinite rigidity of this manifold then follows from a general result recently proven by Cheetham-West \cite{Cw24}. 
\item The integral Dehn surgeries $\mathcal{W}(-1)$, $\mathcal{W}(-2)$, $\mathcal{W}(-4)$, $\mathcal{W}(5)$ are in fact once-punctured torus bundle over circle according to  \cite[Proposition 3]{HMW92}, and their profinite rigidity follows from a general result by Bridson, Reid and Wilton \cite{BRW17}. 
%
%\item 
\end{enumerate}
However, our method of proof is different from \cite{BR20,Cw24,BRW17}, which may provide some new insights.
\fi
\end{remark}
}
\fi

\begin{remark}
It is conjectured that the profinite completion of  the fundamental group of a finite-volume hyperbolic 3-manifold determines its hyperbolic volume. The examples listed in \autoref{MAIN} include some cusped hyperbolic 3-manifolds with small volumes. %For instance, the Whitehead link complement and the Whitehead sister link complement are exactly the minimal volume orientable hyperbolic 3-manifolds with two cusps \cite{Ago10}, and $\mathcal{W}(\frac{5}{2})$ is one of the second smallest volume orientable hyperbolic 3-manifolds with one cusp \cite{GMM09}. 
\revised{
For instance, the figure-eight knot complement (which is the knot $\mathcal{K}_{-1}$ in  \ref{m.Kn} and also the $-1$-surgery on the Whitehead link in \ref{m.52}) and the figure-eight sibling manifold (which is the $5$-surgery on the Whitehead link in \ref{m.52}) are exactly the two minimal volume one-cusped orientable hyperbolic 3-manifold \cite{CM01}. The Whitehead link complement in \ref{m.W} and the Whitehead sister link complement in \ref{m.WS} are exactly the minimal volume two-cusped orientable hyperbolic 3-manifolds \cite{Ago10}. The $\frac{5}{2}$-surgery on the Whitehead link is one of the second smallest volume one-cusped orientable hyperbolic 3-manifold, and the $-2$-surgery on the Whitehead link is one of the third smallest volume one-cusped orientable hyperbolic 3-manifold \cite{GMM09}. 
}
\end{remark}

\subsection{Application to knot complements}

Knot or link complements provide easily accessible examples for 3-manifolds, and the profinite properties of knot complements have long been a popular topic. 
\revised{Boileau-Friedl \cite[Theorem 1.5]{BF19} showed that the complement of any torus knot is profinitely rigid among all knot complements in $S^3$. In fact, its fundamental group is profinitely rigid among the fundamental groups of all compact orientable 3-manifolds by \cite[Corollary 8.3]{Xu}. Similar results were proved for graph knots by Wilkes \cite{Wilkes2019} that the fundamental group of the complement of any graph knot is profinitely rigid among the fundamental groups of all compact orientable 3-manifolds, which implies the profinite rigidity of  prime graph knot complements among all knot complements. 
Some knot invariants are also proven to be detected by the profinite completion of the knot group, including fiberedness and genus by Boileau-Friedl \cite[Theorem 1.2]{BF19}, and the Alexander polynomial by Ueki \cite{Ueki}. 
}
\iffalse
Wilkes \cite{Wilkes2019} showed that the fundamental group of the complement of a graph knot is profinitely rigid among all fundamental groups of compact, orientable 3-manifolds, which implies the profinite rigidity of  prime graph knot complements among all knot complements in $S^3$. This covers partial results on torus knots proposed earlier by Boileau-Friedl \cite{BF19}. Ueki \cite{Ueki} showed that that if two knot complements in $S^3$ are profinitely isomorphic, then they have the same Alexander polynomial. Cheetham-West \cite{Cw23} introduced a criteria for the profinite rigidity of a hyperbolic knot complement among all knot complements in $S^3$, through examining the characterising property and the  profinite rigidity  of its $0$-surgery.
\fi

Profinite properties of knot complements are more computable, and they provide examples that lucidly illustrates our techniques. 
As an application of \autoref{inthm: Dehn filling}, we are able to detect hyperbolic knot complements and {\HSK} complements through the finite quotient  groups, and further to match up the Dehn surgeries of a profinitely isomorphic pair of such knots.

\begin{definition}
A satellite knot $K\subseteq S^3$ is  {\em{\HT}} if $\partial N(K)$ belongs to a hyperbolic piece in the JSJ-decomposition of the knot complement $\cpl{K}$.
\end{definition}

For example, any satellite knot with a hyperbolic pattern is a {\HSK}.  

\begin{mainthm}%[cf. \autoref{PROP: Detect knot complement}, \autoref{THM: Detect satellite knot} and \autoref{LEM: Knot complement}]
\label{inthm: Knot}
%\text{ }
\iffalse
\begin{enumerate}[label=(\arabic*), leftmargin=*]
\item\label{1.10-1} Let $M=\cpl{K}$ be the complement of a hyperbolic knot, resp. a {\HSK}\iffalse, $K\subset S^3$\fi . Suppose $N$ is a compact, orientable 3-manifold such that $\widehat{\pi_1M}\cong \widehat{\pi_1N}$. Then, $N\cong \cpl{K'}$ is also the complement of a hyperbolic knot, resp. a {\HSK}.
\item Let $K$ and $K'$ be hyperbolic knots or {\HSKs} in $S^3$. Suppose that \iffalse$\widehat{\pi_1}(\cpl{K})\cong \widehat{\pi_1}(\cpl{K'})$\fi$\widehat{\pi_1(S^3\setminus K)}\cong \widehat{\pi_1(S^3\setminus{K'})}$. Then, there exists $\sigma\in \{1,-1\}$ such that for any $r\in \Qinf$, the $r$-surgery of $K$ is profinitely isomorphic to the $\sigma r$-surgery of $K'$.
\end{enumerate}
\fi
Let $M=\cpl{K}$ be the complement of a hyperbolic knot, resp. a {\HSK}\iffalse, $K\subset S^3$\fi . Suppose $N$ is a compact, orientable 3-manifold such that $\widehat{\pi_1M}\cong \widehat{\pi_1N}$. Then, $N\cong \cpl{K'}$ is also the complement of a hyperbolic knot, resp. a {\HSK}. In addition, there exists $\sigma\in \{1,-1\}$ such that for any $r\in \Qinf$, the $r$-surgery of $K$ is profinitely isomorphic to the $\sigma r$-surgery of $K'$.
\end{mainthm}

%Based on \autoref{inthm: Knot}, we can improve the conclusions by Cheetham-West \cite{Cw23} in determining the profinite rigidity of  knot complements through their characterising slopes. 

\revised{
Cheetham-West \cite{Cw23} found a sufficient condition for a hyperbolic knot complement to be profinitely rigid among all knot complements in $S^3$, through examining the characterising property and the  profinite rigidity  of its $0$-surgery. 
Based on \autoref{inthm: Knot}, we can improve this criteria. 
}

\begin{definition}\label{DEF: characterizing slope}
Let $K\subseteq S^3$ be a knot.  
\iffalse
\begin{enumerate}[label=(\arabic*), leftmargin=*]
\item
A slope $\alpha \in \Q\setminus\{0\}$ on $\partial N(K)$ is a {\em strongly characterising slope} for $K$, if for any knot $J\subseteq S^3$, 
the Dehn surgery $J(\alpha)$ is homeomorphic to $K(\alpha)$ if and only if $J$ is isotopic to $K$.
\item
$0$ is a {\em strongly characterising slope} for $K$ if for any knot $J\subseteq S^3$, 
$J(0)\cong K(0)$ if and only if $J$ is isotopic to either $K$ or the mirror image of $K$.
\end{enumerate}
\fi
A slope $\alpha \in \Q$ on $\partial N(K)$ is an {\em{\SCS}} for $K$, if for any knot $J\subseteq S^3$, 
the Dehn surgery $J(\alpha)$ being homeomorphic to $K(\alpha)$ implies that $J$ is isotopic to $K$ or its mirror image.
\end{definition}

There is a slight difference between the definition of an {\SCS} and the usual definition of a characterising slope. 
Generally, %$\alpha$ is a characterising slope, if  $J(\alpha)\cong K(\alpha)$ by an orientation-preserving homeomorphism implies that $J$ is isotopic to $K$.  
when referring to a characterising slope, it is required that the homeomorphism $J(\alpha)\cong K(\alpha)$ is orientation-preserving, see \autoref{DEF: SCSCS} for details.  
However, we remove the orientation-preserving restriction here, and add a ``unoriented'' prefix  for distinction. \iffalse It is worth pointing out that when $\alpha \in \Q\setminus\{0\}$, $\alpha$ is a strongly characterising slope for $K$ implies that it is a characterising slope for $K$ in the usual sense. However, the logic reverses for $0$-slope; $0$ is a characterising slope for $K$ in the usual sense implies that $0$ is a strongly characterising slope for $K$. \fi
Regardless of this distinction, if $0$ is a characerising slope of $K$, then it is also an {\SCS} of $K$, see \autoref{LEM: Characterising}.

\begin{maincor}\label{THMB}
Let $K\subseteq S^3$ be either a hyperbolic knot or a {\HSK}. Suppose
\begin{enumerate}[label=(\arabic*), leftmargin=*]
\item there is an {\SCS} $\alpha\in\mathbb{Q}$ of $K$, and
\item the Dehn surgery $K(\alpha)$ is profinitely rigid among all closed, orientable 3-manifolds.
\end{enumerate}
Then the knot complement $\cpl{K}$ is profinitely rigid among all compact, orientable 3-manifolds.
\end{maincor}

As an application of \autoref{THMB}, we show the profinite rigidity of some knot complements in $S^3$, based on the examples listed in \cite{BS22,BS24} where $0$ is a characterising slope.

\begin{maincor}\label{CORC}
Let $K\subseteq S^3$ be one of the following knots: 
\begin{enumerate}[label=(\arabic*), leftmargin=*]
\item the $5_2$ knot in the Rolfsen table \cite{Rolfsen} (\autoref{Fig: 52});
\item the $15n_{43522}$ knot in the Hoste-Thistlethwaite-Weeks table \cite{HTW} (\autoref{Fig: 15n43522});
\item the Pretzel knots $P(-3, 3, 2n + 1)$, where $n\in \Z$ (\autoref{Fig: Pretzel});
\item the satellite knots $\mathcal{W}^+(T_{2,3}, 2)$ and $\mathcal{W}^-(T_{2,3}, 2)$ (\autoref{Fig: Wh}).
\end{enumerate}
Then $\cpl{K}$ is profinitely rigid among all compact, orientable 3-manifolds.
\end{maincor}

\begin{figure}[h]
\subfigure[The $5_2$ knot]{
\includegraphics[width=3cm]{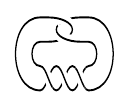}
\label{Fig: 52}
}
\hspace{2mm}
\subfigure[The $15n_{43522}$ knot]{
\includegraphics[width=3cm]{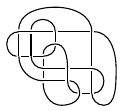}
\label{Fig: 15n43522}
}
\hspace{4mm}
\subfigure[The Pretzel knot $P(-3,3,2n+1)$]{
\includegraphics[width=4.45cm]{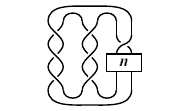}
\label{Fig: Pretzel}
}
\caption{The hyperbolic knots in \autoref{CORC}}
\end{figure}

\begin{figure}[h]
\subfigure[$\mathcal{W}^+(T_{2,3},2)$]{
\includegraphics[width=3cm]{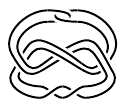}
\label{Fig: Wh+}
}
\hspace{5mm}
\subfigure[$\mathcal{W}^-(T_{2,3},2)$]{
\includegraphics[width=3cm]{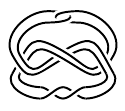}
\label{Fig: Wh-}
}
\caption{The satellite knots in \autoref{CORC}}
\label{Fig: Wh}
\end{figure}

\begin{remark}
Cheetham-West \cite[Theorem 1.1]{Cw23} showed that the complement of the $5_2$ knot, $15n_{43522}$ knot, and   Pretzel knots $P(-3, 3, 2n + 1)$ are profinitely rigid among all knot  complements in $S^3$. We enlarge this scope to all compact, orientable 3-manifolds. In fact, \iffalse the profinite rigidity of these three (families of) knot complements\fi this improvement can be  deduced from \autoref{inthm: Knot}. %(\ref{1.10-1}).
\end{remark}

In fact, \autoref{inthm: Dehn filling}, ie the profinite detection of Dehn filling, shows that the profinite completion of fundamental group seems to impose strong restrictions on a cusped hyperbolic 3-manifold. This may provide more confidence in the following conjecture.
\begin{conjecture}[Part of {\cite[Question 9]{Reid13}}]
Are all cusped hyperbolic 3-manifolds profinitely rigid among compact, orientable 3-manifolds?
\end{conjecture}

\subsection*{Structure of the paper}
\autoref{SEC: Preliminary} includes preliminary materials for profinite groups, especially the profinite completion of 3-manifold groups. 

In \autoref{SEC: Dehn filling}, we establish the main technique \autoref{inthm: Dehn filling}, which enables us to profinitely match up the Dehn fillings between a profinitely isomorphic pair of cusped hyperbolic 3-manifolds. 

In \autoref{subsec: knot}, we apply this technique to knot complements in $S^3$ and prove \autoref{inthm: Knot} as a combination of \autoref{PROP: Detect knot complement}, \autoref{THM: Detect satellite knot} and \autoref{LEM: Knot complement}. This important observation  leads us to the proof of \autoref{THMB} and \autoref{CORC} in \autoref{subsec: BC}.

\autoref{SEC: Exceptional} serves as a preparation for the proof of \autoref{MAIN}. We show a more precise version of \autoref{incor: exceptional}, that the profinite completion distinguishes specific types of exceptional Dehn fillings, including non-aspherical fillings, toroidal fillings, and aspherical fillings containing embedded Klein bottles.  

Finally, in \autoref{SEC: Main}, we prove \autoref{MAIN} through detecting the extremely special patterns of exceptional slopes on the listed hyperbolic manifolds. \ref{m.W}, \ref{m.WS}, \ref{m.L} and \ref{m.M14} are proved as \autoref{THM: link complement}; \ref{m.Kn}, \ref{m.Jn} and \ref{m.EM} are proved as \autoref{Mthm: knot complement}; \ref{m.LA} is proved in \autoref{THM: Berge}; \ref{m.52} is proved in \autoref{THM: W(-5/2)}; \ref{m.LB} is proved in \autoref{THM: LB0}; \ref{m.Lss} and \ref{m.LBr} are proved in \autoref{THM: LBr}.

%\iffalse
\subsection*{Acknowledgement}
The author sincerely thanks Yu Huang for highly inspiring discussions in the surrounding topics, and also his advisor Yi Liu for useful comments. He would also like to thank the referee for careful proofreading and valuable suggestions. 
%\fi

\section{Preliminaries}\label{SEC: Preliminary}

\subsection{Profinite groups}
\revised{A {\em profinite group} $\Pi$ is the projective limit of an inverse system of finite groups $\{P_i\}_{i\in I}$ indexed over a directed partially ordered set $I$, namely $\Pi=\limi_{i\in I}P_i$. The profinite group $\Pi$ is equipped with the subspace topology inherited from the space $\prod_{i\in I}P_i$ assigned  with  the product topology of the discrete finite groups. With this topology, $\Pi$ becomes a topological group, and we always take this as the topology on $\Pi$. } 

Recall in \autoref{DEF: Profinite completion}, we have defined the profinite completion of an abstract group $G$ as a profinite group $\widehat{G}=\limi_{N\lhd_f G} G/N$. %\revise{The profinite completion $\widehat{G}$ is equipped with the subspace topology inherited from the space $\prod_{N\lhd_f G}G/N$ assigned  with  the product topology of the discrete finite groups. }
There is a canonical homomorphism $G\to \widehat{G}$, which sends $g\in G$ to $(gN)_{N\lhd_fG}\in \widehat{G}$. In addition, the profinite completion is functorial, ie any homomorphism $f: G_1\to G_2$, induces a continuous homomorphism $\widehat{f}:\widehat{G_1}\to \widehat{G_2}$ so that the following diagram commutes.
\begin{equation*}
\begin{tikzcd}
G_1 \arrow[r, "f"] \arrow[d] & G_2 \arrow[d] \\
\widehat{G_1} \arrow[r, "\widehat{f}"]   & \widehat{G_2}           
\end{tikzcd}
\end{equation*}

\revised{
For any two finitely generated groups $G_1$ and $G_2$, the theorem of Nikolov-Segal \cite{NS03} implies that any homomorphism from $\widehat{G_1}$ to $\widehat{G_2}$ as abstract groups is continuous. %Thus, when discussing profinite completions of finitely generated groups in the following context, we shall not distinguish between an abstract isomorphism and a continuous one .%, as we are mainly concerning finitely generated groups. 
Thus, when we say that the profinite completions of two finitely generated groups are isomorphic, in addition to a group-theoretical isomorphism, the isomorphism is automatically a homeomorphism with respect to their topologies defined in the beginning of this subsection. 
}

\revised{
\begin{lemma}[{\cite[Proposition 3.2]{Reid13}}]\label{lem: abelianize}
Let $G_1$ and $G_2$ be finitely generated groups such that $\widehat{G_1}\cong \widehat{G_2}$. Then, their abelianizations $G_1^{\mathrm{ab}}$ and $G_2^{\mathrm{ab}}$ are isomorphic. 
\end{lemma}
}
\subsection{Powers in a profinite group}\label{subsec: power}
\revised{
The profinite completion of $\Z$, denoted by $\widehat{\Z}$, is the ring of profinite integers, and it can be identified with the direct product of all the rings of $p$-adic integers $\Z_p$, where $p$ ranges through all prime numbers. We use the symbol $\Zx$ to denote the group of invertibles in $\widehat{\Z}$. In fact, $\Zx=\limi(\Z/n\Z)^\times=\prod \Z_p^{\times}$. 

Let $\Pi$ be a profinite group, and let $x$ be an element in $\Pi$. Then, there is a unique continuous homomorphism $\varphi_x^\Pi: \widehat{\Z}\to \Pi$ such that $\varphi_x^\Pi(1)=x$. The image $\varphi_x^\Pi(\widehat{\Z})$ is indeed the closed subgroup $\overline{\langle x\rangle }$ generated by $x$, where the overline denotes the topological closure in $\Pi$. For any $\lambda\in \widehat{\Z}$, we use the symbol $x^\lambda$ to denote $\varphi_x^\Pi(\lambda)$, which is referred to as the {\em $\lambda$-power} of $x$ in $\Pi$. When $\lambda\in \Zx$, it is clear that $\overline{\langle x \rangle}=\overline{\langle x^\lambda \rangle }$. 

Suppose $\Pi'\subseteq \Pi$ is a closed subgroup such that $x\in \Pi'$. Then, $\Pi'$ is also a profinite group by \cite[Proposition 2.2.1]{RZ10}, and $\varphi_x^\Pi$ is the composition of $\varphi_x^{\Pi'}$ with the inclusion map $\Pi'\hookrightarrow \Pi$. Thus, for any $\lambda\in \widehat{\Z}$, the $\lambda$-power of $x$ in $\Pi'$ is the same as the $\lambda$-power of $x$ in $\Pi$, and we are justified to use the notation $x^\lambda$ without specifying the ambient subgroup. 
%Moreover, if $\Pi$ is an abelian profinite group, $\Pi$ can be viewed as a $\widehat{\Z}$-module in a natural way, where $\lambda\cdot x$ is defined to be $x^\lambda$ for any $x\in \Pi$ and $\lambda\in \widehat{\Z}$. In fact, this $\widehat{\Z}$-action on $\Pi$ is continuous. 
We refer the readers to \cite[Chapter 4]{RZ10} for more details. 
}

\subsection{Residual finiteness}
An abstract group $G$ is called {\em residually finite} if the intersection of all its finite-index subgroups is trivial. It is easy to see that the canonical homomorphism $G\to \widehat{G}$ is injective if and only if $G$ is residually finite. \revised{Thus, when $G$ is residually finite, we always view $G$ as a subgroup of $\widehat{G}$ via the canonical homomorphism, and any element of $G$ is also regarded as an element of $\widehat{G}$. }

Hempel's theorem for residual finiteness \cite{Hem87}, together with the virtual Haken theorem \cite{Ago13}, yields the following result.
\begin{proposition}[{\cite{Ago13,Hem87}}]\label{PROP: RF}
The fundamental group of any compact 3-manifold is residually finite.
\end{proposition}

%In particular, $\pi_1M$ injects into $\widehat{\pi_1M}$.

\subsection{Profinite detection of hyperbolicity}
In our context, a {\em finite-volume hyperbolic 3-manifold} refers to a compact, orientable 3-manifold with empty or incompressible toral boundary, whose interior admits a complete Riemannian metric of constant sectional curvature $-1$.  

%The following proposition was proven by Wilton-Zalesskii \cite{WZ17} and \cite{WZ17b}, see also \cite[Theorem 4.18 and 4.20]{Reid:2018}.

In a series of works, Wilton-Zalesskii showed that the profinite completion detects whether a compact, orientable 3-manifold is finite-volume hyperbolic. The closed case was proven in \cite{WZ17}, and the cusped case was proven in \cite{WZ17b}, see also \cite[Theorem 4.18 and 4.20]{Reid:2018}.

Moreover, when $M$ is a cusped finite-volume hyperbolic 3-manifold, \cite[Proposition 3.1]{WZ19} showed that the conjugacy classes of peripheral subgroups in $\widehat{\pi_1M}$ are exactly the conjugacy classes of the maximal closed subgroups isomorphic to $\widehat{\Z}^2$. Therefore, the number of cusps is exactly the number of such conjugacy classes in $\widehat{\pi_1M}$. We conclude these properties as the following proposition. 

\begin{proposition}[{\cite{WZ17,WZ17b,WZ19}}]\label{PROP: Detect hyperbolic}
Let $M$ be an orientable finite-volume hyperbolic 3-manifold. Suppose $N$ is a compact, orientable 3-manifold so that $\widehat{\pi_1M}\cong \widehat{\pi_1N}$, then $N$ is also a finite-volume hyperbolic 3-manifold, and $N$ has the same number of cusps as $M$. In particular, $M$ is closed if and only if $N$ is closed. %In addition, $N$ is closed if $M$ is closed, and $N$ is cusped if $M$ is cusped.
\end{proposition}

\iffalse

\begin{proposition}[{\cite{WZ17,WZ19}} ]\label{PROP: Detect cusps}
Suppose $M$ and $N$ are orientable finite-volume hyperbolic 3-manifolds with cusps, and $\widehat{\pi_1M}\cong \widehat{\pi_1N}$. Then $M$ and $N$ have the same number of cusps.
\end{proposition}
\fi

\subsection{Peripheral structure}
Let $M$ be a compact, orientable 3-manifold with boundary. We use the notation $\partial_iM$ to denote a boundary component of $M$. Up to a choice of basepoint, the inclusion induces a homomorphism $\iota_{\ast}:\pi_1(\partial_iM)\to \pi_1(M)$, and $\partial_iM$ is {\em incompressible} if and only if $\iota_{\ast}$ is injective. The image of $\iota_{\ast}$ is a conjugacy representative of the {\em peripheral subgroup} corresponding to $\partial_iM$. When $\partial_iM$ is incompressible, by a slight abuse of notation, we sometimes omit the notation $\iota_\ast$ and simply denote $\pi_1\partial_iM$ as a peripheral subgroup in $\pi_1M$. 

\iffalse
\begin{proposition}\label{PROP: boundary}
Suppose $M$ is a compact orientable 3-manifold, and $\partial_iM$ is an incompressible boundary component. Then the peripheral subgroup $\pi_1\partial_iM$ is separable in $\pi_1M$, and $\pi_1M$ induces the full profinite topology on $\pi_1\partial_iM$.
\end{proposition}
\begin{proof}
{\cite[Theorem 6.19]{Wil19}} proved this proposition assuming that all boundary components of  $M$ are incompressible. Indeed, when $M$ has compressible boundary components, through a standard argument using loop theorem, we can attach a compression-body to each compressible boundary component, and obtain a boundary-incompressible 3-manifold $M^+$, so that the inclusion induces an isomorphism on the fundamental group $\pi_1(M)\ttt\pi_1(M^+)$. In fact, this process is realized by  repeatedly attaching 2-handles to the compressible boundary components, and then cap off the emerging boundary spheres by 3-handles. Then, 
$\partial_iM$ is still a boundary component of $M^+$, and $\pi_1\partial_iM$ can be viewed as a peripheral subgroup in $\pi_1(M^+)$. Thus, the result follows from the boundary-incompressible case by \cite[Theorem 6.19]{Wil19}.
\end{proof}
\fi

Let $G$ be a group and $H$ be a subgroup of $G$. We denote by $\overline{H}$ the closure of the image of $H$ in $\widehat{G}$ through the canonical homomorphism $G\to \widehat{G}$.

\begin{proposition}\label{COR: Peripheral subgroup}
Suppose $M$ is a compact orientable 3-manifold, and $\partial_iM\cong T^2$ is an incompressible toral boundary component. Then, via the inclusion map, there is an isomorphism $\overline{\pi_1\partial_iM}\cong \widehat{\pi_1\partial_iM}\cong \widehat{\Z}^2$.
\end{proposition}
\begin{proof}
According to \cite{Ham01}, any abelian subgroup of $\pi_1M$ is separable. As a result, 
%According to \autoref{PROP: boundary}, 
$\pi_1M$ induces the full profinite topology on $\pi_1\partial_iM\cong \Z^2$. The standard theories of profinite groups \cite[Lemma 3.2.6]{RZ10} implies that $\widehat{\pi_1\partial_iM}\to \widehat{\pi_1M}$ is injective. Note that the image of this homomorphism is exactly $\overline{\pi_1\partial_iM}$. Thus, $\overline{\pi_1\partial_iM}\cong \widehat{\pi_1\partial_iM}\cong \widehat{\Z}^2$. 
\iffalse as abelian profinite groups. 
In addition, $\pi_1\partial_iM$ is free-abelian with rank~$2$. According to \cite[Lemma 2.9]{Xu}, there is a canonical isomorphism $ \tensor \pi_1\partial_iM\cong \widehat{\Z}^2\ttt \widehat{\pi_1\partial_iM}$ as $\widehat{\Z}$-modules.
\fi
\end{proof}

The peripheral structure of a compact 3-manifold is given by the conjugacy classes of peripheral subgroups in the fundamental group of the 3-manifold. We extend this concept into the profinite settings.

Let $M$ and $N$ be two compact orientable 3-manifolds. For each boundary component of $M$, up to a choice of basepoints, we fix the homomorphism induced by inclusion $\iota_{i_{\ast}}: \pi_1(\partial_iM)\to \pi_1(M)$. Similarly, fix $\rho_{j_{\ast}}: \pi_1(\partial_jN)\to \pi_1(N)$.
\begin{definition}\label{DEF: Peripehral structure}
Let $M$ and $N$ be two compact orientable 3-manifolds. Suppose $f:\widehat{\pi_1M}\to \widehat{\pi_1N}$ is an isomorphism. We say that $f$ {\em respects the peripheral structure} if there is a one-to-one correspondence between the boundary components of $M$ and $N$, denoted by $\partial_iM\longleftrightarrow \partial_{\sigma(i)}N$, such that for each $i$, $f(\overline{\iota_{i_\ast}(\pi_1(\partial_iM))})$ is a conjugate of $\overline{\rho_{\sigma(i)_\ast}(\pi_1(\partial_{\sigma(i)}N))}$ in $\widehat{\pi_1N}$.
\end{definition}

We remark that \autoref{DEF: Peripehral structure} also includes compressible boundary components.  
In addition, when $M$ and $N$ have empty boundary, any isomorphism $f:\widehat{\pi_1M}\ttt\widehat{\pi_1N}$ respects the peripheral structure by our definition. 

Analogous to \autoref{Finite quotient}, the profinite completion of fundamental group together with the peripheral structure encodes the same data as the ``peripheral pair quotients'' of the fundamental group. As is pointed out in \cite[Theorem 2.5]{Xu}, $\widehat{\pi_1M}\cong \widehat{\pi_1N}$ by an isomorphism respecting the peripheral structure if and only if $\mathcal{C}(\pi_1(M);\iota_{1_\ast}(\pi_1(\partial_1M)),\cdots,\iota_{n_\ast}(\pi_1(\partial_nM)))=\mathcal{C}(\pi_1(N);\rho_{1_\ast}(\pi_1(\partial_1N)),\cdots,\rho_{n_\ast}(\pi_1(\partial_nN)))$.

\section{Peripheral $\Zx$-regularity and correspondence of Dehn fillings}\label{SEC: Dehn filling}

The key concept introduced in \cite{Xu} is the so-called peripheral $\Zx$-regularity. In this section, we will show that peripheral $\Zx$-regularity eventually implies that a  profinite
isomorphism  between two cusped finite-volume hyperbolic 3-manifolds carries profinite isomorphisms between their Dehn fillings. 

\subsection{Peripheral $\Zx$-regularity}

\iffalse
For preciseness, let us first review the related definitions.

\begin{proposition}[{\cite[Theorem 6.19]{Wil19}} ]\label{COR: Irreducible induce full top on peripheral}
Let $M$ be a compact, orientable, irreducible 3-manifold with incompressible boundary, then for each boundary component $\partial_iM$, any conjugacy representative of the peripheral subgroup $\pi_1\partial_iM$ is separable in $\pi_1M$, and  $\pi_1M$ induces the full profinite topology on $\pi_1\partial_iM$.
\end{proposition}
\cite[Proposition 6.1]{Xu} presents an alternative proof for this result. When $M$ has toral boundary, $\pi_1\partial_iM$ is free-abelian with rank $2$. Then according to \autoref{COR: Irreducible induce full top on peripheral}, the standard theories of profinite groups \cite[Lemma 3.2.6]{RZ10} and \cite[Lemma 2.9]{Xu} imply  a canonical isomorphism $\overline{\pi_1\partial_iM}\cong \widehat{\pi_1\partial_iM}\cong \tensor \pi_1\partial_iM$, where $\overline{\pi_1\partial_iM}$ denotes the closure of $\pi_1\partial_iM$ in $\widehat{\pi_1M}$.
\fi

%Recall that when $M$ is a compact, orientable 3-manifold  with incompressible toral boundary, we have established a canonical  isomorphism $\overline{\pi_1\partial_iM}\cong \widehat{\pi_1\partial_iM}\cong \tensor \pi_1\partial_iM$ according to \autoref{COR: Peripheral subgroup}, where $\overline{\pi_1\partial_iM}$ denotes the closure of the peripheral subgroup in $\widehat{\pi_1M}$.
\revised{Recall that for a compact orientable 3-manifold $M$ with toral boundary, we have $\pi_1M$ sitting inside $\widehat{\pi_1M}$ as a subgroup via the canonical homomorphism according to  \autoref{PROP: RF}, and so any peripheral subgroup $\pi_1\partial_iM$ is also naturally viewed as a subgroup of $\widehat{\pi_1M}$.}

\begin{definition}[Peripheral $\Zx$-regularity]\label{DEF: peripheral Zx-regularity}
Let $M$ and $N$ be compact, orientable 3-manifolds with incompressible toral boundary. Suppose $f:\widehat{\pi_1M}\ttt\widehat{\pi_1N}$ is an isomorphism.  Let $\partial_i M$ be a boundary component of $M$, and let $\pi_1\partial_iM$ be a conjugacy representative of the corresponding peripheral subgroup. We say that $f$ is \textit{peripheral $\Zx$-regular} at $\partial_iM$ if the following holds.
\begin{enumerate}[label=(\arabic*), leftmargin=*]
\item\label{AAAAA1} There exist a boundary component $\partial_j N$ of $N$, a conjugacy representative of the peripheral subgroup $\pi_1\partial_jN$, and an element $g\in \widehat{\pi_1N}$ such that $f(\overline{\pi_1\partial_iM})=g^{-1} \cdot \overline{\pi_1\partial_jN}\cdot g$.
\item %Let $f'=C_g\circ f$, where $C_g$ denotes the conjugation by $g$, so that $f'$ restricts to an isomorphism $\pi_1\partial_iM\otimes_{\Z}\widehat{\Z}\cong \overline{\pi_1\partial_iM}\ttt \overline{\pi_1\partial_jN}\cong \pi_1\partial_jN \otimes_{\Z}\widehat{\Z}$.
\iffalse
Let $C_g$ denote the conjugation by $g$, and $C_g\circ f$ restricts to an isomorphism
\begin{equation*}
\begin{tikzcd}
C_{g}\circ f:\;\tensor \pi_1\partial_iM\cong \overline{\pi_1\partial_iM} \arrow[r, "\cong"] & \overline{\pi_1\partial_jN}\cong \tensor\pi_1\partial_jN .
\end{tikzcd}
\end{equation*}
There exist an element $\lambda\in \Zx$ and an isomorphism $\psi:\pi_1\partial_iM\ttt \pi_1\partial_jN $  such that $C_{g}\circ f=\lambda \otimes \psi$, where $\lambda$ denotes the scalar multiplication by $\lambda$ in $\widehat{\Z}$.
\fi
There exist an element $\lambda\in \Zx$ and an isomorphism $\psi:\pi_1\partial_iM\ttt \pi_1\partial_jN $  such that $f(\gamma)=g^{-1}\cdot \psi(\gamma)^\lambda \cdot g$ for any $\gamma \in \pi_1\partial_iM$. 
\end{enumerate}
\iffalse
\begin{enumerate}[label=(\arabic*), leftmargin=*]
\item There exists a peripheral subgroup $Q=\pi_1(\partial_jN)$ of $N$ such that $f(\overline{P})$ is a conjugate of $\overline{Q}$.
\item Let $\widetilde{f}$ denote $f$ composed with a conjugation in $\widehat{\pi_1N}$, so that $\widetilde{f}$ restricts to an isomorphism $P\otimes_{\Z}\widehat{\Z}\cong \overline{P}\ttt \overline{Q}\cong Q\otimes_{\Z}\widehat{\Z}$.
Then $\widetilde{f}|_{\overline{P}}=\lambda \cdot {\psi}$, for some isomorphism $\psi\in \Hom_{\mathbb{Z}}(P,Q)$.
\end{enumerate}
\fi
%We say that $f$ is \textit{peripheral $\Zx$-regular} at $\partial_iM$ if it is peripheral $\lambda$-regular for some $\lambda\in \Zx$.
\end{definition}

One essential result in \cite{Xu} is that any profinite isomorphism  between cusped finite-volume hyperbolic 3-manifolds is always peripheral $\Zx$-regular.

\revised{
\begin{theorem}[{\cite[Theorem 7.2]{Xu}}]\label{hypMfd}
Let $M$ and $N$ be orientable cusped finite-volume hyperbolic 3-manifolds, and suppose $f:\widehat{\pi_1M}\ttt\widehat{\pi_1N}$ is an isomorphism. Then, $f$ respects the peripheral structure and $f$ is peripheral $\Zx$-regular at all boundary components of $M$. 
\iffalse
\begin{enumerate}[label=(\arabic*), leftmargin=*]
\item\label{hyp1} $f$ respects the peripheral structure, ie there is a one-to-one correspondence between the boundary components of $M$ and $N$, which we simply denote as $\partial_iM\leftrightarrow\partial_iN$, such that $f(\overline{\pi_1\partial_iM})$ is a conjugate of $\overline{\pi_1\partial_iN}$ in $\widehat{\pi_1N}$.
\item\label{hyp2} $f$ is peripheral $\Zx$-regular at every boundary component of $M$, that is: 
$$C_{g_i}\circ f|_{\overline{\pi_1\partial_iM}}:\tensor \pi_1\partial_iM\xrightarrow{\lambda_i\otimes \psi_i} \tensor\pi_1\partial_iN$$
for some $g_i\in \widehat{\pi_1N}$, $\lambda_i\in \Zx$ and some isomorphism $\psi_i:\pi_1\partial_iM\ttt\pi_1\partial_iN$. %() 
\end{enumerate}
\fi
\end{theorem}
%In fact, we can prove that the coefficients $\lambda\in \Zx$  corresponding to different boundary components are equal up to $\pm$-signs. 
Note that in the setting of \autoref{hypMfd}, given a boundary component $\partial_i M$, the corresponding boundary component $\partial_j N$ appearing in \ref{AAAAA1} of \autoref{DEF: peripheral Zx-regularity} is unique by \cite[Lemma 4.5]{WZ17}, which gives a one-to-one correspondence between the boundary components of $M$ and $N$.} 
The proof of this theorem was actually based on Liu's observation of homological $\Zx$-regularity \cite{Liu23}.

More generally, peripheral $\Zx$-regularity also appears in profinite isomorphisms between mixed 3-manifolds, based on the profinite detection of JSJ-decomposition \cite{WZ19}. 

In the following context, a mixed 3-manifold refers to a compact, orientable, irreducible 3-manifold with empty or incompressible toral boundary, which admits non-trivial JSJ-decomposition and contains at least one hyperbolic JSJ-piece.

\begin{proposition}\label{THM: Mixed Peripheral regular}
Let $M$ be a mixed 3-manifold, and let $\partial_1^{hyp}M,\cdots, \partial_n^{hyp}M$ be the boundary components of $M$ belonging to hyperbolic JSJ-pieces. Suppose $N$ is a compact, orientable 3-manifold, and $f:\widehat{\pi_1M}\ttt \widehat{\pi_1N}$ is an isomorphism.
\begin{enumerate}[label=(\arabic*), leftmargin=*]
\item\label{3.3-1} $N$ is also a mixed 3-manifold; in particular, $N$ is irreducible and boundary-incompressible.
\item\label{3.3-2} There are exactly $n$ boundary components of $N$ which belong to  hyperbolic JSJ-pieces, and we denote these boundary components as $\partial_1^{hyp}N,\cdots,\partial _n^{hyp} N$.
\item\label{3.3-3} Up to a reordering, $f(\overline{\pi_1\partial_i^{hyp}M})$ is a conjugate of $\overline{\pi_1\partial_i^{hyp}N}$ in $\widehat{\pi_1N}$.
\item\label{3.3-4} $f$ is peripheral $\Zx$-regular at each $\partial_i^{hyp}M$ ($1\le i\le n$).\iffalse:
$$C_{g_i}\circ f|_{\overline{\pi_1\partial_iM}}:\tensor \pi_1\partial_iM\xrightarrow{\lambda_i\otimes \psi_i} \tensor\pi_1\partial_iN$$
for some $g_i\in \widehat{\pi_1N}$, $\lambda_i\in \Zx$ and some isomorphism $\psi_i\in \mathrm{Hom}_{\Z}(\pi_1\partial_iM,\pi_1\partial_iN)$. %() 
\fi
\end{enumerate}
\end{proposition}

The first three items are well-known for experts, and the fourth item is a direct corollary of \autoref{hypMfd}. For completeness, we include a proof of \autoref{THM: Mixed Peripheral regular} in \autoref{APP}.
%The proof of \autoref{THM: Mixed Peripheral regular} is displayed in \com{Appendix}. 

\subsection{Dehn filling}
Let us first recall the definition of Dehn filling. Let $T^2$ denote a torus, the collection of \textit{slopes} on $T^2$ is denoted by $$\slope(T^2)=\left.\left\{\alpha \in \pi_1(T^2)\left|\,\begin{gathered} \alpha \text{ is a non-zero primitive element in}\\ \text{the free abelian group }\pi_1(T^2)\cong \Z^2\end{gathered}\right.\right\}\right /\alpha\sim -\alpha.$$
For a non-zero primitive element $\alpha \in \pi_1(T^2)$, we denote by $[\alpha]$ its image in $\slope(T^2)$. Usually, we simply denote an element in $\slope(T^2)$ by a lower-case letter, for example $c\in \slope(T^2)$. 
  
We denote $\slp(T^2)=\slope(T^2)\cup \left\{\varnothing\right\}$, where $\varnothing$ is just a symbol used to denote ``no Dehn filling''. 
\begin{definition}[Dehn filling]
Let $M$ be a compact, orientable 3-manifold with toral boundary components $\partial_1M,\cdots, \partial_n M$, and let  $c_i\in \slp(\partial_iM)$ be a choice of slopes. For each $1\le i\le n$, if $c_i\in \slope(\partial_iM)$, we glue a solid torus $D^2\times S^1$ to $M$ along the boundary torus $\partial_iM$, so that the meridian of the solid torus is attached to $c_i$; and if $c_i=\varnothing$, we do nothing and skip this boundary component. The resulting  compact orientable 3-manifold is called the \textit{Dehn filling} of $M$ along $c_1,\cdots, c_n$, and is denoted by $M_{c_1,\cdots, c_n}$, or $M_{(c_i)}$ for short.
%
%the Dehn filling of $M$ with respect to $c_1,\cdots, c_n$ is a compact orientable 3-manifold  denoted by $M_{c_1,\cdots, c_n}$, which is obtained from $M$ by gluing solid tori $D^2\times S^1$ to the boundary components
%
\end{definition}

\begin{lemma}\label{LEM: Van-Kampen}
Let $M$ be a compact, orientable 3-manifold with toral boundary components $\partial_1M,\cdots, \partial_n M$, and let  $c_i\in \slp(\partial_iM)$ be a choice of slopes. Up to a choice of basepoints, which may result in certain conjugations, we fix the maps $\varphi_i:\pi_1(\partial_iM)\to \pi_1(M)$. For $1\le i \le n$, let 
%$\gamma_i=\varphi_i(c_i)\in \pi_1(M)$ 
$\gamma_i$ be the image of $c_i$ in $\pi_1(M)$ through $\varphi_i$ 
if $c_i\in \slope(\partial_iM)$; and let $\gamma_i$ be the identity element if $c_i=\varnothing$.
\begin{enumerate}[label=(\arabic*), leftmargin=*]
\item\label{2.5-1} $\pi_1M_{(c_i)}\cong \pi_1M/\langle\!\langle \gamma_1,\cdots, \gamma_n\rangle\!\rangle$, where $\langle\!\langle \gamma_1,\cdots, \gamma_n\rangle\!\rangle$ denotes the normal subgroup generated by $\gamma_1,\cdots, \gamma_n$.
\item\label{2.5-2} $\widehat{\pi_1M_{(c_i)}}\cong \widehat{\pi_1M}/\overline{\langle\!\langle \gamma_1,\cdots, \gamma_n\rangle\!\rangle}$, where $\overline{\langle\!\langle \gamma_1,\cdots, \gamma_n\rangle\!\rangle}$ denotes the closed normal subgroup generated by $\gamma_1,\cdots, \gamma_n$.
\end{enumerate}
\end{lemma}
\begin{proof}
In fact, a Dehn filling can be viewed as attaching 2-handles to the simple closed curves $c_i$, and then capping off the boundary spheres by 3-handles. Therefore, \ref{2.5-1} is directly obtained from van Kampen's theorem.

Indeed, we have a short exact sequence:
\begin{equation*}
\begin{tikzcd}
1 \arrow[r] & {H=\langle\!\langle \gamma_1,\cdots, \gamma_n\rangle\!\rangle} \arrow[r, "\iota"] & \pi_1M \arrow[r] & \pi_1M_{(c_i)} \arrow[r] & 1.
\end{tikzcd}
\end{equation*}  
The right-exactness of the completion functor \cite[Proposition 3.2.5]{RZ10} then  implies the following exact sequence:  
\begin{equation*}
\begin{tikzcd}
\widehat{H} \arrow[r, "\widehat{\iota}"] & \widehat{\pi_1M} \arrow[r] & \widehat{\pi_1M_{(c_i)}} \arrow[r] & 1.
\end{tikzcd}
\end{equation*}
Note that the image of $\widehat{\iota}:\widehat{H}\to \widehat{\pi_1M}$ is exactly the closure of $H$ in $\widehat{\pi_1M}$, which is the closed normal subgroup generated by $\gamma_1,\cdots,\gamma_n$, so this implies \ref{2.5-2}.
\end{proof}

As an application of \autoref{hypMfd}, we obtain the key technique \autoref{inthm: Dehn filling} which is rephrased as follows. %This enables us to derive information of Dehn filling from the finite quotient groups of a cusped hyperbolic 3-manfold.
\begin{THMA}\label{PROP: Dehn filling}
Suppose that $M$ and $N$ are orientable cusped finite-volume hyperbolic 3-manifolds and that  $\widehat{\pi_1M}\cong \widehat{\pi_1N}$. Then, there is a one-to-one correspondence between the cusps of $M$ and $N$ denoted by $\partial_iM\leftrightarrow \partial_iN$ ($1\le i \le n$) up to a reordering, and there exist isomorphisms $\psi_i: \pi_1\partial_iM\ttt \pi_1\partial_iN$ between the conjugacy representatives of the peripheral subgroups such that for any slopes $c_i\in \slp(\partial_iM)$ ($1\le i \le n$), 
$$
\widehat{\pi_1M_{(c_i)}}\cong \widehat{\pi_1N_{(\pisi_{i}(c_i))}}
$$
by an isomorphism respecting the peripheral structure.
Here, $\pisi_{i}$ denotes the bijection $\slp(\partial_iM)\to \slp(\partial_iN)$  induced by the group isomorphism $\psi_i: \pi_1\partial_iM\to \pi_1\partial_iN$, and sends $\varnothing$ to $\varnothing$.
\end{THMA}
%\com{added peripheral structure}
\begin{proof}
We  choose conjugacy representatives of the peripheral subgroups $\pi_1\partial_iM\subseteq \pi_1M$, $\pi_1\partial_j N\subseteq \pi_1N$, and  fix a profinite isomorphism $f:\widehat{\pi_1M}\ttt\widehat{\pi_1N}$. 
%According to \autoref{hypMfd}~\ref{hyp1}, $f$ respects the peripheral structure, so up to a reordering of boundary components, $f(\overline{\pi_1\partial_iM})$ is a conjugate of $\overline{\pi_1\partial_i N}$. In addition, \autoref{hypMfd}~\ref{hyp2} further implies that up to composing with a conjugation, there exist $\lambda_i\in \Zx$ and $\psi_i:\pi_1\partial_iM\ttt\pi_1\partial_iN$ such that $conj\circ f|_{\overline{\pi_1\partial_iM}}=\lambda_i\otimes \psi_i$. 
\revised{
According to \autoref{hypMfd}, up to a reordering, $f(\overline{\pi_1\partial_iM})$ is a conjugate of $\overline{\pi_1\partial_i N}$ for each $i$, and there exist coefficients $\lambda_i\in \Zx$ and isomorphisms $\psi_i:\pi_1\partial_iM\ttt\pi_1\partial_iN$ such that for any $\gamma_i\in \pi_1\partial_iM$, $f(\gamma_i)$ is a conjugate of $\psi_i(\gamma_i)^{\lambda_i}$ in $\widehat{\pi_1N}$. 
We show that the isomorphisms $\psi_i$ satisfy the requirement of the theorem. 
}

\revised{
Let $\gamma_i\in \pi_1\partial_iM$ denote a representative element of $c_i$, and $\gamma_i=1$ if $c_i=\varnothing$. %Recall that the $\widehat{\Z}$-module structure on an abelian profinite group is exactly taking powers with exponents in $\widehat{\Z}$, so the peripheral $\Zx$-regularity implies that $f(\gamma_i)$ is a conjugate of $\psi_i(\gamma_i)^{\lambda_i}$ in $\widehat{\pi_1N}$. 
%Recall that the $\widehat{\Z}$-module structure on an abelian profinite group is exactly taking powers with exponents in $\widehat{\Z}$, so the peripheral $\Zx$-regularity implies that $f(\gamma_i)$ is a conjugate of $\psi_i(\gamma_i)^{\lambda_i}$ in $\widehat{\pi_1N}$, no matter $c_i\in \slope(\partial_iM)$ or $c_i=\varnothing$. 
%
Since each $\lambda_i\in \Zx$, by the reasoning in \autoref{subsec: power}, we have  $f(\overline{\langle \! \langle \gamma_1,\cdots, \gamma_n \rangle \! \rangle })=\overline{\langle\!\langle f(\gamma_1),\cdots,f(\gamma_n)\rangle\!\rangle}=\overline{\langle\!\langle \psi_1(\gamma_1)^{\lambda_1},\cdots, \psi_n(\gamma_n)^{\lambda_n}\rangle\!\rangle }=\overline{\langle\!\langle \psi_1(\gamma_1),\cdots, \psi_n(\gamma_n)\rangle\!\rangle }$. }Moreover, according to \autoref{LEM: Van-Kampen}, $\widehat{\pi_1M_{(c_i)}}\cong{\widehat{\pi_1M}}/{\overline{\langle\!\langle \gamma_1,\cdots,\gamma_n\rangle\!\rangle}}$ and $\widehat{\pi_1N_{(\pisi_{i}(c_i))}}\cong \widehat{\pi_1N}/ \overline{\langle\!\langle \psi_1(\gamma_1),\cdots, \psi_n(\gamma_n)\rangle\!\rangle }$. 
Therefore, $f$ descends to an isomorphism
\iffalse
\begin{equation*}
\begin{tikzcd}
f_{\star }: {\widehat{\pi_1M_{(c_i)}}\cong\quo{\widehat{\pi_1M}}{\overline{\langle\!\langle \gamma_1,\cdots,\gamma_n\rangle\!\rangle}}} \arrow[r, "\cong"] & {\quo{\widehat{\pi_1N}}{\overline{\langle\!\langle f(\gamma_1),\cdots,f(\gamma_n)\rangle\!\rangle}}\cong \widehat{\pi_1N_{(\pisi_{i}(c_i))}}}.
\end{tikzcd}
\end{equation*}
\fi
\begin{equation*}
\begin{tikzcd}[column sep=tiny,row sep=tiny]
f_{\star }:\; \widehat{\pi_1M_{(c_i)}}\arrow[r, symbol=\cong] & {\quo{\widehat{\pi_1M}}{\overline{\langle\!\langle \gamma_1,\cdots,\gamma_n\rangle\!\rangle}}} \arrow[rrr, "\cong"] & & & {\quo{\widehat{\pi_1N}}{f(\overline{\langle \! \langle \gamma_1,\cdots, \gamma_n \rangle \! \rangle })}} \arrow[d, symbol=\mathop{=}] \\
 & & & & {\quo{\widehat{\pi_1N}}{\overline{\langle\!\langle \psi_1(\gamma_1),\cdots, \psi_n(\gamma_n)\rangle\!\rangle }}}\arrow[r, symbol=\cong] &{\widehat{\pi_1N_{(\pisi_{i}(c_i))}}}.
\end{tikzcd}
\end{equation*}

Now we prove that $f_{\star}$ respects the peripheral structure. Indeed, $\partial M_{(c_i)}$ and $\partial N_{(\pisi(c_i))}$ consist of the boundary components  $\partial_j M$ and $\partial_j N$, where $c_j=\varnothing$. The corresponding peripheral subgroups are the images of $\iota_{j_{\star}}:\pi_1\partial_jM\hookrightarrow \pi_1M\to \pi_1M/\langle\!\langle \gamma_1,\cdots, \gamma_n\rangle\!\rangle\cong\pi_1M_{(c_i)}$ and $\rho_{j_{\ast}}:\pi_1\partial_jN\hookrightarrow \pi_1N\to \pi_1N/\langle\!\langle \psi_1(\gamma_1),\cdots, \psi_n(\gamma_n)\rangle\!\rangle\cong\pi_1N_{(\pisi(c_i))}$ respectively. Recall that $f$ respects the peripheral structure, ie $f(\overline{\pi_1\partial_jM})=g_j\cdot\overline{\pi_1\partial_jN}\cdot g_j^{-1}$ for some $g_j\in \widehat{\pi_1N}$. %Thus, in the profinite settings, the following diagram commutes for each index $j$ such that $c_j=\varnothing$, where $C_{g_j}$ denotes the conjugation of an element $g_j\in \widehat{\pi_1N}.
Desending to $f_{\star}:\widehat{\pi_1M_{(c_i)}}\cong \widehat{\pi_1N_{(\pisi_{i}(c_i))}}$, we have $f_{\star}(\overline{\iota_{j_\star}(\pi_1\partial_jM)})=\phi(g_j)\cdot \overline{\rho_{j_\star}(\pi_1\partial_jN)}\cdot \phi(g_j)^{-1}$, where $\phi: \widehat{\pi_1N}\to \widehat{\pi_1N}/\overline{\langle\!\langle \psi_1(\gamma_1),\cdots, \psi_n(\gamma_n)\rangle\!\rangle }\cong \widehat{\pi_1N_{(\pisi(c_i))}}$ denotes the quotient homomorphism. Thus, $f_\star$ also respects the peripheral structure.
\end{proof}

\begin{remark}
\hyperref[PROP: Dehn filling]{Theorem A} also applies to non-primitive boundary slopes which yield orbifold Dehn fillings.
\end{remark}

By \autoref{THM: Mixed Peripheral regular}, we can derive an analogous version of \hyperref[PROP: Dehn filling]{Theorem A} for mixed 3-manifolds.

\begin{theorem}\label{PROP: Mixed Dehn filling}
Let $M$ and $N$ be mixed 3-manifolds whose hyperbolic pieces carry some boundary components. Suppose that  $\widehat{\pi_1M}\cong \widehat{\pi_1N}$. Then, there is a one-to-one correspondence between the boundary components of $M$ and $N$ belonging hyperbolic JSJ-pieces, denoted by $\partial_i^{hyp}M\leftrightarrow \partial_i^{hyp}N$ ($1\le i \le n$), and there exist isomorphisms $\psi_i: \pi_1\partial_i^{hyp}M\ttt \pi_1\partial_i^{hyp}N$ between the conjugacy representatives of the peripheral subgroups such that for any slopes $c_i\in \slp(\partial_i^{hyp}M)$ ($1\le i \le n$), 
$$
\widehat{\pi_1M_{(c_i)}}\cong \widehat{\pi_1N_{(\pisi_{i}(c_i))}}.
$$
\end{theorem}
\begin{proof}
The proof of this theorem is exactly the same as the first half of the proof of \hyperref[PROP: Dehn filling]{Theorem A}, with \autoref{hypMfd} replaced by \autoref{THM: Mixed Peripheral regular}.
\end{proof}

\section{Knot and link complements}\label{subsec: knot}

\subsection{Detection of knot and link complements}

As a special application of \autoref{inthm: Dehn filling}, let us consider a typical example of cusped finite-volume hyperbolic 3-manifold, namely the complement of a hyperbolic link in $S^3$. In the following context, all links are tame, and we also include knots  when we refer to  links.

\begin{theorem}\label{PROP: Detect knot complement}
Let $M=\cpl{L}$ be the complement of a hyperbolic link $L$ in $S^3$ (ie the link complement is hyperbolic). Suppose $N$ is a compact orientable 3-manifold such that $\widehat{\pi_1M}\cong \widehat{\pi_1N}$, then $N$ is homeomorphic to the complement of a hyperbolic link $L'\subseteq S^3$, where $L$ and $L'$ have the same number of components.
\end{theorem}
\begin{proof}
Suppose $L$ has $n$ components. 
Since $M$ is finite-volume hyperbolic, according to \autoref{PROP: Detect hyperbolic}, $N$ is also finite-volume hyperbolic with $n$ cusps.  According to \hyperref[PROP: Dehn filling]{Theorem A}, up to reordering the boundary components of $N$, there exist bijections $\pisi_i: \slope(\partial_i M)\to\slope(\partial_i N)$ for $1\le i\le n$, such that  $\widehat{\pi_1 M_{(c_i)}}\cong \widehat{\pi_1N_{(\pisi_i(c_i))}}$ for any $(c_1,\cdots, c_n)\in \slope(\partial_1 M)\times\cdots \times \slope(\partial _ n M)$. 

Let $c_i\in \slope(\partial_i M)$ be the meridian of the $i$-th component of $L$, which is the trivial Dehn filling slope. Then $M_{(c_i)}\cong S^3$, and $\widehat{\pi_1 M_{(c_i)}}\cong \widehat{\pi_1N_{(\pisi_i(c_i))}}$ is the trivial group. Note that $\pi_1N_{(\pisi_i(c_i))}$ is residually finite by \autoref{PROP: RF}, so $\pi_1N_{(\pisi_i(c_i))}$ is also the trivial group. Since $N_{(\pisi_i(c_i))}$ is closed, the validity of the Poincar\'e conjecture implies that $N_{(\pisi_i(c_i))}\cong S^3$. 

Let $L'$ denote the link consisting of the core curves of the Dehn-filled solid tori in $N_{(\pisi_i(c_i))}=S^3$. Then $N\cong \cpl{L'}$ is a link complement, where $L'$ has $n$ components, and $L'$ is hyperbolic since $N\cong \cpl{L'}$ is finite-volume hyperbolic.
\end{proof}

\begin{remark}
In particular, when $L$ has exactly one component, ie $L$ is a knot, we derive that any compact orientable 3-manifold profinitely isomorphic to a hyperbolic knot complement is also a hyperbolic knot complement.
\end{remark}

The same procedure also holds for a link complement which is a mixed manifold, whose boundary components belong to hyperbolic pieces. For brevity of notation, we only describe this statement for {\HSKs}.

\begin{theorem}\label{THM: Detect satellite knot}
Let $K\subseteq S^3$ be a {\HSK}. Denote $M=\cpl{K}$, and suppose $N$ is a compact orientable 3-manifold such that $\widehat{\pi_1M}\cong \widehat{\pi_1N}$. Then, $N$ is also homeomorphic to a knot complement $\cpl{K'}$, where $K'$ is also a {\HSK}.
\end{theorem}
\begin{proof}
By assumption, $M$ is a mixed 3-manifold with exactly one boundary component, which belongs to a hyperbolic piece. Since $\widehat{\pi_1M}\cong \widehat{\pi_1N}$, \autoref{THM: Mixed Peripheral regular}~\ref{3.3-1} implies that $N$ is also a mixed 3-manifold.  According to \autoref{PROP: Mixed Dehn filling}, there exists a boundary component $\partial_0N$ of $N$ belonging to a hyperbolic piece, and an isomorphism $\psi:\pi_1\partial M\to \pi_1\partial_0N$ such that $\widehat{\pi_1M_{c}}\cong \widehat{\pi_1N_{\pisi(c)}}$ for all $c\in \slope(\partial M)$. 

Let $c$ be the meridian of $K$, so that $M_{c}\cong S^3$. Then $\widehat{\pi_1M_{c}}\cong \widehat{\pi_1N_{\pisi(c)}}$ is the trivial group, and the residual finiteness of $\pi_1N_{\pisi(c)}$ (\autoref{PROP: RF}) implies that $\pi_1N_{\pisi(c)}$ is also the trivial group. Therefore, $N_{\pisi(c)}$ is closed, since any compact orientable 3-manifold with non-empty non-spherical boundary has positive first Betti number. In particular, $\partial_0N$ is the unique boundary component of $N$. Again, the validity of the Poincar\'e conjecture   implies that $N_{\pisi(c)}\cong S^3$. 

Let $K'$ be the core curve of the Dehn-filled solid torus in $N_{\pisi(c)}\cong S^3$. Then $N\cong \cpl{K'}$. Since $N$ is a mixed manifold, $K'$ must be a satellite knot. In addition, $\partial_0 N=\partial (N({K'}))$ belongs to a hyperbolic piece, so $K'$ is a {\HSK}. 
\end{proof}

\subsection{Identifying Dehn surgery coefficients}

Indeed, for the above types of  knot complements, we can obtain a more precise version of \autoref{inthm: Dehn filling} by explicitly matching up the Dehn-surgery coefficients, as shown in \autoref{LEM: Knot complement} below. 

\begin{lemma}\label{LEM: null-homologous slope}
\begin{enumerate}[label=(\arabic*), leftmargin=*]
\item\label{5.22-1} Let $M$ be a compact, orientable 3-manifold such that $\partial M$ is a torus. Then, there exists a unique slope $c\in \slope(\partial M)$ so that $c$ represents the null homology class in $H_1(M;\mathbb{Q})$; equivalently, $c$ belongs to the kernel of the homomorphism $\iota_{\ast}:H_1(\partial M;\Z)\to \iffalse{H_1^{free}(M;\Z)=}\fi H_1(M;\Z)_{\mathrm{free}}\iffalse{/\mathrm{torsion}}\fi$ induced by inclusion. This boundary slope is called the ``rational null-homologous boundary slope''  of $M$.
\item\label{5.22-2} Suppose $M$ and $N$ are compact, orientable 3-manifolds such that $\partial M\cong \partial N\cong T^2$. Let $c_0\in \slope(\partial M)$ and $c_0'\in \slope(\partial N)$ be the rational null-homologous boundary slopes of $M$ and $N$ respectively. Suppose $\pisi: \slope(\partial M) \to \slope(\partial N)$ is a bijection such that $\widehat{\pi_1M_c}\cong \widehat{\pi_1N_{\pisi(c)}}$ for any $c\in \slope(\partial M)$. Then $\pisi(c_0)=c_0'$.
\end{enumerate}
\end{lemma}
\begin{proof}
\ref{5.22-1} follows from the so-called ``half lives half dies'' theorem, see \cite[Lemma 8.15]{Lic97}. 

For \ref{5.22-2}, by \autoref{lem: abelianize}, $\widehat{\pi_1M_c}\cong \widehat{\pi_1N_{\pisi(c)}}$ implies that $H_1(M_c;\Z)\cong H_1(N_{\pisi(c)};\Z)$, so $b_1(M_c)=b_1(N_{\pisi(c)})$. Note that $b_1(M_{c_0})=b_1(M)$, and  $b_1(M_c)=b_1(M)-1$ if $c\neq c_0$. Likewise, $b_1(N_{c_0'})=b_1(N)$, and $b_1(N_{c'})=b_1(N)-1$ if $c'\neq c_0'$. Comparing the first Betti numbers, we derive that $b_1(M)=b_1(N)$ and $\pisi(c_0)=c_0'$.
\end{proof}

%Before stating this result, 
Before completing the final part of \autoref{inthm: Knot}, 
let us first review the concept of Dehn surgeries along a knot in $S^3$.

Let $K$ be a knot in $S^3$, and let $X_K=\cpl{K}$ denote its complement. Dehn fillings of $X_K$ can be parametrized by $\Qinf$, once we fix an orientation of $S^3$. To be explicit, let $[m]$ be a slope on $\partial N(K)=\partial X_K$, so that the simple closed curve $[m]$ bounds a disk in $N(K)$, which is called the {\em meridian} of $K$. By Alexander duality, $\partial X_K$ contains a unique null-homologous slope, which we denote as $[l]$, or the preferred {\em longitude} of $K$. We further choose orientations on $m$ and $l$, so that the ordered basis $(m,l)$ yields the boundary orientation of $\partial N(K)$, where $N(K)$ inherits its orientation from $S^3$ as a submanifold. 
 Then, the boundary slopes of $X_K$ can be expressed by $[pm+ql]$, where $(p,q)$ is a pair of coprime integers. The Dehn filling of $X_K$ along the slope $[pm+ql]$ is also called a {\em $\frac{p}{q}$-Dehn surgery} of $K$, where $\frac{p}{q}\in \Qinf$, and is denoted by $K(\frac{p}{q})$. The well-definedness follows from the fact that $[pm+ql]$ and $[-pm-ql]$ represent  the same boundary slope.

\begin{theorem}\label{LEM: Knot complement}
Let $K$ and $K'$ be  hyperbolic knots or {\HSKs} in $S^3$. Suppose that $\widehat{\pi_1X_K}\cong \widehat{\pi_1X_{K'}}$. Then, there exists $\sigma \in \{ 1,-1\}$, such that for any $r\in \Qinf$, $\widehat{\pi_1K({r})}\cong \widehat{\pi_1K'({\sigma r})}$.
\end{theorem}
\begin{proof}
For brevity of notation, we denote $M=X_K$ and $M'=X_{K'}$. According to \autoref{PROP: Detect hyperbolic}, $K$ and $K'$ are either both hyperbolic knots, or both hyperbolic-type satellite knots. 
By \hyperref[PROP: Dehn filling]{Theorem A} for hyperbolic knots and \autoref{PROP: Mixed Dehn filling} for {\HSKs}, there exists an isomorphism $\psi: \pi_1\partial M \ttt\pi_1\partial M'$ such that for any $c\in \slope(\partial M)$, 
\begin{equation}\label{3.9equ0}
\widehat{\pi_1 M_c}\cong \widehat{\pi_1M'_{\pisi(c)}}.
\end{equation}

 Let $(m,l)$ and $(m',l')$ be the oriented meridian-longitude pair of $K$ and $K'$ respectively. Note that the longitude is the rational null-homologous boundary slope of a knot complement. Thus,  
according to \autoref{LEM: null-homologous slope}~\ref{5.22-2}, $\pisi([l])=[l']$, ie $\psi(l)=\pm l'$.

Now we show that $\psi(m)=\pm m'$. Indeed, $M_{[m]}\cong S^3$, and $\widehat{\pi_1M'_{\pisi[m]}}\cong \widehat{\pi_1M_{[m]}}$ is the trivial group. The same argument as in the proof of \autoref{PROP: Detect knot complement}, ie the residual finiteness of the fundamental   group together with the validity of Poincar\'e conjecture, implies that $M'_{\pisi[m]}$ is homeomorphic to $S^3$. Note that $K'$ is a non-trivial knot; according to \cite[Theorem 2]{GL89}, the only Dehn-surgery of $K'$ yielding $S^3$ is the trivial one, ie the $\infty$-surgery. Therefore, $\pisi[m]$ represents the same boundary slope as $[m']$, ie $\psi(m)=\pm m'$.

\iffalse
Next, we show that $\psi(l)=\pm l'$. Suppose $\psi(l)=pm'+ql'$. In fact, the profinite isomorphism $\widehat{\pi_1 M_{[l]}}\cong \widehat{\pi_1M'_{[\psi(l)]}}$ implies an isomorphism on the first homology $H_1(M_{[l]};\Z)\cong H_1(M'_{[\psi(l)]};\Z)$ (cf. \cite[Lemma 2.8]{Xu}). Note that $l$ is null-homologous in $M$, so $H_1(M_{[l]};\Z)\cong H_1(M;\Z)\cong \Z$ by Alexander duality. On the other hand, the homology class of $m'$ is a generator of $H_1(M';\Z)\cong \Z$, so $H_1(M'_{[pm'+ql']};\Z)\cong \Z/p\Z$. Therefore,  $p=0$, and $\psi(l)=\pm l'$.
\fi

As a consequence, there exist $\sigma,\delta\in \{1,-1\}$, such that $\psi(am+bl)=\delta(\sigma am'+bl')$ for any $a,b\in \Z$.  

For any coprime pair of integers $(p,q)$, $M_{[pm+ql]}=K(\frac{p}{q})$, and  
$M'_{\pisi[pm+ql]}=M'_{[\delta(\sigma pm'+ql')]}=M'_{[\sigma pm'+ql']}=K'(\sigma\frac{p}{q})$.  
Thus, $\widehat{\pi_1K(\frac{p}{q})}\iffalse{=\widehat{\pi_1M}_{[pm+ql]}\cong {\widehat{\pi_1}M'_{[\delta(\sigma pm'+ql')]}}=}\fi\cong\widehat{\pi_1K'(\sigma\frac{p}{q})}$ by (\ref{3.9equ0}), which finishes the proof of this proposition.
\end{proof}

\subsection{Criteria for profinite rigidity of a knot complement}\label{subsec: BC}

\iffalse
We can now prove a generalized version of \autoref{THMB}. Let us first generalize \autoref{DEF: characterizing slope} to a characterizing set of slopes.
\begin{definition}
Let $K\subseteq S^3$ be a knot. A non-empty subset $\Lambda\subseteq \Qinf$, finite or infinite, is called a {\em characterising set of slopes} of $K$ if for any knot $J\subseteq S^3$, $K(\alpha)\cong J(\alpha)$ for all $\alpha\in \Lambda$ if and only if $K$ is isotopic to $J$.
\end{definition}

\begin{theorem}
Let $K\subseteq S^3$ be either a hyperbolic knot or a {\HSK}. Suppose
\begin{enumerate}[label=(\arabic*), leftmargin=*]
\item\label{B1} there is a characterising set of slopes $\Lambda$ of $K$, and
\item\label{B2} for each $\alpha\in \Lambda$, the Dehn surgery $K(\alpha)$ is profinitely rigid among all closed, orientable 3-manifolds.
\end{enumerate}
Then the knot complement $X_K=\cpl{K}$ is profinitely rigid among all compact, orientable 3-manifolds.
\end{theorem}
\fi

In \autoref{DEF: characterizing slope} we have introduced the {\SCSs} on a knot $K\subseteq S^3$. Let us first clarify its difference with the usual definition of characterising slopes.
\begin{definition}\label{DEF: SCSCS}
Let $K$ be a knot in $S^3$, and let $\alpha\in \mathbb{Q}$ be a slope on $\partial N(K)$. 
\begin{enumerate}[label=(\arabic*), leftmargin=*]
\item\label{SCSCS1} $\alpha$ is an {\em\SCS} of $K$ if for any knot $J\subseteq S^3$, $J(\alpha)\cong K(\alpha)$ implies that $J$ is isotopic to $K$ or its mirror image.
\item\label{SCSCS2} $\alpha$ is a {\em characterising slope} of $K$ if for any knot $J\subseteq S^3$, $J(\alpha)\cong K(\alpha)$ by an orientation-preserving homeomorphism implies that $J$ is isotopic to $K$.
\end{enumerate}
\end{definition}

Generally, for $\alpha\neq 0$, neither one of {\SC} and characterising directly implies the other one. However, for the $0$ slope, we have the following relation.
\begin{lemma}\label{LEM: Characterising}
Let $K$ be a knot in $S^3$. If $0$ is a characterising slope of $K$, then $0$ is also an {\SCS} of $K$.
\end{lemma}
\begin{proof}
Suppose $J$ is a knot in $S^3$, and there is a homeomorphism $F: J(0)\to K(0)$. 
There are two possibilities. First, if $F$ is orientation-preserving, then \autoref{DEF: SCSCS}~\ref{SCSCS2} implies that  $J$ is isotopic to $K$. 
Second, if $F$ is orientation-reversing, let $\overline{J}$ denote the mirror image of $J$. Then $\overline{J}(0)\cong J(0)$ by an orientation-reversing homeomorphism. Thus,  there exists an orientation-preserving homeomorphism $\overline{J}(0)\cong K(0)$. Then, \autoref{DEF: SCSCS}~\ref{SCSCS2} implies that $\overline{J}$ is isotopic to $K$, ie $J$ is the mirror image of $K$.
\end{proof}

%We now prove a generalized version of \autoref{THMB}. 
%In fact, we can generalize 
\autoref{DEF: SCSCS}~\ref{SCSCS1} can also be generalized to a  set of slopes.

\begin{definition}
Let $K\subseteq S^3$ be a knot. A non-empty subset $\Lambda\subseteq \mathbb{Q}$, finite or infinite, is called an  {\em{\SetSCS}} of $K$ if for any knot $J\subseteq S^3$, $J(\alpha)\cong K(\alpha)$ for all $\alpha\in \Lambda$ implies that $J$ is isotopic to $K$ or its mirror image.
\end{definition}

\iffalse
We are now ready to prove \autoref{THMB}.
\begin{THMB}
Let $K\subseteq S^3$ be either a hyperbolic knot, or a {\HSK}. Suppose
\begin{enumerate}[label=(\arabic*), leftmargin=*]
\item\label{B1} there is an {\SCS} $\alpha$ of $K$, and
\item\label{B2} the Dehn surgery $K(\alpha)$ is profinitely rigid among all closed, orientable 3-manifolds.
\end{enumerate}
Then the knot complement $\cpl{K}$ is profinitely rigid among all compact, orientable 3-manifolds.
\end{THMB}
\fi

\begin{theorem}[Generalized version of \autoref{THMB}]\label{THM: B+}
Let $K\subseteq S^3$ be either a hyperbolic knot or a {\HSK}. Suppose
\begin{enumerate}[label=(\arabic*), leftmargin=*]
\item\label{B1} there is an {\SetSCS} $\Lambda$ for $K$, and
\item\label{B2} for each $\alpha\in \Lambda$, the Dehn surgery $K(\alpha)$ is profinitely rigid among all closed, orientable 3-manifolds.
\end{enumerate}
Then, the knot complement $X_K=\cpl{K}$ is profinitely rigid among all compact, orientable 3-manifolds.
\end{theorem}

\iffalse
\begin{proof}
Let $N$ be a compact, orientable 3-manifold such that $\widehat{\pi_1X_K}\cong \widehat{\pi_1N}$. According to \autoref{PROP: Detect knot complement} for the hyperbolic case and \autoref{THM: Detect satellite knot} for the satellite case, $N\cong X_J\iffalse\cpl{J}\fi$ is also the  complement of a hyperbolic knot or a {\HSK} $J$. In addition, by \autoref{LEM: Knot complement}, there exists $\sigma \in \{\pm1\}$ such that $\widehat{\pi_1K(r)}\cong \widehat{\pi_1J(\sigma r)}$ for all $r\in \Qinf$. By possibly replacing $J$ with its mirror image, we may assume $\sigma=1$. 

Therefore, $\widehat{\pi_1K(\alpha)}\cong \widehat{\pi_1J(\alpha)}$, and  
condition (\ref{B2}) implies that $K(\alpha)\cong J(\alpha)$. Then,  condition (\ref{B1}) implies that $J$ is either equivalent or mirror image to $K$. \iffalse, according to \autoref{DEF: characterizing slope}.\fi Thus, $N\cong X_J\cong X_K$.
\end{proof}
\fi

\begin{proof}
Let $N$ be a compact, orientable 3-manifold such that $\widehat{\pi_1X_K}\cong \widehat{\pi_1N}$. According to \autoref{PROP: Detect knot complement} for the hyperbolic case and \autoref{THM: Detect satellite knot} for the satellite case, $N\cong X_J\iffalse\cpl{J}\fi$ is also the  complement of a hyperbolic knot or a {\HSK} $J$. In addition, by \autoref{LEM: Knot complement}, there exists $\sigma \in \{\pm1\}$ such that $\widehat{\pi_1K(r)}\cong \widehat{\pi_1J(\sigma r)}$ for all $r\in \Qinf$. By possibly replacing $J$ with its mirror image, we may assume $\sigma=1$, so $\widehat{\pi_1K(\alpha)}\cong \widehat{\pi_1J(\alpha)}$ for all $\alpha\in \Lambda$. Then,  
condition  \ref{B2}  implies that $K(\alpha)\cong J(\alpha)$ for all $\alpha\in \Lambda$, and   condition  \ref{B1}  thus implies that $J$ is either equivalent or mirror image to $K$. \iffalse, according to \autoref{DEF: characterizing slope}.\fi Hence, $N\cong X_J\cong X_K$.
\end{proof}

Note that any {\SCS} of $K$ is an {\SetSCS} with one element. Therefore, \autoref{THM: B+} does imply \autoref{THMB}. 

Now we can apply  \autoref{THMB} as a technique for the proof of \autoref{CORC}.

\begin{CORC}\label{pfCORC}
Let $K\subseteq S^3$ be one of the following knots:
\begin{enumerate}[label=(\arabic*), leftmargin=*]
\item\label{CC1} the $5_2$ knot (\autoref{Fig: 52});
\item\label{CC2} the $15n_{43522}$ knot  (\autoref{Fig: 15n43522});
\item\label{CC3} the Pretzel knots $P(-3, 3, 2n + 1)$, where $n\in \Z$ (\autoref{Fig: Pretzel});
\item\label{CC4} the satellite knots $\mathcal{W}^+(T_{2,3}, 2)$ and $\mathcal{W}^-(T_{2,3}, 2)$ (\autoref{Fig: Wh}).
\end{enumerate}
Then $X_K=\cpl{K}$ is profinitely rigid among all compact, orientable 3-manifolds.
\end{CORC}

\begin{proof}
 The knots in \ref{CC1}, \ref{CC2} and \ref{CC3} are hyperbolic knots, while the knots in \ref{CC4} are {\HSKs}. In fact, the $5_2$ knot is the $k3_2$ hyperbolic knot  as tabulated in the census \cite{CDW}, and the $15n_{43522}$ knot is the $k8_{128}$ hyperbolic knot  in the census \cite{CKM}. The hyperbolicity of  $P(-3, 3, 2n + 1)$ follows from \cite{LM99}. The  satellite knots $\mathcal{W}^+(T_{2,3}, 2)$ and $\mathcal{W}^-(T_{2,3}, 2)$  are Whitehead-doubles of non-trivial knots, and so they are {\HSKs}.  

Therefore, we can apply \autoref{THMB} to prove the profinite rigidity of these knot complements. Indeed, \cite[Theorem 1.1]{BS24} shows that $0$ is a characterising slope for $5_2$, and \cite[Theorem 1.1]{BS22} shows that $0$ is a characterising slope for $15n_{43522}$, $P(-3,3,2n+1)$ or $\mathcal{W}^{\pm}(T_{2,3},2)$. Thus, by \autoref{LEM: Characterising}, $0$ is an {\SCS} for any listed knot $K$. Hence, according to \autoref{THMB}, it suffices to show that $K(0)$ is profinitely rigid among all closed, orientable 3-manifolds.

Indeed, when $K$ is one of the hyperbolic knots listed in \ref{CC1}, \ref{CC2} and \ref{CC3}, Cheetham-West showed in the proof of \cite[Theorem 1.1]{Cw23} that $K(0)$ is a graph manifold with JSJ-graph a single vertex and one cycle edge. Therefore, $K(0)$ is profinitely rigid in closed, orientable 3-manifolds according to the profinite classification of closed graph manifolds by Wilkes  \cite[Theorem A]{Wil18}.

We claim that when $K= \mathcal{W}^{\pm}(T_{2,3},2)$, $K(0)$ is a graph manifold with JSJ-graph shown in \autoref{Fig: JSJK(0)}. Let $J$ denote the companion knot of $K$, which is a right-handed trefoil. The complement $X_J$ is a Seifert fibered space $(D^2;(2,1),(3,1))$. Let $M_1$ denote the complement of the Whitehead link, then $X_K=M_1\cup_{T^2}X_J$. Note that each boundary component of $M_1$ already contains a null-homologous slope, corresponding to the preferred longitude of the Whitehead link. Thus, the $0$-surgery on $K$ is in fact the Dehn filling corresponding to the longitude of the pattern Whitehead link. Let $\mathcal{W}(0)$ denote $0$-surgery on one component of the Whitehead link. Then, $K(0)=\mathcal{W}(0)\cup_{T^2} X_J$. In fact, according to \cite{MP02}, $\mathcal{W}(0)$ is a graph manifold obtained from a single Seifert fibered piece $P\times S^1$ by gluing together two of its boundary components, where $P=S^2\setminus \{3\text{ disks}\}$ denotes a pair of pants. In addition, on $\partial \mathcal{W}(0)$, the slope of a regular fiber in the Seifert fibration is exactly the meridian of the unfilled Whitehead link component. Since $K=\mathcal{W}^{\pm}(T_{2,3},2)$ is a Whitehead double of twisting number $2$, the meridian of the pattern Whitehead link corresponds to the slope $2$ on $\partial N(J)$. On the other hand, $J=T_{2,3}$, so  the slope of a regular fiber of the Seifert fibered space $X_J$ on $\partial N(J)$ is $6$ in the $\Qinf$-parametrization. Therefore, the gluing between $\mathcal{W}(0)$ and $X_J$ along $T^2=\partial N(J)$ is not fiber-parallel. Thus, $\partial N(J)$ is a JSJ-torus of $X_K$ according to \cite[Proposition 1.9]{AFW15}, and the JSJ-graph of $X_K$ is obtained by connecting the JSJ-graphs of $\mathcal{W}(0)$ and $X_J$ along a single edge corresponding to $\partial N(J)$, which is exactly shown as \autoref{Fig: JSJK(0)}.

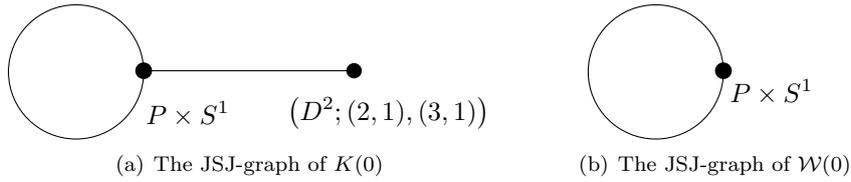
\begin{figure}[h]
\subfigure[The JSJ-graph of $K(0)$]{
\begin{tikzpicture}[x=0.75pt,y=0.75pt,yscale=-1,xscale=1]
%uncomment if require: \path (0,487); %set diagram left start at 0, and has height of 487

%Shape: Arc [id:dp16452427805756953] 
\draw  [draw opacity=0] (93.5,56.75) .. controls (93.5,75.39) and (78.39,90.5) .. (59.75,90.5) .. controls (41.11,90.5) and (26,75.39) .. (26,56.75) .. controls (26,38.11) and (41.11,23) .. (59.75,23) .. controls (78.19,23) and (93.18,37.79) .. (93.49,56.16) -- (59.75,56.75) -- cycle ; \draw    (93.5,56.75) .. controls (93.5,75.39) and (78.39,90.5) .. (59.75,90.5) .. controls (41.11,90.5) and (26,75.39) .. (26,56.75) .. controls (26,38.11) and (41.11,23) .. (59.75,23) .. controls (78.19,23) and (93.18,37.79) .. (93.49,56.16) ; \draw [shift={(93.49,56.16)}, rotate = 78.5] [color={rgb, 255:red, 0; green, 0; blue, 0 }  ][fill={rgb, 255:red, 0; green, 0; blue, 0 }  ][line width=0.75]      (0, 0) circle [x radius= 3.69, y radius= 3.69]   ; 
%Straight Lines [id:da6260313130070299] 
\draw    (93.49,56.16) -- (198.5,56.16) ;
\draw [shift={(198.5,56.16)}, rotate = 0] [color={rgb, 255:red, 0; green, 0; blue, 0 }  ][fill={rgb, 255:red, 0; green, 0; blue, 0 }  ][line width=0.75]      (0, 0) circle [x radius= 3.35, y radius= 3.35]   ;

% Text Node
\draw (93.49,70.16) node [anchor=north west][inner sep=0.75pt]   [align=left] {$\displaystyle P\times S^{1}$};
% Text Node
\draw (164,68) node [anchor=north west][inner sep=0.75pt]   [align=left] {$\displaystyle \left( D^{2} ;( 2,1) ,( 3,1)\right)$};

\end{tikzpicture}
\label{Fig: JSJK(0)}
}
\hspace{8mm}
\subfigure[The JSJ-graph of $\mathcal{W}(0)$]{
\begin{tikzpicture}[x=0.75pt,y=0.75pt,yscale=-1,xscale=1]
%uncomment if require: \path (0,487); %set diagram left start at 0, and has height of 487

%Shape: Arc [id:dp16452427805756953] 
\draw  [draw opacity=0] (93.5,56.75) .. controls (93.5,75.39) and (78.39,90.5) .. (59.75,90.5) .. controls (41.11,90.5) and (26,75.39) .. (26,56.75) .. controls (26,38.11) and (41.11,23) .. (59.75,23) .. controls (78.19,23) and (93.18,37.79) .. (93.49,56.16) -- (59.75,56.75) -- cycle ; \draw    (93.5,56.75) .. controls (93.5,75.39) and (78.39,90.5) .. (59.75,90.5) .. controls (41.11,90.5) and (26,75.39) .. (26,56.75) .. controls (26,38.11) and (41.11,23) .. (59.75,23) .. controls (78.19,23) and (93.18,37.79) .. (93.49,56.16) ; \draw [shift={(93.49,56.16)}, rotate = 78.5] [color={rgb, 255:red, 0; green, 0; blue, 0 }  ][fill={rgb, 255:red, 0; green, 0; blue, 0 }  ][line width=0.75]      (0, 0) circle [x radius= 3.69, y radius= 3.69]   ; 

% Text Node
\draw (95.49,59.16) node [anchor=north west][inner sep=0.75pt]   [align=left] {$\displaystyle P\times S^{1}$\text{ }\text{ }\text{ }};

\end{tikzpicture}
}
\caption{The JSJ-decomposition in the proof of \hyperref[pfCORC]{Theorem E}}
\end{figure}

In particular, $K(0)$ is a graph manifold with a non-bipartite JSJ-graph, and it follows from \cite[Theorem A]{Wil18} that $K(0)$ is profinitely rigid among closed orientable 3-manifolds. The proof is complete by adopting \autoref{THMB}.
\end{proof}

\section{Profinite detection of exceptional Dehn fillings}\label{SEC: Exceptional}
In this section, we focus on general cusped hyperbolic 3-manifolds. 
Let $M$ be a finite-volume hyperbolic 3-manifold with cusps. 
By Thurston's theorem \cite[Theorem 5.8.2]{Thurston}, most Dehn fillings of $M$ are also hyperbolic. 
%It is known that most Dehn fillings of $M$ are also hyperbolic \cite[Theorem 5.8.2]{Thurston}. 
A Dehn filling $M_{(c_i)}$ is called {\em exceptional} if $M_{(c_i)}$ is not a finite-volume hyperbolic 3-manifold. 

In practice, we usually consider Dehn fillings on one particular boundary component $\partial_{i_0}M$ of $M$, 
and we denote 
$$
\mathcal{E}(M,\partial_{i_0}M)=\left\{ c\in \slope(\partial_{i_0}M)\mid M_c \text{ is an exceptional Dehn filling}\right\}
$$
as the collection of all exceptional slopes on $\partial_{i_0}M$. 
Indeed, 
%, and 
 \cite[Theorem 5.8.2]{Thurston} implies that only finitely many slopes on $\partial_{i_0}M$ yield exceptional Dehn fillings, ie $\# \mathcal{E}(M,\partial_{i_0}M)<+\infty$.

Note that any finite-volume hyperbolic 3-manifold is aspherical and atoroidal, and  it does not contain an embedded Klein bottle. Therefore, we consider three particular subsets of $\mathcal{E}(M,\partial_{i_0}M)$:
\begin{equation*}
\begin{gathered}
\mathcal{E}_s(M,\partial_{i_0}M)=\left\{c\in \slope(\partial_{i_0}M)\mid M_c \text{ is not aspherical} \right\};\\
\mathcal{E}_t(M,\partial_{i_0}M)=\left\{ c\in \slope(\partial_{i_0}M)\mid M_c \text{ is toroidal}\right\}\quad \text{(see \autoref{DEF: toroidal})};\\
\mathcal{E}_k(M,\partial_{i_0}M)=\left\{ c\in \slope(\partial_{i_0}M)\left |\,\begin{gathered} M_c\text{ is aspherical, and }M_c\\ \text{ contains an embedded Klein bottle}\end{gathered}\right.\right\}.
\end{gathered}
\end{equation*}

Let us first prove some general results detecting these properties. These results might be familiar to experts, but we still include detailed  proofs for completeness.
%The remaining part of this section is devoted to the proof of the following proposition.
%\subsection{Profinite detection of exceptional Dehn fillings}

\subsection{Detecting prime decomposition and asphericity}%, tori and Klein bottles}

\begin{lemma}\label{LEM: Prime}
Let $M$ and $N$ be compact, orientable 3-manifolds with empty or toral boundary. The prime decomposition of $M$ and $N$ yields the following connected sums
\begin{equation}\label{Prime decomposition}
\begin{gathered}
M=(D^2\times S^1)^{\# r_1}\# (S^2\times S^1)^{\# r_2} \# M_1\#\cdots \#M_s,\\
N=(D^2\times S^1)^{\# r_1'}\# (S^2\times S^1)^{\# r_2'} \# N_1\#\cdots \#N_{s'},
\end{gathered}
\end{equation}
where the $D^2\times S^1$ factors correspond to the compressible boundary components, and the $M_i$ or $N_j$ factors are irreducible and boundary-incompressible. 

Suppose $f:\widehat{\pi_1M}\ttt\widehat{\pi_1N}$ is an isomorphism. Then $r_1+r_2=r_1'+r_2'$, $s=s'$, and up to a reordering, $f$ induces an isomorphism $f_i:\widehat{\pi_1M_i}\ttt\widehat{\pi_1N_i}$. In addition, if $f$ respects the peripheral structure, then each $f_i$ also respects the peripheral structure.
\end{lemma}
\begin{proof}
The proof mainly follows from \cite[Theorem 6.22]{Wil19}, which further requires that $M$ and $N$ are boundary-incompressible. However, we can make some adjustments to transform the boundary-compressible case to the boundary-incompressible case.

We Dehn fill the compressible boundary components of $M$ along their null-homotopic slopes, and obtain a boundary-incompressible 3-manifold 
$$M^+=(S^2\times S^1)^{\# (r_1+r_2)}\# M_1\#\cdots\# M_s.$$ 

In particular, the inclusion $M\hookrightarrow M^+$ induces an isomorphism between the fundamental groups $\pi_1M\ttt\pi_1M^+$.

Similarly, let $$N^+=(S^2\times S^1)^{\# (r_1'+r_2')}\#N_1\#\cdots \#N_{s'}$$ be the Dehn-filled manifold so that $\pi_1N\ttt\pi N^+$. Then, $f$ induces an isomorphism $f^+:\widehat{\pi_1M^+}\cong \widehat{\pi_1M}\ttt\widehat{\pi_1N}\cong \widehat{\pi_1N^+}$.

\cite[Theorem 6.22]{Wil19} shows that the profinite completion of fundamental group detects the prime decomposition of a compact, orientable, boundary-incompressible 3-manifold. Thus,   $r_1+r_2=r_1'+r_2'$, $s=s'$, and up to a reordering, $\widehat{\pi_1M_i}\cong \widehat{\pi_1N_i}$ for each $1\le i \le s$. 

To be explicit, up to a choice of conjugacy representatives, we have $\pi_1M^+=\pi_1M_1\ast \cdots \ast \pi_1M_s\ast F_{r_1+r_2}$ and $\widehat{\pi_1M^+}=\widehat{\pi_1M_1}\amalg \cdots \amalg \widehat{\pi_1M_s}\amalg \widehat{F_{r_1+r_2}}$, and likewise for $\widehat{\pi_1N^+}$. This can be checked by examining the universal property of  free profinite products.  %$\pi_1M_i$ can be viewed as a free factor of ${\pi_1M^+}$. \cite[Proposition 9.1.8]{RZ10} implies that $\pi_1M^+$ induces the full profinite topology on $\pi_1M_i$, so the homomorphism $\widehat{\pi_1M_i}\to \widehat{\pi_1M^+}$ is injective according to \cite[Lemma 3.2.6]{RZ10}. Thus, we can view $\widehat{\pi_1M_i}$ and $\widehat{\pi_1N_i}$ as closed subgroups of $\widehat{\pi_1M^+}$ and $\widehat{\pi_1N^+}$ respectively. 
With this decomposition, \cite[Theorem 6.22]{Wil19} actually shows that $f^+$ sends each $\widehat{\pi_1M_i}$ to a conjugate of $\widehat{\pi_1N_i}$ in $\widehat{\pi_1N^+}$, and we denote $f_i:\widehat{\pi_1M_i}\ttt\widehat{\pi_1N_i}$ as $f^+$ composing with a conjugation and then restricted on $\widehat{\pi_1M_i}$.

\revised{
When $f$ respects the peripheral structure, we first claim that $f^+$ respects the peripheral structure. Indeed, the closure of a peripheral subgroup corresponding to a compressible boundary component is isomorphic to $\widehat{\Z}$, while the closure of a peripheral subgroup corresponding to an incompressible boundary component is isomorphic to $\widehat{\Z}^2$. By \autoref{lem: abelianize}, $\widehat{\Z}$ and $\widehat{\Z}^2$ are not isomorphic, so the bijection between the boundary components of $M$ and $N$ carried by $f$ distinguishes between the compressible ones and the incompressible ones. Note that the peripheral subgroups of $M^+$ and $N^+$ are exactly the images in $\pi_1M^+$ and $\pi_1N^+$ of the peripheral subgroups corresponding to the incompressible boundary components in $\pi_1M$ and $\pi_1N$. Then, by the same reasoning as the last paragraph in the proof of  \autoref{inthm: Dehn filling}, we derive that $f^+$ respects the peripheral structures. 

Hence, it suffices to show that when $f^+$ respects the peripheral structure, each $f_i$ also respects the peripheral structure. Let $g_i\in \pi_1N^+$ such that $f|_{\widehat{\pi_1M_i}}= C_{g_i}\circ f_i$, where $C_{g_i}(x)=g_ixg_i^{-1}$ is the conjugation by $g_i$. Then, $f(\widehat{\pi_1M_i})=C_{g_i}(\widehat{\pi_1N_i})$. Suppose $\partial_kM_i$ is a boundary component of $M_i$, and $f^+(\overline{\pi_1\partial_kM_i})=C_{h}(\overline{\pi_1\partial_lN_j})$ for some $h\in \widehat{\pi_1N^+}$, where $\pi_1\partial_lN_j$ is a peripheral subgroup in $\pi_1N_j$. Then, $C_{g_i}(\widehat{\pi_1N_i})$ intersects nontrivially with $C_h(\widehat{\pi_1N_j})$. However, $\widehat{\pi_1N^+}=\widehat{\pi_1N_1}\amalg\cdots \amalg \widehat{\pi_1N_{s'}}\amalg \widehat{F_{r_1'+r_2'}}$, so the family of subgroups $\{ \widehat{\pi_1N_1},\cdots  , \widehat{\pi_1N_{s'}}\}$ is malnormal in $\widehat{\pi_1N^+}$. Thus, $j=i$ and $\alpha_{k,i}:=g_i^{-1}h \in \widehat{\pi_1N_i}$. Consequently, $f_i(\overline{\pi_1\partial _k M_i})= C_{\alpha_{k,i}}(\overline{\pi_1\partial_l N_i})$. This matches up a one-to-one correspondence between the boundary components of $M_i$ and $N_i$, which proves that each $f_i$ respects the peripheral structure.}
\end{proof}

\begin{corollary}\label{Cor: Spherical}
Let $M$ and $N$ be compact, orientable 3-manifolds with empty or toral boundary. Suppose $\widehat{\pi_1M}\cong \widehat{\pi_1N}$. %
%, and that $M$ and $N$ have the same number of boundary components. 
Then $M$ is aspherical if and only if $N$ is aspherical, except for the case that one of $M$ and $N$ is $D^2\times S^1$ while the other one is $S^2\times S^1$.
\end{corollary}
\begin{proof}
It is easy to verify that $M$ is non-aspherical if and only if one of the following conditions holds:
\begin{enumerate}[label=(\roman*), leftmargin=*]
\item $M$ has non-trivial prime decomposition;
\item $M\cong S^2\times S^1$;
\item $M$ has finite fundamental group.
\end{enumerate}

In the first two cases, $M$ contains an essential $S^2$; while in the third case, $M$ is covered by $S^3$. Thus, it suffices to examine these three conditions.

When $M$  has non-trivial prime decomposition, according to \autoref{LEM: Prime}, $\widehat{\pi_1M}\cong \widehat{\pi_1N}$ implies that $N$ also has non-trivial prime decomposition, so $N$ is also non-aspherical.

When $M\cong S^2\times S^1$, $\widehat{\pi_1N}\cong \widehat{\pi_1M}\cong\widehat{\Z}$. The residual finiteness of $\pi_1N$ (cf. \autoref{PROP: RF}) implies that $\pi_1N\cong \Z$, so either $N\cong S^2\times S^1$ is non-aspherical, or $N\cong D^2\times S^1$ is the exception.

When $M$ has finite fundamental group, $\widehat{\pi_1N}\cong \widehat{\pi_1M}$ is finite, and the residual finiteness of $\pi_1N$ (\autoref{PROP: RF}) implies that $\pi_1N$ is also finite, so $N$ is also non-aspherical.
\end{proof}

\iffalse
\begin{remark}
In fact, the assumption that $M$ and $N$ have the same number of boundary components is only used to avoid the case when one of $M$ and $N$ is $D^2\times S^1$ while the other one is $S^2\times S^1$.
\end{remark}
\fi

\subsection{Detecting toroidality}

\revised{
In this paper, the term toroidal is taken to mean geometrically toroidal.
\begin{definition}\label{DEF: toroidal}
A compact, orientable 3-manifold $M$ is {\em toroidal} if $M$ contains an embedded incompressible torus which is not parallel to any boundary component. 
\end{definition}
}

\begin{lemma}\label{LEM: Toroidal}
Let $M$ and $N$ be compact, orientable 3-manifolds with empty or toral boundary. Suppose that $f:\widehat{\pi_1M}\ttt \widehat{\pi_1N}$ is an isomorphism respecting the peripheral structure. Then, $M$ is toroidal if and only if $N$ is toroidal.
\end{lemma}
\begin{proof}
\iffalse
The prime decomposition of $M$ yields a connected sum $$M=(D^2\times S^1)^{\# r_1}\# (S^2\times S^1)^{\# r_2} \# M_1\#\cdots \#M_s,$$ where each $M_i$ is irreducible, and $\partial M_i$ is either empty or consists of incompressible tori. In fact, the $D^2\times S^1$ factors are correspondent to the compressible boundary components of $M$. Let $$M^+=(S^2\times S^1)^{\# (r_1+r_2)}\# M_1\#\cdots\# M_s$$ be obtained from Dehn filling the compressible boundary components. Then $M^+$ is boundary-incompressible, and $\pi_1M^+\cong \pi_1M$. 

Similarly, $N=(D^2\times S^1)^{\# r_1'}\# (S^2\times S^1)^{\# r_2'} \# N_1\#\cdots \#N_{s'}$, where each $N_i$ is irreducible, with empty or incompressible toral boundary; and $N^+=(S^2\times S^1)^{\# (r_1'+r_2')}\#N_1\#\cdots \#N_{s'}$, so that $N^+$ is boundary incompressible and $\pi_1N^+\cong \pi_1N$.

Then, $\widehat{\pi_1M^+}\cong \widehat{\pi_1M}\cong \widehat{\pi_1N}\cong \widehat{\pi_1N^+}$. \cite[Theorem 6.22]{Wil19} shows that the profinite completion of fundamental group detects the prime decomposition of a compact, orientable, boundary-incompressible 3-manifold. Explicitly, $r_1+r_2=r_1'+r_2'$, $s=s'$, and up to a reordering, $\widehat{\pi_1M_i}\cong \widehat{\pi_1N_i}$.
\fi
We denote as in (\ref{Prime decomposition}) the prime decomposition of $M$ and $N$. Then, \autoref{LEM: Prime} implies that $s=s'$, and $f$ induces isomorphisms $f_i:\widehat{\pi_1M_i}\ttt\widehat{\pi_1N_i}$ respecting the peripheral structure.

Note that $M$ is toroidal if and only if one of $M_1,\cdots, M_s$ is toroidal, and $N$ is toroidal if and only if one of $N_1,\cdots, N_{s'}$ is toroidal. Thus, it suffices to show that for each irreducible and $\partial$-irreducible prime factor,   $M_i$ is toroidal if and only if $N_i$ is toroidal. 

Based on  the hyperbolization theorem (see \cite[Theorem 1.7.5]{AFW15}) and standard results on Seifert fibered spaces, $M_i$ is not toroidal if and only if $M_i$ falls into one of the following 3 possibilities.
\begin{enumerate}[label=(\roman*), leftmargin=*]
%\item\label{con1} $\pi_1M_i$ is infinite;
\item\label{con2} $M_i$ is  a finite-volume hyperbolic manifold;
\item\label{concha} $M_i$ is a closed Seifert fibered space that admits a Seifert fibration over $S^2$ with at most 3 exceptional fibers, but does not admit $\mathbb{E}^3$ geometry; 
%not a closed small Seifert fibered space, ie $M_i$ does not admit a Seifert fibration over \iffalse base orbifold is\fi a sphere with at most 3 singular points; 
%\item\label{con3} $M_i$ is not homeomorphic to the thickened torus $S^1\times S^1\times I$, or equivalently, $\pi_1M_i\not\cong \Z^2$;
%\item\label{con4} %$M_i$ is not homeomorphic to the orientable $I$-bundle over Klein bottle $S^1\ttimes S^1\ttimes I$, or equivalently, $\pi_1M_i\not \cong \Z\rtimes \Z$.
%$M_i$ is not homeomorphic to the thickened torus $S^1\times S^1\times I$, or the orientable $I$-bundle over Klein bottle $S^1\ttimes S^1\ttimes I$; 
\item\label{c?}
$M_i$ is   homeomorphic to one of the following listed Seifert fibered spaces with boundary: 
\begin{itemize}
\item $S^1\times S^1\times I=(S^1\times I;)$;
\item the orientable $I$-bundle over Klein bottle $S^1\ttimes S^1\ttimes I=(D^2;(2,1),(2,1))=(\mathcal{MB};)$, where $\mathcal{MB}$ denotes the M\"obius band;
\item $P\times S^1=(P;)$, where $P=\Sigma_{0,3}$ denotes a pair of pants\iffalse (or the three-punctured sphere)\fi;
\item $(S^1\times I; (p,q))$, $p>1$, ie a Seifert fibered manifold with base orbifold an annulus with one singular point;
\item $(D^2;(p_1,q_1),(p_2,q_2))$, $p_1,p_2>1$ and $(p_1,p_2)\neq (2,2)$, ie a Seifert fibered manifold with base orbifold a disk with two singular points and having negative Euler characteristic.
\end{itemize}
\end{enumerate}

It suffices to show that these three conditions hold for $M_i$ if and only if they hold for $N_i$.

%Note that the fundamental group of a compact 3-manifold is residually finite, according to \com{Hem} together with the virtual Haken theorem \com{Agol}. Therefore,
\iffalse
The residual finiteness of $\pi_1M_i$ and $\pi_1N_i$ (cf.  \autoref{PROP: RF}) implies that 
$\pi_1M_i$ is infinite if and only if $\widehat{\pi_1M_i}$ is infinite, and $\pi_1N_i$ is infinite if and only if $\widehat{\pi_1N_i}$ is infinite. This shows that condition (\ref{con1}) can be detected from the isomorphism type of the profinite completion.
\fi

The detection of condition \ref{con2} follows directly from \autoref{PROP: Detect hyperbolic}.

Moving on to condition \ref{concha}, suppose $M_i$ is a closed Seifert fibered space that admits a Seifert fibration over $S^2$ with at most 3 exceptional fibers, and $M_i$ does not admit $\mathbb{E}^3$ geometry. Then $N_i$ is also closed since $f_i$ respects the peripheral structure, and it follows from \cite[Theorem 8.4]{WZ17} that $N_i$ is also a Seifert fibered space admitting the same geometry of $M_i$. In particular, $N_i$ does not admit $\mathbb{E}^3$ geometry. In addition, $M_i$ and $N_i$ cannot admit $\mathbb{S}^2\times \mathbb{E}^1$ geometry since they are irreducible.  If $M_i$ admits $\mathbb{S}^3$ geometry, then $N_i$ also admits $\mathbb{S}^3$ geometry, and any such manifold admits a Seifert fibration over $S^2$ with at most 3 exceptional fibers. If they admit $Nil$ geometry or $\widetilde{\SL(2,\mathbb{R})}$ geometry, \cite[Theorem 1.2]{Wil17} implies that $M_i$ is homeomorphic with $N_i$, so the condition holds automatically for $N_i$. Finally, if $M_i$ and $N_i$ admit $\mathbb{H}^2\times \mathbb{E}^1$ geometry, then they admit unique Seifert fibrations; and by \cite[Theorem 5.2]{Wil17}, $M_i$ and $N_i$ have the same base orbifold though they are not necessarily homeomorphic. Hence, $N_i$ also admits a Seifert fibration over $S^2$ with at most 3 exceptional fibers if $M_i$ does.

For \ref{c?}, suppose $M_i$ belongs to this list. Since $M_i$ and $N_i$ are irreducible with non-empty incompressible toral boundary, it follows from the results of Wilton-Zalesskii \cite{WZ19,WZ17b}   that $N_i$ is also Seifert fibered. The isomorphism $f_i:\widehat{\pi_1M_i}\ttt\widehat{\pi_1N_i}$   respects the peripheral structure, so \cite[Theorem 5.8]{Wil17}   implies that $M_i$ and $N_i$ have the same base orbifolds. The list of Seifert fibered spaces in \ref{c?} are only characterized by their base orbifolds. Therefore, $N_i$ belongs to the list when $M_i$ does, which finishes the proof.
\end{proof}

\begin{remark}
It is necessary to assume that $f$ respects the peripheral structure. In fact, let $\Sigma_{0,3}$ denote a three-punctured sphere and $\Sigma_{1,1}$ denote a once-punctured torus. Then $\pi_1(\Sigma_{0,3}\times S^1)\cong F_2\times \Z\cong \pi_1(\Sigma_{1,1}\times S^1)$. However, $\Sigma_{0,3}\times S^1$ is atoroidal while $\Sigma_{1,1}\times S^1$ is toroidal. Indeed, these two manifolds have different numbers of boundary components, so any isomorphism $f:\widehat{\pi_1(\Sigma_{0,3}\times S^1)} \to \widehat{\pi_1(\Sigma_{1,1}\times S^1)}$ does not respect the peripheral structure.
\end{remark}

\subsection{Detection of embedded Klein bottles}
For brevity of notation, we denote by $\SK$ the orientable $I$-bundle over the Klein bottle. In fact, $\SK\cong S^1\ttimes S^1\ttimes I$ admits two Seifert fibrations $(D^2;(2,1),(2,1))\cong (\mathcal{MB};)$, where $\mathcal{MB}$ denotes the M\"obius band. 
\begin{lemma}\label{LEM: Klein}
Let $M$ be a compact, orientable, aspherical 3-manifold with empty or toral boundary. Then $M$ contains an embedded Klein bottle if and only if one of the following conditions hold:
\begin{enumerate}[label=(\roman*), leftmargin=*]
\item\label{4.6-1} $M$ is closed, admitting $Sol$ geometry, but does not fiber over $S^1$.
\item\label{4.6-2} $M$ is not a closed $Sol$-manifold, and the JSJ-decomposition (possibly trivial) of $M$ contains a Seifert fibered piece whose (one possible) base orbifold is
\begin{itemize}
\item either non-orientable,
\item or contains two singular points of index $2$.
\end{itemize}
\end{enumerate}
\end{lemma}
\begin{proof}
The ``if'' direction is clear. Indeed, if $M$ satisfies condition \ref{4.6-1}, then $M$ is a torus semi-bundle. In this case, $M$ contains an embedded submanifold homeomorphic to $\mathcal{N}$, and thus contains an embedded Klein bottle. If $M$ satisfies condition \ref{4.6-2}, then $M$ contains a Seifert fibered submanifold, whose base orbifold is either a M\"obius band, or a disk with two singular points of index $2$. Then, this Seifert fibered submanifold is homeomorphic to $\mathcal{N}$, and thus $M$ contains an embedded Klein bottle.

Now we prove the ``only if'' direction.  When $M$ has compressible boundary,  we have $M\cong D^2\times S^1$ since $M$ is aspherical, and in this case, $M$ does not contain an embedded Klein bottle. Thus, we may assume that $M$ is boundary-incompressible. Suppose that $M$ contains an embedded Klein bottle $K^2$. Then the regular neighbourhood of $K^2$ in $M$ is homeomorphic to $\mathcal{N}$.

 Let $T^2_{\mathcal N}$ denote the boundary torus of $\mathcal{N}$.  
We first claim that $T^2_{\mathcal N}$ is incompressible in $M$. Suppose  by contrary that $T^2_{\mathcal N}$ is compressible; since $M$ is aspherical, $T^2_{\mathcal N}$ bounds a solid torus in $M$ according to the loop theorem. Then $M=\mathcal{N}\cup_{T^2_{\mathcal N}} D^2\times S^1$. Note that $\mathcal{N}$ admits a two-fold cover $\widetilde{\mathcal{N}}\cong S^1\times S^1\times I$, where the covering map restricts to a homeomorphism from each boundary component of $\widetilde{\mathcal{N}}$ to the boundary of $ \mathcal{N}$. Therefore, $M$ admits a two-fold cover $\widetilde{M}\cong D^2\times S^1\cup_{T^2}D^2\times S^1$, which is either $S^2\times S^1$ or a lens space. In either case, $\widetilde{M}$ is non-aspherical, so $M$ is non-aspherical, which contradicts with our assumption.

Let $M'=M\setminus \mathring{\mathcal{N}}$. Then $\partial M'=\partial M \sqcup T^2_{\mathcal  N}$, which is incompressible. 

Let us first consider two special cases. If $M'\cong S^1\times S^1\times I$, then $M\cong \mathcal{N}$ and satisfies condition \ref{4.6-2}. If $M'$ is homeomorphic to $\mathcal{N}$, then $M$ is a torus semi-bundle. 
\revised{
Both $\mathcal{N}$ and $M'$ admit two Seifert fibrations. Up to isotopy, if one of the two possible regular fibers of $\mathcal{N}$ is glued along $T^2_{\mathcal{N}}$ to one of the two  possible regular fibers of $M'$, then the Seifert fibration structures on both sides combine into a Seifert fibration of $M$  whose base orbifold is obtained by gluing the base orbifolds of both sides along their boundary circles. In particular, the base orbifold  belongs to the following three possibilities:
\begin{itemize}
\item $S^2$ with four singular points of index $2$,
\item $\mathbb{RP}^2$ with two singular points of index $2$,
\item a Klein bottle $K^2$,
\end{itemize}
any of which satisfies condition \ref{4.6-2}. Otherwise, the gluing along $T^2_{\mathcal{N}}$ is never fiber-parallel, so by \cite[Proposition 1.6.2]{AFW15}, $T^2_{\mathcal{N}}$ is a JSJ-torus of $M$. In this case, $M$ is a $Sol$-manifold satisfying condition \ref{4.6-1}.
}

Now we assume that $M'$ is neither homeomorphic to $S^1\times S^1\times I$, nor to $\mathcal{N}$, so $M'$ is not  a closed $Sol$-manifold. Consider the JSJ-decomposition of $M'$, whose collection of JSJ-tori is denoted by $\mathcal{T}'$ (possibly empty). Then, there is a unique JSJ-piece in $M'$ which contains the boundary component $T^2_{\mathcal N}$, and we denote this JSJ-piece as $M_0$.

When $M_0$ is a hyperbolic piece, let $\mathcal{T}=\mathcal{T}'\cup \{T^2_{\mathcal{N}}\}$. Then \cite[Proposition 1.6.2]{AFW15} implies that $\mathcal{T}$ is exactly the collection of JSJ-tori in $M$. Therefore, the JSJ-decomposition of $M$ contains $\mathcal{N}$ as a Seifert fibered piece, so it satisfies condition \ref{4.6-2}.

When $M_0$ is a Seifert fibered piece, note that $M'$ is not $S^1\times S^1\times I$ or $\mathcal{N}$, so the base orbifold of $M_0$ has negative Euler characteristic, and $M_0$ admits a unique Seifert fibration. On the other hand, $\mathcal{N}$ admits two Seifert fibrations, one of which has base orbifold a M\"obius band, and the other has base orbifold a disk with two singular points of index $2$.

If the regular fiber of $M_0$ is not parallel to any of the two regular fibers of $\mathcal{N}$ on the common boundary $T^2_{\mathcal{N}}$, then by \cite[Proposition 1.6.2]{AFW15}, $\mathcal{T}=\mathcal{T}'\cup \{T^2_{\mathcal{N}}\}$ is still the JSJ-tori of $M$. Similarly, $M$ contains $\mathcal{N}$ as a JSJ-piece, and satisfies condition \ref{4.6-2}.

If the regular fiber of $M_0$ is parallel to one of the two regular fibers of $\mathcal{N}$. Then, $\mathcal{T}'$ is in fact the JSJ-tori of $M$, and $\mathcal{N}\cup_{T^2_{\mathcal{N}}}M_0$ yields an entire Seifert fibered piece in $M$. The base orbifold of $\mathcal{N}\cup M_0$ is obtained by gluing the base orbifold of $M_0$ with the base orbifold of $\mathcal{N}$ along an $S^1$-boundary, and the choice of the base orbifold of $\mathcal{N}$ depends on which regular fiber of $\mathcal{N}$ is parallel to that of $M_0$ on $T^2_\mathcal{N}$. Thus, the base orbifold of $\mathcal{N}\cup M_0$ either contains a M\"obius band, which is non-orientable, or contains a disk with two singular points of index $2$. This shows that $M$ satisfies condition \ref{4.6-2}.
\end{proof}

\begin{lemma}\label{LEM: Orbifold}
Let $M$ and $N$ be orientable, aspherical Seifert fibered spaces with either empty or incompressible boundary, such that $\widehat{\pi_1M}\cong \widehat{\pi_1N}$. Suppose $M$ admits a Seifert fibration over a base orbifold $\O$, where $\O$ is either non-orientable, or contains two singular points with index $2$. Then $N$ also admits a Seifert fibration over a base orbifold $\O'$, which is either non-orientable, or contains two singular points with index $2$.
\end{lemma}
\begin{proof}
We first consider the case when $M$ is closed. Then  \cite[Corollary 4.12]{Reid:2018} implies that $N$ is also closed. According to \cite[Theorem 1.2]{Wil17}, if $M$ admits $\mathbb{E}^3$, $Nil$ or $\widetilde{\mathrm{SL}(2,\mathbb{R})}$ geometry, then $N$ is homeomorphic to $M$, and the conclusion is trivial. On the other hand, if $M$ admits $\mathbb{H}^2\times \mathbb{R}$ geometry, then $M$ admits a unique Seifert fibration. \cite[Theorem 5.2]{Wil17} implies that $N$ admits the same geometry, and has the same base orbifold as $M$  
%, whose base orbifold is denoted by $\O_M$. \cite[Theorem 6.2 (1)]{Wil17} implies that $N$ admits the same geometry as $M$, so $N$ also admits a unique Seifert fibration over $\O_N$. \cite[Theorem 6.2 (3)]{Wil17} shows that $\O_M\cong \O_N$, 
even though $M$ and $N$ may not be homeomorphic, which implies our conclusion.

Now let us consider the case when $M$ has non-empty boundary. Then the above reasoning shows that $N$ also has non-empty boundary. If $M$ is homeomorphic to $S^1\times S^1\times I$ or $\mathcal{N}$, then it is easy to verify that $N\cong M$ (cf. \cite[Proposition 8.2]{Xu}), and the conclusion follows directly. Thus, we may assume that $M$ fibers over a hyperbolic orbifold. In this case, $M$ admits a unique Seifert fibration, and we denote the base orbifold as $\O_M$. In addition, $N$ also fibers over a hyperbolic orbifold, and we denote the unique base orbifold as $\O_N$. 
\iffalse
\com{Wilkes} showed that the closure of the infinite cyclic subgroup generated by a regular fiber of $M$ (resp. $N$), denoted by $\overline{K_M}$ (resp. $\overline{K_N}$), is the unique maximal virtually central procyclic subgroup  in $\widehat{\pi_1M}$ (resp. $\widehat{\pi_1N}$). $\O_M$ (resp. $\O_N$) is orientable if and only if $\overline{K_M}$ (resp. $\overline{K_N}$) is central in $\widehat{\pi_1M}$ or $\widehat{\pi_1N
\fi
In this case, the base orbifolds $\O_M$ and $\O_N$ may not be identical. 
However, in \cite[Proposition 8.2]{Xu}, it is proven that the isomorphism type of $\widehat{\pi_1M}$ (resp. $\widehat{\pi_1N}$) determines whether the base orbifold $\O_M$ (resp. $\O_N$) is orientable, and determines the isomorphism type of the orbifold fundamental group $\pi_1(\O_M)$ (resp. $\pi_1(\O_N)$). Thus, $\widehat{\pi_1M}\cong \widehat{\pi_1N}$ implies that $\O_M$ is non-orientable if and only if $\O_N$ is non-orientable. In addition, there is a one-to-one correspondence between the conjugacy classes of maximal torsion subgroups in $\pi_1(\O_M)$ (resp. $\pi_1(\O_N)$) with the cone points in $\O_M$ and $\O_N$. To be explicit, $\O_M$ contains two singular points of index $2$ if and only if $\pi_1(\O_M)\cong \pi_1(\O_N)$ contains two conjugacy classes of maximal torsion subgroups isomorphic to $\Z/2\Z$, which implies that $\O_N$ also contains two singular points of index $2$.
\end{proof}

\begin{lemma}\label{LEM: Profinite detect Klein}
Let $M$ and $N$ be compact, orientable 3-manifolds with empty or toral boundary such that $\widehat{\pi_1M}\cong \widehat{\pi_1N}$. Suppose $M$ is aspherical, and $M$ contains an embedded Klein bottle. Then $N$ is also aspherical and contains  an embedded Klein bottle.
\end{lemma}
\begin{proof}
Since $M$ contains an embedded Klein bottle, $M$ is not homeomorphic to $D^2\times S^1$. As $M$ is aspherical, it follows from \autoref{Cor: Spherical} that $N$ is also aspherical. In addition, $M$ is boundary-incompressible, and it follows from \autoref{LEM: Prime} that $N$ is also boundary-incompressible.

It suffices to show that if condition \ref{4.6-1} or \ref{4.6-2} in \autoref{LEM: Klein} holds for $M$, then the same condition holds for $N$. 

First, we suppose that  $M$ satisfies condition  \ref{4.6-1} in \autoref{LEM: Klein}, ie $M$ is a closed $Sol$-manifold which is not fibered over $S^1$. Since $\widehat{\pi_1M}\cong \widehat{\pi_1N}$, it follows from \cite[Corollary 4.12]{Reid:2018} that $N$ is also closed, and \cite[Theorem 8.4]{WZ17} shows that $N$ also admits $Sol$-geometry. In addition, \cite[Corollary 1.2]{Jz20} shows that the profinite completion of fundamental group of a compact, orientable 3-manifold determines whether this 3-manifold fibers over $S^1$. Thus, $N$ is also non-fibered, and satisfies condition  \ref{4.6-1} of \autoref{LEM: Klein}.

Next, we suppose that $M$ satisfies condition  \ref{4.6-2} in \autoref{LEM: Klein}. Similarly, \cite[Theorem 8.4]{WZ17} implies that $N$ is not a closed $Sol$-manifold.  \cite[Theorem B]{WZ19} shows that the profinite completion of the fundamental group detects the JSJ-decomposition of a compact, orientable, irreducible 3-manifold with empty or incompressible toral boundary which does not admit $Sol$ geometry. 
In fact, \cite[Theorem B]{WZ19} was originally proven for closed manifolds, but the proof can be easily generalized to the bounded case, see for example \cite[Theorem 5.5]{Xu}. 
In particular, there is a one-to-one correspondence between the Seifert fibered JSJ-pieces of $M$ and $N$, such that the corresponding pieces are profinitely isomorphic. Since one of the Seifert pieces of $M$ has a base orbifold which is either non-orientable or contains two singular points of index $2$, and this Seifert piece is obviously boundary incompressible and aspherical, it follows from \autoref{LEM: Orbifold} that the corresponding Seifert piece of $N$ also has a base orbifold which is either non-orientable or contains two singular points of index $2$. Therefore, $N$ also satisfies condition \ref{4.6-2} in \autoref{LEM: Klein}.
\end{proof}

\begin{remark}
The aspherical condition is important. In fact, there exist two lens spaces with isomorphic fundamental groups, such that one of them contains a Klein bottle while the other one does not; see \cite[Corollary 6.4]{Bre69}.
\end{remark}

\subsection{Profinite detection of exceptional Dehn fillings}

\begin{proposition}\label{PROP: Exceptional}
Suppose $M$ and $N$ are orientable cusped finite-volume hyperbolic 3-manifolds, and $\widehat{\pi_1M}\cong \widehat{\pi_1N}$. Let $\partial_{i_0}M$ be a boundary component of $M$. Then,  there exists a boundary component $\partial_{j_0}N$ of $N$, and an isomorphism $\psi:\pi_1\partial_{i_0}M\ttt\pi_1\partial_{j_0}N$, such that 
\begin{enumerate}[label=(\arabic*), leftmargin=*]
\item\label{4.3-1} 
$\pisi(\mathcal{E}(M,\partial_{i_0}M))=\mathcal{E}(N,\partial_{j_0}N)$, where $\pisi$ is the bijection $\slope(\partial_{i_0}M)\to \slope(\partial_{j_0}N)$;
%the map ${\pisi}{:\slope(\partial_{i_0}M)\to \slope(\partial_{j_0}N)}$ establishes a bijection between $\mathcal{E}(M,\partial_{i_0}M)$ and $\mathcal{E}(N,\partial_{j_0}N)$;
%, and that 
\item\label{4.3-2} $\pisi(\mathcal{E}_s(M,\partial_{i_0}M))=\mathcal{E}_s(N,\partial_{j_0}N)$;
\item\label{4.3-3} $\pisi(\mathcal{E}_t(M,\partial_{i_0}M))=\mathcal{E}_t(N,\partial_{j_0}N)$; 
\item\label{4.3-4} $\pisi(\mathcal{E}_k(M,\partial_{i_0}M))=\mathcal{E}_k(N,\partial_{j_0}N)$.%\com{, and $\pisi(\mathcal{E}_k(M,\partial_{i_0}M))=\mathcal{E}_k(N,\partial_{j_0}N)$}.
\end{enumerate}
\end{proposition}
%We are now ready to prove \autoref{PROP: Exceptional}.

\begin{proof}%[Proof of \autoref{PROP: Exceptional}]
According to \hyperref[PROP: Dehn filling]{Theorem A}, there is a boundary component $\partial_{j_0}N$ of $N$, and an isomorphism $\psi:\pi_1\partial_{i_0}M\ttt\pi_1\partial_{j_0}N$, such that for any $c\in \slope(\partial_{i_0}M)$, $\widehat{\pi_1M_c}\cong \widehat{\pi_1N_{\pisi(c)}}$ by an isomorphism which respects the peripheral structure. 

According to \autoref{PROP: Detect hyperbolic}, $M_c$ is finite-volume hyperbolic if and only if $N_{\pisi(c)}$ is finite-volume hyperbolic. Hence, by definition of exceptional Dehn filling, $c\in \mathcal{E}(M,\partial_{i_0}M)$ if and only if $\pisi(c)\in \mathcal{E}(N,\partial_{j_0}N)$, which proves \ref{4.3-1}. 
Since $M_c$ and $N_{\pisi(c)}$ have the same number of boundary components,  \autoref{Cor: Spherical} implies that $M_c$ is aspherical if and only if $N_{\pisi(c)}$ is aspherical, which proves \ref{4.3-2}. 
Similarly, \ref{4.3-3} follows from \autoref{LEM: Toroidal} since the isomorphism $\widehat{\pi_1M_c}\cong \widehat{\pi_1N_{\pisi(c)}}$ respects the peripheral structure, and \ref{4.3-4} follows from \autoref{LEM: Profinite detect Klein}.
\end{proof}

As a typical example, when $M=X_K$ and $N=X_{K'}$ are hyperbolic knot complements, we have established a standard parametrization $\slope(\partial M)=\Qinf$ and $\slope(\partial N)=\Qinf$ in \autoref{subsec: knot}. Thus, $\mathcal{E}(M,\partial M)$ and $\mathcal{E}(N,\partial N)$ can be viewed as subsets of $\Qinf$. For brevity of notation, when $M=X_K$, we simply denote $\mathcal{E}(M,\partial M)$ as $\mathcal{E}(K)$.

Based on \autoref{LEM: Knot complement}, we  obtain a more precise version of \autoref{PROP: Exceptional} for hyperbolic knot complements.

\begin{proposition}\label{PROP: Knot exceptional coefficient}
Suppose $K$ and $K'$ are hyperbolic knots in $S^3$, and $\widehat{\pi_1X_K}\cong \widehat{\pi_1X_{K'}}$. Then there exists $\sigma\in \{1,-1\}$, such that
\begin{enumerate}[label=(\arabic*), leftmargin=*]
\item $\mathcal{E}(K)=\sigma\cdot \mathcal{E}(K')$;
\item $\mathcal{E}_s(K)=\sigma\cdot \mathcal{E}_s(K')$;
\item $\mathcal{E}_t(K)=\sigma\cdot \mathcal{E}_t(K')$;
\item $\mathcal{E}_k(K)=\sigma\cdot \mathcal{E}_k(K')$.
%\item \com{$\mathcal{E}_k(K)=\sigma\cdot \mathcal{E}_k(K')$.}
\end{enumerate}
\end{proposition}
\begin{proof}
As in the proof of \autoref{PROP: Exceptional}, let $\psi:\pi_1\partial X_K\ttt\pi_1\partial X_{K'}$ denote the isomorphism such that $\widehat{\pi_1 (X_K)_c}\cong \widehat{\pi_1 (X_{K'})_{\pisi(c)}}$ for any $c\in \slp(\partial X_K)$. 
According to \autoref{LEM: Knot complement}, under the parametrization by $\Qinf$, the map $\pisi: \slope(\partial X_K)=\Qinf\to \slope(\partial X_{K'})=\Qinf$ is a multiplication by $\sigma \in \{\pm 1\}$. 
Therefore, the results of this proposition follow directly from \autoref{PROP: Exceptional}.
\end{proof}

\section{Profinite rigidity from exceptional slopes}\label{SEC: Main}

%\subsection{Exceptional slopes with large geometric intersection number}
%\subsection{Geometric intersection number}
In this section, we shall prove \autoref{MAIN}. 
The key concept we use in this section is the {\em geometric intersection number}. 
For two slopes on a torus $c_1,c_2\in \slope(T^2)$, we denote by $\Delta(c_1,c_2)$ the geometric intersection number of $c_1$ and $c_2$. Note that $\Delta(c_1,c_2)$ is a non-negative integer, and $\Delta(c_1,c_2)=0$ if and only if $c_1=c_2$.

\begin{example}\label{ex}
Let $K$ be a knot in $S^3$, and we parametrize $\slope(\partial X_K)$ by $\Qinf$ as in \autoref{subsec: knot}. For $r_1,r_2\in \Qinf$, we can express them as   reduced fractions $r_1=\frac{p_1}{q_1}$ and $r_2=\frac{p_2}{q_2}$. Then $\Delta(r_1,r_2)=\abs{\det \left(\begin{smallmatrix} p_1 & q_1 \\ p_2 & q_2\end{smallmatrix}\right)}$.
\end{example}

The next lemma is an obvious but useful observation.
\begin{lemma}\label{LEM: preserve}
Let $T_1^2$ and $T_2^2$ denote two tori, and let $\psi:\pi_1(T_1^2)\ttt\pi_1(T_2^2)$ be a group isomorphism. Then, for any slopes $c_1,c_2\in \slope(T_1^2)$, $\Delta(c_1,c_2)=\Delta(\pisi(c_1),\pisi(c_2))$.
\end{lemma}
\begin{proof}
One can view $\psi:\pi_1(T_1^2)\ttt\pi_1(T_2^2)$ as a homeomorphism $\Psi: T_1^2\to T_2^2$, and directly derive this lemma by definition.
\end{proof}

%The following results, proven in a series of works by \com{???}, enable us to determine certain cusped hyperbolic 3-manifolds through their exceptional Dehn fillings.

\subsection{Some Dehn surgeries on the Whitehead link}
We start from the easiest example, so as to illustrate how our method works. In fact, the manifolds listed in \autoref{MAIN} can be characterised by their extremely special patterns of exceptional Dehn fillings. 
%The following result enables us to determine a certain cusped hyperbolic 3-manifolds through their exceptional (toroidal) Dehn fillings.

For $r\in \Qinf$, let $\mathcal{W}(r)$ denote the 3-manifold obtained by $r$-surgery on one component of the Whitehead link (\autoref{Fig: W}).

\begin{proposition}[{\cite[Theorem 1.1]{Gor98}}]\label{52}
Suppose $N$ is a finite-volume hyperbolic 3-manifold with one cusp. 
\begin{enumerate}[label=(\arabic*), leftmargin=*]
\item\label{52.1} $\max\{\Delta(c_1,c_2)\mid c_1,c_2\in \mathcal{E}_t(N,\partial N)\}=6$ if and only if $N$ is homeomorphic to $\mathcal{W}(-2)$. 
\item\label{52.2}  $\max\{\Delta(c_1,c_2)\mid c_1,c_2\in \mathcal{E}_t(N,\partial N)\}=7$ if and only if $N$ is homeomorphic to $\mathcal{W}(\frac{5}{2})$.
\item\label{52.3} $\max\{\Delta(c_1,c_2)\mid c_1,c_2\in \mathcal{E}_t(N,\partial N)\}=8$ if and only if $N$ is homeomorphic to either $\mathcal{W}(-1)$ (the figure-eight knot complement) or $\mathcal{W}(5)$ (the figure-eight sibling manifold).
\end{enumerate}
\end{proposition}

\begin{proposition}[{\cite[Corollary 1.5]{Lee07}}]\label{W-4}
Suppose $N$ is a finite-volume hyperbolic 3-manifold with one cusp. Then  $\max\{\Delta(c_1,c_2)\mid c_1,c_2\in \mathcal{E}_k(N,\partial N)\}=5$ if and only if $N$ is homeomorphic to $\mathcal{W}(-4)$. 
\end{proposition}
\begin{proof}
The ``if'' direction follows from \cite[Table A.4]{MP02}, and the ``only if'' direction follows from \cite[Corollary 1.5]{Lee07}. 
\end{proof}

It is well-known that all the exceptional slopes on the Whitehead link (\autoref{Fig: W}) are $\infty,0,1,2,3,4$ \cite{MP02}. Thus, the above listed $\mathcal{W}(-2)$, $\mathcal{W}(\frac{5}{2})$, $\mathcal{W}(-1)$, $\mathcal{W}(5)$ and $\mathcal{W}(-4)$ are all finite-volume hyperbolic 3-manifolds with one cusp. Note that the Dehn surgery coefficients in \cite{Gor98} differ from the ones we use here by a $\pm$-sign,  due to the chirality of the Whitehead link.

\begin{theorem}\label{THM: W(-5/2)}
The manifolds $\mathcal{W}(-2)$, $\mathcal{W}(\frac{5}{2})$, $\mathcal{W}(-4)$, $\mathcal{W}(-1)$ and $\mathcal{W}(5)$ are all profinitely rigid among compact  orientable 3-manifolds.
\end{theorem}
\begin{proof}
%For simplicity, we denote $M=\mathcal W (-\frac{5}{2})$. 
%It is known that all the exceptional slopes on the Whitehead link are $\infty,0,-1,-2,-3,-4$ \com{cite}, so $\mathcal W(-\frac{5}{2})$ is a finite-volume hyperbolic 3-manifold with one cusp. 
We start from the example of $\mathcal{W}(-2)$. 
Let $N$ be a compact orientable 3-manifold so that $\widehat{\pi_1\mathcal{W}(-2)}\cong \widehat{\pi_1N}$. Then, \autoref{PROP: Detect hyperbolic} implies that $N$ is also finite-volume hyperbolic with exactly one cusp. Therefore, \autoref{PROP: Exceptional} implies that there is an isomorphism $\psi:\pi_1(\partial \mathcal{W}(-2))\to\pi_1(\partial N)$, such that $\pisi(\mathcal{E}_t(\mathcal{W}(-2),\partial \mathcal{W}(-2)))=\mathcal{E}_t(N,\partial N)$.  Note that $\pisi$ is induced by a group isomorphism, so \autoref{LEM: preserve} implies that $\pisi$ preserves the geometric intersection number $\Delta(\cdot,\cdot)$. Thus, $\max\{\Delta(c_1,c_2)\mid c_1,c_2\in \mathcal{E}_t(N,\partial N)\}=\max\{\Delta(c_1',c_2')\mid c_1',c_2'\in \mathcal{E}_t(\mathcal{W}(-2),\partial \mathcal{W}(-2))\}=6$. According to \autoref{52}~\ref{52.1},  the only possibility is $N\cong \mathcal W(-2)$. 

The profinite rigidity of $\mathcal{W}(\frac{5}{2})$ and $\mathcal{W}(-4)$ follow from the same proof as $\mathcal W(-2)$, based on \autoref{52}~\ref{52.2} and \autoref{W-4} respectively.

Moving on to $\mathcal{W}(-1)$ and $\mathcal{W}(5)$, the same proof implies that if a compact orientable 3-manifold $N$ is profinitely isomorphic to $\mathcal{W}(-1)$ or $\mathcal{W}(5)$, then $N$ is homeomorphic to one of $\mathcal{W}(-1)$ and $\mathcal{W}(5)$. Thus, it suffices to show that $\mathcal{W}(-1)$ and $\mathcal{W}(5)$ are not profinitely isomorphic. In fact, $\mathcal{W}(-1)$ is the figure-eight knot complement, while the figure-eight sibling manifold $\mathcal{W}(5)$ is not a knot complement. Therefore, \autoref{PROP: Detect knot complement} guarantees that $\widehat{\pi_1\mathcal{W}(-1)}\not\cong \widehat{\pi_1\mathcal{W}(5)}$, finishing the proof.
\end{proof}

\subsection{Whitehead link, Whitehead sister link, $\frac{3}{10}$ two-bridge link, and $M_{14}$}

Recall that $\mathcal{W}$, $\mathcal{L}$ and $\mathcal{WS}$ denote the Whitehead link, the two-bridge link defined by the rational number $\frac{3}{10}$ in the Schubert normal form, and the Whitehead sister link which is the $(-2,3,8)$-Pretzel link, respectively. %\com
{$M_{14}$ is the double branched cover along the tangle $\mathcal{Q}$ in $S^2\times I$ (\autoref{Fig: M14}) defined in \cite{GW}.}

\begin{proposition}[{\cite[Corollary 1.3]{GW}}]\label{PROP: GW1}
Suppose $N$ is a finite-volume hyperbolic 3-manifold with two cusps, and let $\partial_0N$ be a boundary component of $N$.
\begin{enumerate}[label=(\arabic*), leftmargin=*]
\item\label{GW.1}  $\max\{\Delta(c_1,c_2)\mid c_1,c_2\in \mathcal{E}_t(N,\partial_0N)\}=4$ if and only if $N$ is homeomorphic to one of $\cpl{\mathcal W}$, $\cpl{\mathcal L}$, and $M_{14}$.
\item\label{GW.2}  $\max\{\Delta(c_1,c_2)\mid c_1,c_2\in \mathcal{E}_t(N,\partial_0N)\}=5$ if and only if  $N$ is homeomorphic to $\cpl{\mathcal{WS}}$.
\end{enumerate}
\end{proposition}

For simplicity, we denote $M_1=\cpl{\mathcal{W}}$, $M_2=\cpl{\mathcal{L}}$, and $M_3=\cpl{\mathcal{WS}}$ following the notation of \cite{GW}. It is well-known that $M_1$, $M_2$ and $M_3$ are finite-volume hyperbolic 3-manifolds, and \cite[Lemma 23.13]{GW} proved that $M_{14}$ is also finite-volume hyperbolic.

\begin{lemma}\label{LEM: distinguish 1,2,14}
Any two of $\widehat{\pi_1M_1}$, $\widehat{\pi_1M_2}$, $\widehat{\pi_1M_{14}}$ are not isomorphic.
\end{lemma}
\begin{proof}

We first show that $\widehat{\pi_1M_1}\not \cong \widehat{\pi_1M_{14}}$ and $\widehat{\pi_1M_2}\not \cong \widehat{\pi_1M_{14}}$. In fact, $M_1$ and $M_2$ are both hyperbolic link complements in $S^3$, while in the proof of \cite[Lemma 24.2]{GW}, Gordon-Wu showed that $M_{14}$ is not a link complement. Therefore, \autoref{PROP: Detect knot complement} implies $\widehat{\pi_1M_1}\not \cong \widehat{\pi_1M_{14}}$ and $\widehat{\pi_1M_2}\not \cong \widehat{\pi_1M_{14}}$. %Thus, for $j=1,2$, there exist slopes $c_j^{\prime}$ and $c_j^{\prime\prime}$ on the two boundary components of $M_j$ such that the Dehn filling $(M_j)_{c_j^\prime,c_j^{\prime\prime}}\cong S^3$. If either $\widehat{\pi_1M_1} \cong \widehat{\pi_1M_{14}}$ or $\widehat{\pi_1M_2}\cong \widehat{\pi_1M_{14}}$, \autoref{PROP: Dehn filling} then implies that there exist slopes $c',c''$ on two boundary components of $M_{14}$, such that $\widehat{\pi_1(M_{14})_{c',c''}}$ is the trivial group. 

Next, we show that $\widehat{\pi_1M_1}\not\cong \widehat{\pi_1M_2}$. 
\revised{
Suppose by contrary that $\widehat{\pi_1M_1}\cong \widehat{\pi_1M_2}$. Then, by \autoref{inthm: Dehn filling}, there exist a boundary component $\partial_\ast M_1$, a boundary component $\partial_\star M_2$, and a bijection $\pisi:  \slope(\partial_\ast M_1) \to \slope(\partial_\star M_2)$ such that for any slope $c\in \slope(\partial_\ast M_1)$, $\widehat{\pi_1(M_1)_c}\cong \widehat{\pi_1(M_2)_{\pisi(c)}}$. In particular, $b_1((M_1)_c)= b_1 ((M_2)_{\pisi(c)})$ by \autoref{lem: abelianize}.  %Now, let  $c\in \slope(\partial_\ast M_1)$ be the longitude of the component of $\mathcal{W}$ correspondin 

The boundary component $\partial_\ast M_1$ is given by a component $K_\ast$ of the Whitehead link $\mathcal{W}$. Now we let $c\in \slope(\partial_\ast M_1)$ be the longitude of $K_\ast$. Since the linking number of the two components in $\mathcal{W}$ is $0$, the slope $c$ is actually null homologous in $M_1$. Thus, $b_1((M_1)_c)=b_1(M_1)=2$. However,  the linking number of the two components in $\mathcal{L}$ is $\pm3$ which is  non-zero, so both boundary components of $M_2=\cpl{\mathcal{L}}$ are $H_1$-injective. Consequently, $b_1((M_2)_{c'})=1$ for any $c'\in \slope(\partial_\star M_2)$.  In particular, $b_1((M_1)_c)\neq b_1((M_2)_{\pisi(c)})$, which yields a contradiction. 
}
\end{proof}

\begin{theorem}\label{THM: link complement}
$M_1,M_2,M_3,M_{14}$ are profinitely rigid among compact, orientable 3-manifolds.
\end{theorem}
\begin{proof}
Recall that $M_1,M_2,M_3,M_{14}$  are finite-volume hyperbolic 3-manifolds with two cusps. 
According to \autoref{PROP: Detect hyperbolic}, any compact orientable 3-manifold profinitely isomorphic to $M_1$, $M_2$, $M_3$ or $M_{14}$ is also a finite-volume hyperbolic 3-manifolds with two cusps. 

We first prove the profinite rigidity of $M_3$. Suppose $N$ is a finite-volume hyperbolic 3-manifold with two cusps such that $\widehat{\pi_1M_3}\cong \widehat{\pi_1N}$. Let $\partial_0M_3$ be a boundary component of $M_3$ such that $\max\{\Delta(c_1,c_2)\mid c_1,c_2\in \mathcal{E}_t(M_3,\partial_0M_3)\}=5$ according to \autoref{PROP: GW1}. Then similar to the proof of \autoref{THM: W(-5/2)}, we can derive from \autoref{PROP: Exceptional} and \autoref{LEM: preserve} that there exists a boundary component $\partial_iN$ of $N$ such that $\max\{\Delta(c_1',c_2')\mid c_1',c_2'\in \mathcal{E}_t(N,\partial_iN)\}=\max\{\Delta(c_1,c_2)\mid c_1,c_2\in \mathcal{E}_t(M_3,\partial_0M_3)\}=5$. \autoref{PROP: GW1}~\ref{GW.2} then implies that $N\cong M_3$.

Next, we show the profinite rigidity of $M_1$, $M_2$ and $M_{14}$. Suppose $j\in \{1,2,14\}$ and $N$ is a finite-volume hyperbolic 3-manifold with two cusps such that $\widehat{\pi_1M_j}\cong \widehat{\pi_1N}$. Similar to the proof for $M_3$, there exists a boundary component $\partial_iN$ of $N$ such that $\max\{\Delta(c_1,c_2)\mid c_1,c_2\in \mathcal{E}_t(N,\partial_iN)\}=4$. \autoref{PROP: GW1}~\ref{GW.1} then implies that $N$ is homeomorphic to one of $M_1$, $M_2$ and $M_{14}$. We have shown in \autoref{LEM: distinguish 1,2,14} that any two of $M_1$, $M_2$ and $M_{14}$ are not profinitely isomorphic. Hence, the only remaining possibility is  $N\cong M_j$.
\end{proof}

\subsection{More hyperbolic knot complements in $S^3$}

In this subsection, we consider three families of hyperbolic knots.
\begin{enumerate}[label=(\arabic*), leftmargin=*]
\item The twist knots $\mathcal{K}_n$ (\autoref{Fig: Kn}) can also be viewed as a $\frac{1}{n}$-Dehn surgery on one component of the Whitehead link (\autoref{Fig: W}). For example, $\mathcal{K}_0$ is the unknot, $\mathcal{K}_1$ is the trefoil, and $\mathcal{K}_{-1}$ is the figure-eight knot. In fact, according to \cite{MP02}, for all $n\in \Z\setminus\{0,1\}$, $\mathcal{K}_n$ is a hyperbolic knot.
\item The braid knots $\mathcal{J}_n$ represented by $\Jn$ (\autoref{Fig: Jn}) can be viewed as a $\frac{1}{n}$-surgery on one component of the two-bridge link $\mathcal{L}$ shown in \autoref{Fig: L}. Again, \cite{MP02} shows that $\mathcal{J}_n$ is hyperbolic if and only if $n\in \Z\setminus \{-1,0,1\}$.
\item The \EM knots $K(3,1,n,0)$ defined in \cite{EM96} (\autoref{Fig: EM}) can be viewed as a $-\frac{1}{n}$-surgery on the unknotted component of the Whitehead sister link $\mathcal{WS}$ shown in \autoref{Fig: WS}. Indeed, \cite[Proposition 2.2]{EM96} shows that $K(3,1,n,0)$ is hyperbolic for all $n\in \Z\setminus \{0\}$, and the mirror-image of $K(3,1,n,0)$ is  $K(2,-1,1-n,0)$ by  \cite[Proposition 1.4]{EM96}.
\end{enumerate}

\begin{proposition}[{\cite[Theorem 24.4]{GW}}]\label{THM: Knot maximal}
Suppose ${K}$ is a hyperbolic knot in $S^3$.
\begin{enumerate}[label=(\arabic*), leftmargin=*]
\item\label{km.3} If  there exist $c_1,c_2\in \mathcal{E}_t(K)$ such that $\Delta(c_1,c_2)=8$, then $K$ is the figure-eight knot $\mathcal{K}_{-1}={4_1}$.
\item\label{km.1}  If there exist $c_1,c_2\in \mathcal{E}_t(K)$ such that $\Delta(c_1,c_2)=4$, then $K$ is equivalent or mirror image to $\mathcal{K}_n$ ($n\in \Z\setminus\{0,1\}$) or $\mathcal{J}_n$ ($n\in \Z\setminus\{-1,0,1\}$).
\item\label{km.2}  If  there exist $c_1,c_2\in \mathcal{E}_t(K)$ such that $\Delta(c_1,c_2)=5$, then $K$ is the \EM knot $K(3,1,n,0)$ or its mirror-image $K(2,-1,1-n,0)$ for some $n\in \Z\setminus\{0\}$.
\end{enumerate}
\end{proposition}

\begin{proposition}\label{PROP: coefficient}
\begin{enumerate}[label=(\arabic*), leftmargin=*]
\item\label{ek.1} $\mathcal{E}_t(\mathcal{K}_{-1})=\{-4,0,4\}$. In particular, $\Delta(-4,4)=8$. 
\item\label{ek.2} For $n\in \Z\setminus\{-1,0,1\}$, $\mathcal{E}_t(\mathcal{K}_n)=\{0,4\}$, and $\Delta(0,4)=4$.
\item\label{ek.3} For $n\in \Z\setminus\{-1,0,1\}$, $\mathcal{E}_t(\mathcal{J}_n)=\{2-9n,-2-9n\}$, and $\Delta(2-9n,-2-9n)=4$.
\item\label{ek.4} For $n\in \Z\setminus \{0\}$,
 $\{25n-9,25n-\frac{13}{2}\}\subseteq \mathcal{E}_t(K(3,1,n,0))$, and  $\Delta(25n-9,25n-\frac{13}{2})=5$. 
%;
%\item $\max\{\Delta(c_1,c_2)\mid c_1,c_2\in \mathcal{E}_t(K(3,1,n,0))\}=5$;
 In addition, 
%\item 
$25n-\frac{13}{2}$ is the unique half-integral element in $\mathcal{E}_t(K(3,1,n,0))$.% In particular,
%\end{enumerate}
\end{enumerate}
\end{proposition}
\begin{proof}
\ref{ek.1} and \ref{ek.2} are exactly (4) and (3) of \cite[Theorem 1.1]{BW01}, and the calculation of geometric intersection number follows from \autoref{ex}. \ref{ek.3} is proven in \cite[Corollary 24.5]{GW}, see also \cite{MP02}. 

Now we move on to \ref{ek.4}. 
The fact that $\{25n-9,25n-\frac{13}{2}\}\subseteq \mathcal{E}_t(K(3,1,n,0))$ follows from \cite[Proposition 5.4 (1) and (11)]{EM02}\footnote{In fact, there is a calculation error in the Seifert invariants in \cite[Proposition 5.4 (1)]{EM02}, but the surgery coefficient is still correct.}, together with the fact that $K(3,1,n,0)$ is the mirror-image of $K(2,-1,1-n,0)$ \cite[Proposition 1.4]{EM96}. In particular, when $n=1$, $K(3,1,n,0)$ is the $(-2,3,7)$-Pretzel knot, and $\mathcal{E}_t(K(3,1,n,0))=\{16,\frac{37}{2},20\}$ strictly contains $\{25n-9,25n-\frac{13}{2}\}=\{16,\frac{37}{2}\}$. Nevertheless, any hyperbolic knot in $S^3$ has at most one non-integral toroidal slope as shown by \cite[Theorem 1]{GWZ}. Thus, for any $n\in \Z\setminus\{0\}$, $\mathcal{E}_t(K(3,1,n,0))$ contains exactly one half-integral element, which is $25n-\frac{13}{2}$ stated above.
\end{proof}

\begin{lemma}\label{LEM: kj}
\begin{enumerate}[label=(\arabic*), leftmargin=*]
\item\label{kj.1} For any $m\in \Z\setminus\{0,1\}$ and $n\in \Z\setminus\{-1,0,1\}$, $\widehat{\pi_1 X_{\mathcal{K}_m}}\not\cong \widehat{\pi_1X_{\mathcal{J}_{n}}}$. 
\item\label{kj.3} For $m,n\in \Z\setminus\{-1,0,1\}$, $\widehat{\pi_1 X_{\mathcal{J}_m}}\cong \widehat{\pi_1 X_{\mathcal{J}_n}}$ if and only if $m=\pm n$; in this case, $\mathcal{J}_m$ and $\mathcal{J}_n$ are either isotopic or the mirror images of each other. 
\item\label{kj.4} For $m,n\in \Z\setminus\{0\}$, $\widehat{\pi_1X_{K(3,1,m,0)}}\cong \widehat{\pi_1X_{K(3,1,n,0)}}$ if and only if $m=n$.
\item\label{kj.2} For $m,n\in \Z\setminus\{0,1\}$, $\widehat{\pi_1 X_{\mathcal{K}_m}}\cong \widehat{\pi_1 X_{\mathcal{K}_n}}$ if and only if $m=n$.
\end{enumerate}
\end{lemma}
\begin{proof}
Recall by \autoref{PROP: Knot exceptional coefficient} that if two hyperbolic knot complements $X_K$ and $X_{K'}$ are profinitely isomorphic, then either $\mathcal{E}_t(K)=\mathcal{E}_t(K')$ or $\mathcal{E}_t(K)=-\mathcal{E}_t(K')$. 

According to \autoref{PROP: coefficient}, $0\in \mathcal{E}_t(\mathcal{K}_m)$ while $0\notin \mathcal{E}_t(\mathcal{J}_n)$. Thus, this implies that $\widehat{X_{\mathcal{K}_m}}\not \cong \widehat{X_{\mathcal{J}_n}}$, which proves \ref{kj.1}.

Similarly, if $\widehat{\pi_1 X_{\mathcal{J}_m}}\cong \widehat{\pi_1 X_{\mathcal{J}_n}}$, then $\{-2-9m,2-9m\}=\mathcal{E}_t(\mathcal{J}_m)=\pm \mathcal{E}_t(\mathcal{J}_n)=\pm \{-2-9n,2-9n\}$ by \autoref{PROP: coefficient}, which implies $m=\pm n$. In addition, $\mathcal{J}_n$ can be viewed as a $\frac{1}{n}$-surgery on the   link $\mathcal{L}$, while $\mathcal{L}$ is achiral, ie isotopic to its mirror image. Thus, the $\frac{1}{n}$-surgery on $\mathcal{L}$ yields the mirror-image of the $-\frac{1}{n}$-surgery on $\mathcal{L}$. This implies that $\mathcal{J}_n$ and $\mathcal{J}_{-n}$ are the mirror images of each other. In particular, $X_{\mathcal{J}_n}$ is homeomorphic to $X_{\mathcal{J}_{-n}}$, so $\widehat{\pi_1 X_{\mathcal{J}_n}}\cong \widehat{\pi_1 X_{\mathcal{J}_{-n}}}$. This finishes the proof of \ref{kj.3}.

For \ref{kj.4}, note that $25n-\frac{13}{2}$ is the unique half-integral element in $\mathcal{E}_t(K(3,1,n,0))$. Thus, if $\widehat{\pi_1X_{K(3,1,m,0)}}\cong \widehat{\pi_1X_{K(3,1,n,0)}}$, the same reasoning implies $25m-\frac{13}{2}=\pm(25n-\frac{13}{2})$. Since $m$ and $n$ are integers, the only possibility is $m=n$.

\ref{kj.2} is difficult to detect from the toroidal slopes; instead, we can detect these knots through the profinite rigidity of Alexander polynomials. Ueki \cite[Theorem 1.1]{Ueki} shows that if two knot complements $X_{K}$ and $X_{K'}$ are profinitely isomorphic, then their Alexander polynomials $\Delta_{K}(t)$ and $\Delta_{K'}(t)$ coincide up to a multiplicative unit in $\Z[t^\pm]$. According to \cite[Chapter 7.B, Exercise 7]{Rolfsen}, the Alexander polynomial of the twist knot $\mathcal{K}_m$ is $\Delta_{\mathcal{K}_m}(t)\doteq mt^2+(1-2m)t+m$. It is easy to verify that for any $m\neq n\in \Z$, $\Delta_{\mathcal{K}_m}(t)$ and $\Delta_{\mathcal{K}_n}(t)$ are not equivalent up to multiplicative unit, and therefore, $\widehat{\pi_1 X_{\mathcal{K}_m}}\not\cong \widehat{\pi_1 X_{\mathcal{K}_n}}$.
\end{proof}

\begin{theorem}\label{Mthm: knot complement}
The knot complements $X_{\mathcal{K}_n}$ ($n\in \Z\setminus \{0,1\}$), $X_{\mathcal{J}_n}$ ($n\in \Z\setminus \{-1,0,1\}$), and $X_{K(3,1,-n,0)}$ ($n\in \Z\setminus\{0\}$) are profinitely rigid among all compact orientable 3-manifolds. 
\end{theorem}
%$\mathcal{K}_1$ is trefoil and $\mathcal{K}_{-1}$ is figure-eight. 等同于Whitehead做-1/n手术
\begin{proof}
For brevity of notation, let $K$ denote one of the above listed knots, ie $K$ belongs to $\mathcal{K}_n$ ($n\in \Z\setminus \{0,1\}$), ${\mathcal{J}_n}$ ($n\in \Z\setminus \{-1,0,1\}$), and ${K(3,1,-n,0)}$ ($n\in \Z\setminus\{0\}$), and we denote $M=X_K$. 

Suppose $N$ is a compact orientable 3-manifold such that $\widehat{\pi_1M}\cong \widehat{\pi_1N}$. Then according to \autoref{PROP: Detect knot complement}, $N\cong X_{K'}$ is also a hyperbolic knot complement. In addition, $\mathcal{E}_t(K')=\pm \mathcal{E}_t(K)$ according to \autoref{PROP: Knot exceptional coefficient}. According to \autoref{ex}, the $\pm$-sign does not affect the geometric intersection number $\Delta(\cdot,\cdot)$, so the possible geometric intersection numbers of toroidal slopes of $K$ and $K'$ coincide.

When $K=\mathcal{K}_{-1}$, \autoref{PROP: coefficient} \ref{ek.1} implies that there exist $c_1,c_2\in \mathcal{E}_t(K)$ such that $\Delta(c_1,c_2)=8$. By the above reasoning, there exist $c_1',c_2'\in \mathcal{E}_t(K')$ such that $\Delta(c_1',c_2')=8$, and $K'$ is equivalent to the figure-eight knot $K_{-1}$ according to \autoref{THM: Knot maximal} \ref{km.3}. Thus, $N\cong M$ is the figure-eight knot complement.

When $K=K(3,1,n,0)$, \autoref{PROP: coefficient} \ref{ek.4} implies that there exist $c_1,c_2\in \mathcal{E}_t(K)$ such that $\Delta(c_1,c_2)=5$. The above reasoning shows that there exist $c_1',c_2'\in \mathcal{E}_t(K')$ such that $\Delta(c_1',c_2')=5$. According to \autoref{THM: Knot maximal} \ref{km.2}, $K'$ is either equivalent or mirror image to a \EM knot $K(3,1,n',0)$ ($n'\in \Z\setminus \{0\}$). Thus, $\widehat{\pi_1M}\cong \widehat{\pi_1N}$ implies $\widehat{\pi_1X_{K(3,1,n,0)}}\cong \widehat{\pi_1X_{K(3,1,n',0)}}$. \autoref{LEM: kj} then guarantees $n=n'$, ie $N\cong X_{K'}\cong X_{K(3,1,n,0)}\cong M$.

When $K=\mathcal{K}_n$  or $K=\mathcal{J}_n$ for some $n\in \Z\setminus\{-1,0,1\}$, \autoref{PROP: coefficient} \ref{ek.2} and \ref{ek.3} imply that there exist $c_1,c_2\in \mathcal{E}_t(K)$ such that $\Delta(c_1,c_2)=4$. As a result, there exist $c_1',c_2'\in \mathcal{E}_t(K')$ such that $\Delta(c_1',c_2')=4$, and \autoref{THM: Knot maximal} \ref{km.1} implies that $K'$ is either  equivalent or mirror image to $\mathcal{K}_{n'}$ ($n'\in \Z\setminus \{0,1\}$) or $\mathcal{J}_{n'}$  ($n'\in \Z\setminus\{-1,0,1\}$). \iffalse In addition, $\widehat{\pi_1X_K}\iffalse\cong \widehat{\pi_1M}\cong \widehat{\pi_1N}\fi\cong \widehat{\pi_1X_{K'}}$.\fi  According to \ref{kj.1}, \ref{kj.3}, \ref{kj.2} of \autoref{LEM: kj}, these candidates can be profinitely distinguished from each other, so $\widehat{\pi_1X_K}\cong \widehat{\pi_1X_{K'}}$ ensures that $K'$ can only be equivalent or mirror image to $K$. Thus, $N\cong X_{K'}\cong X_{K}= M$, finishing the proof.
\end{proof}

\subsection{The Berge manifold \iffalse$W_3^1W_7^{-3}$ knot complement in $D^2\times S^1$\fi}

The complement of the link $L_A$ shown in \autoref{Fig: LA} can also be viewed as the complement of the $W_3^1W_7^{-3}$ braid knot in $D^2\times S^1$, which is conventionally called the Berge manifold. Indeed, this manifold is hyperbolic according to \cite[Theorem 3.2]{Ber91}. In addition, we have the following characterization.

\begin{proposition}[{\cite[Corollary 2.9]{Ber91}}]\label{PROP: d2s1}
Suppose $N$ is a finite-volume hyperbolic 3-manifold with two cusps. There exist three different slopes $c_1,c_2,c_3$ on one boundary component of $N$ such that $N_{c_1}\cong N_{c_2}\cong N_{c_3}\cong D^2\times S^1$ if and only if $N$ is homeomorphic to $\cpl{L_A}$.
\end{proposition}

%\com{Remark hyperbolicity}

\begin{theorem}\label{THM: Berge}
The Berge manifold $\cpl{L_A}$ is profinitely rigid among all compact orientable 3-manifolds.
\end{theorem}
\begin{proof}
Suppose $N$ is a compact orientable 3-manifold such that $\widehat{\pi_1N}\cong \widehat{\pi_1 (S^3\setminus L_A)}$. Since $\cpl{L_A}$ is a finite-volume hyperbolic 3-manifold according to \cite[Theorem 3.2]{Ber91}, \autoref{PROP: Detect hyperbolic}  implies that $N$ is also a finite-volume hyperbolic 3-manifold with two cusps. 

\autoref{PROP: d2s1} shows that there are three different slopes on one boundary component of $\cpl{L_A}$ such that the corresponding Dehn filling is homeomorphic to $D^2\times S^1$. According to \hyperref[PROP: Dehn filling]{Theorem A}, there exist three different slopes $c_1,c_2,c_3$ on one boundary component $\partial_0N$ of $N$ such that $\widehat{\pi_1N_{c_1}}\cong \widehat{\pi_1N_{c_2}}\cong \widehat{\pi_1N_{c_3}}\cong \widehat{\pi_1(D^2\times S^1)}\cong \widehat{\Z}$. Note that the group $\Z$ is profinitely rigid among all finitely generated, residually finite groups. So according to \autoref{PROP: RF}, we derive that $\pi_1N_{c_1}\cong \pi_1N_{c_2}\cong \pi_1N_{c_3}\cong \Z$. For each $i=1,2,3$, $N_{c_i}$ is a compact orientable 3-manifold whose boundary is a torus, and the only such manifold with infinite cyclic fundamental group is $D^2\times S^1$. Thus, $N$ also admits three different Dehn fillings on $\partial_0N$ yielding $D^2\times S^1$, and \autoref{PROP: d2s1} implies that $N\cong \cpl{L_A}$.
\end{proof}

\subsection{Detecting Klein bottles and Dehn surgeries on $\mathcal{L}$ and $L_B$}
We first clarify the notations. Recall that $L_B$ is the link shown in \autoref{Fig: LB}. For $u,r,s\in \Qinf$, let $L_B(u)$ denote $u$-surgery on the $L_B^{(1)}$ component of $L_B$, while the other two components of $L_B$ yield two boundary components of the 3-manifold $L_B(u)$; let $L_B(u,r)$ denote $u$-surgery on the $L_B^{(1)}$ component and $r$-surgery on the $L_B^{(2)}$ component, which is a 3-manifold with one torus boundary; let $L_B(u,r,s)$ denote $u$-surgery on the $L_B^{(1)}$ component, $r$-surgery on the $L_B^{(2)}$ component, and $s$-surgery on the $L_B^{(3)}$ component, which is a closed 3-manifold. 

Note that there is an orientation-preserving homeomorphism of $(S^3,L_B)$ which interchanges the $L_B^{(2)}$ and $L_B^{(3)}$ components. Thus,
\begin{equation}\label{EQU: LB1}
L_B(u,r,s)\cong L_B(u,s,r).
\end{equation}

Recall that $\mathcal{L}$ is the link shown in \autoref{Fig: L}. Similarly, let $\mathcal{L}(r)$ denote $r$-surgery on one component of $\mathcal{L}$, and $\mathcal{L}(r,s)$ denote $r$-surgery on one component and $s$-surgery on the other component of $\mathcal{L}$. Note that $\mathcal{L}$ is achiral, ie equivalent to its mirror image. Thus,
\begin{equation}\label{EQU: L1}
\mathcal{L}(r)\cong \mathcal{L}(-r)\quad\text{and}\quad\mathcal{L}(r,s)\cong \mathcal{L}(-r,-s).
\end{equation}

%\com{To be revised}
\iffalse
For brevity of notation, we denote by $Y$ the 3-manifold obtained by 0-surgery on the red component of link $L_B$ shown in \autoref{Fig: LB}, while the other two components of $L_B$ yield two boundary components of $Y$. In fact, $Y$ is a finite-volume 3-manifold with two cusps according to \cite[Proposition 2.6]{Lee06}. The following proposition characterizes the unique Dehn filling pattern of $Y$ yielding Klein bottles.
\fi

Some of the Dehn surgeries on $\mathcal{L}$ and $L_B$ can be characterized by their exceptional Dehn fillings, as shown in  the following proposition.

\begin{proposition}[{\cite[Theorem 1.1]{Lee06}}]\label{PROP: Klein Detects}
Let $N$ be a cusped finite-volume hyperbolic 3-manifold, and  let $\partial_0N$ be a boundary component of $N$. 
Suppose \iffalse$\max\{\Delta(c_1,c_2)\mid c_1,c_2\in \mathcal{E}_k(N,\partial_0N)\}=4$\fi there exist $c_1,c_2\in \slope(\partial_0N)$ with $\Delta(c_1,c_2)=4$, such that both $N_{c_1}$ and $N_{c_2}$ contain an embedded Klein bottle.
\begin{enumerate}[label=(\arabic*), leftmargin=*]
\item If $N$ has exactly one cusp, then $N$ is either homeomorphic to $\mathcal{L}(n-\frac{1}{2})$ for some $n\in \Z$, or homeomorphic to $L_B(0,r)$ for some $r\in \mathbb{Q}\setminus\{0,4\}$. 
\item If $N$ has more than one cusp, then $N$ is homeomorphic to $L_B(0)$.
\end{enumerate}
\end{proposition}
Since $\infty, -2,-1,0,1,2$ are all the exceptional slopes on one component of $\mathcal{L}$ according to \cite{MP02}, $\mathcal{L}(n-\frac{1}{2})$ is hyperbolic for all $n\in \Z$. Furthermore, \cite[Proposition 2.6]{Lee06} showed that $L_B(0)$ is hyperbolic, and \cite[Proposition 2.7]{Lee06} showed that $L_B(0,r)$ is hyperbolic for any $r\in \mathbb{Q}\setminus\{0,4\}$.

In the following lemma, we list in details the boundary slopes of the manifolds in  \autoref{PROP: Klein Detects} yielding Klein bottles that achieves the geometric intersection number $4$, and specify the corresponding Dehn filled 3-manifolds. 

 Recall the notations that $\mathcal{N}$ is the orientable $I$-bundle over Klein bottle, and $M_1$ is the complement of the Whitehead link $\mathcal{W}$ in $S^3$. In the following context, we denote $M_{3_1}$ as the complement of the trefoil in $S^3$, ie $M_{3_1}$ is a Seifert fibered space $(D^2;(2,1),(3,1))$.

\begin{lemma}\label{LEM: Aspherical slope}
\begin{enumerate}[label=(\arabic*), leftmargin=*]
\item\label{5.17-1} $L_B(0,0)\cong L_B(0,4)$ can be decomposed as $\mathcal{N}\cup_{T^2} M_1$, ie gluing $\mathcal{N}$ and $M_1$ along a torus boundary.
\item\label{5.17-2} For each $r\in \mathbb{Q}$, $L_B(0,r,0)$ and $L_B(0,r,4)$ contain embedded Klein bottles. They can be decomposed as $\mathcal{N}\cup_{T^2} M'$, where $\partial M'=T^2$ and $M'$ is either a hyperbolic manifold, a Seifert fibered space with hyperbolic base orbifold, or a graph manifold with non-trivial JSJ-decomposition.
\item\label{5.17-3} $\mathcal{L}(\frac{3}{2},-2)=(\mathbb{RP}^2;(2,1),(3,1),-1)$ and $\mathcal{L}(n-\frac{1}{2},-2)=\mathcal{N}\cup_{T^2}M_{3_1}$ for $n\in \Z\setminus\{2\}$.
\item\label{5.17-4} $\mathcal{L}(-\frac{3}{2},2)=(\mathbb{RP}^2;(2,1),(3,1),-1)$ and $\mathcal{L}(n-\frac{1}{2},2)=\mathcal{N}\cup_{T^2}M_{3_1}$ for $n\in \Z\setminus\{-1\}$.
\end{enumerate}

In particular, all the above listed Dehn-filled manifolds are aspherical, and contain an embedded Klein bottle.
\end{lemma}
\begin{proof}
\ref{5.17-1} is proved in \cite[Lemma 2.3 (2)]{Lee06}.

\ref{5.17-2} According to (\ref{EQU: LB1}), $L_B(0,r,0)\cong L_B(0,0,r)$ and $L_B(0,r,4)\cong L_B(0,4,r)$. Thus, they can be viewed as Dehn fillings on $L_B(0,0)$ and $L_B(0,4)$. According to \ref{5.17-1}, $L_B(0,0,r)$ and $L_B(0,4,r)$ can be decomposed by an embedded torus into $\mathcal{N}\cup M'$, where $M'$ is a Dehn filling of $M_1$ on one of its boundary components $\partial_1M_1$. The Dehn fillings of $M_1$ on  $\partial_1M_1$ have been classified by \cite[Table A.1]{MP02}. In fact, there is exactly one slope yielding the solid torus $D^2\times S^1$, while the other slopes yield either a hyperbolic manifold, a Seifert fibered space with hyperbolic base orbifold, or a graph manifold with non-trivial JSJ-decomposition. In other words, there is only one element $r_0\in \Qinf$ such that $L_B(0,0,r_0)\cong \mathcal{N}\cup_{T^2} (D^2\times S^1)$, which is indeed non-aspherical; while for any $r\neq r_0\in \Qinf$, $L_B(0,0,r)\cong \mathcal{N}\cup_{T^2}M'$ is aspherical. We claim that $r_0=\infty$. In fact, by (\ref{EQU: LB1}), $L_B(0,0,\infty)\cong L_B(0,\infty,0)$ is a Dehn filling on $L_B(0,\infty)$. \cite[Lemma 2.3 (1)]{Lee06} showed that $L_B(0,\infty)\cong \mathcal{N}$, and thus $L_B(0,\infty,0)\cong \mathcal{N}\cup_{T^2} (D^2\times S^1)$ is non-aspherical. Similarly, there is only one element $r_4\in \Qinf$ such that $L_B(0,4,r_4)\cong \mathcal{N}\cup_{T^2} (D^2\times S^1)$ is non-aspherical, while for any $r\neq r_4\in \Qinf$, $L_B(0,4,r_4)\cong \mathcal{N}\cup_{T^2}M'$ is aspherical. One can similarly show that $r_4=\infty$ by \cite[Lemma 2.3 (1)]{Lee06}. Thus, this proves \ref{5.17-2}. 

\ref{5.17-3} follows from \cite[Table 3]{MP02}. In fact, the link $\mathcal{L}$ is obtained by a $-\frac{1}{2}$-surgery on one component of the three-chain link (\autoref{Fig: 3chain}) in $S^3$. Thus, for any $r,s\in \Qinf$, $\mathcal{L}(r,s)$ can be viewed as $(-\frac{1}{2},r-2,s-2)$-surgery on three components of the three-chain link, for which \cite{MP02} presents a complete classification of exceptional Dehn surgeries.

\begin{figure}[h]
\centering
\includegraphics[width=3.5cm]{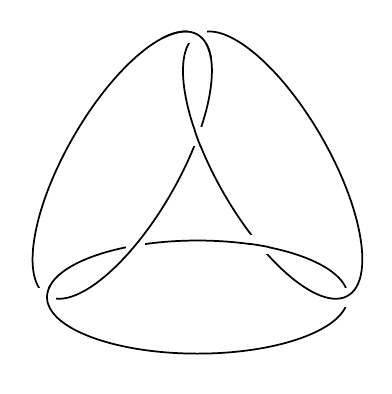}
\caption{The three-chain}
\label{Fig: 3chain}
\end{figure}

\ref{5.17-4} follows directly from \ref{5.17-3} together with the achirality of $\mathcal{L}$ (\ref{EQU: L1}).

Note that the above listed Dehn-filled manifolds are either irreducible 3-manifolds admitting non-trivial JSJ-decomposition, or Seifert fibered spaces with hyperbolic base orbifolds. Thus, these  Dehn-filled manifolds are all aspherical. They obviously contain embedded Klein bottles, as they satisfy \autoref{LEM: Klein} \ref{4.6-2}.
\end{proof}

Recall that in the definition of the subset $\mathcal{E}_k(M,\partial_{i_0}M)$ of exceptional slopes, we require the Dehn-filled manifold which contains Klein bottles to be aspherical. Therefore, by \autoref{LEM: Aspherical slope}, we can obtain the following corollary.
\begin{corollary}\label{COR: Klein slopes}
We follow the $\Qinf$-parametrization of boundary slopes of the link complements $\cpl{\mathcal{L}}$ and $\cpl{L_B}$ as  the parametrization for the boundary slopes of $L_B(0)$, $L_B(0,r)$ and $\mathcal{L}(n-\frac{1}{2})$. Then,
\begin{enumerate}[label=(\arabic*), leftmargin=*]
\item\label{5.18-1} $\{0,4\}\subseteq \mathcal{E}_k\left(L_B(0),\partial N(L_B^{(2)})\right)$, and $\Delta(0,4)=4$;
\item\label{5.18-2}  $\{0,4\}\subseteq \mathcal{E}_k\left(L_B(0,r),\partial N(L_B^{(3)})\right)$ for each $r\in \mathbb{Q}\setminus\{0,4\}$, and $\Delta(0,4)=4$;
\item\label{5.18-3} $\{-2,2\}\subseteq \mathcal{E}_k\left(\mathcal{L}(n-\frac{1}{2}), \partial (\mathcal{L}(n-\frac{1}{2}))\right)$ for each $n\in \Z$, and $\Delta(-2,2)=4$.
\end{enumerate}
\end{corollary}

With the asphericity of the Dehn-filled manifolds verified, we can deduce profinite rigidity of these manifolds from \autoref{PROP: Exceptional} \ref{4.3-4}.

\begin{theorem}\label{THM: LB0}
$L_B(0)$ is profinitely rigid among all compact, orientable 3-manifolds.
\end{theorem}
\begin{proof}
Suppose $N$ is a compact orientable 3-manifold such that $\widehat{\pi_1L_B(0)}\cong \widehat{\pi_1N}$. Recall that $L_B(0)$ is finite-volume hyperbolic by \cite[Proposition 2.6]{Lee06}. Thus, $N$ is also finite-volume hyperbolic with two cusps according to \autoref{PROP: Detect hyperbolic}.

By \autoref{COR: Klein slopes} \ref{5.18-1}, there exist two slopes $c_1,c_2\in \mathcal{E}_k(L_B(0),\partial(L_B(0)))$ such that $\Delta(c_1,c_2)=4$. Then,   \autoref{PROP: Exceptional} together with \autoref{LEM: preserve} implies that there exist two slopes $c_1',c_2'\in \mathcal{E}_k(N,\partial_{j}N)$ such that $\Delta(c_1',c_2')=4$. Since $N$ has two cusps, \autoref{PROP: Klein Detects} implies that $N$ is homeomorphic to $L_B(0)$.
\end{proof}

To prove the profinite rigidity of $\mathcal{L}(n-\frac{1}{2})$ and $L_B(0,r)$, we have to firstly distinguish these manifolds through the profinite completion of their fundamental groups. Let us first make some preparations.

\begin{lemma}\label{LEM: Max}
Let $M$ be one of $\mathcal{L}(n-\frac{1}{2})$ ($n\in \Z$) or $L_B(0,r)$ ($r\in \mathbb{Q}\setminus\{0,4\}$). Then $\max\{\Delta(c_1,c_2)\mid c_1,c_2\in \mathcal{E}_k(M,\partial M)\}=4$.
\end{lemma}
\begin{proof}
We have shown by \autoref{COR: Klein slopes} that $\max\{\Delta(c_1,c_2)\mid c_1,c_2\in \mathcal{E}_k(M,\partial M)\}\ge 4$. If $\max\{\Delta(c_1,c_2)\mid c_1,c_2\in \mathcal{E}_k(M,\partial M)\}> 4$, then \cite[Corollary 1.5]{Lee07} shows that $M$ is homeomorphic to one of $\mathcal{W}(5),\mathcal{W}(-1),\mathcal{W}(-2),\mathcal{W}(-4)$\iffalse\footnote{The Dehn-surgery coefficient in \cite{Lee07} differs from the one we use by a $\pm$-sign.}\fi. We verify that $M$ is not homeomorphic to $\mathcal{W}(5),\mathcal{W}(-1),\mathcal{W}(-2),\mathcal{W}(-4)$. 

In fact, according to \cite[Table A.2]{MP02} for $\mathcal{W}(5),\mathcal{W}(-1),\mathcal{W}(-2)$ and \cite[Table A.4]{MP02} for $\mathcal{W}(-4)$:
\begin{enumerate}[label=(\roman*), leftmargin=*]
\item $\mathcal{E}_k(\mathcal{W}(5),\partial\mathcal{W}(5))=\{c_5,c_5'\}$, and $\Delta(c_5,c_5')=8$;
\item $\mathcal{E}_k(\mathcal{W}(-1),\partial\mathcal{W}(-1))=\{c_{-1},c_{-1}'\}$, and $\Delta(c_{-1},c_{-1}')=8$;
\item $\mathcal{E}_k(\mathcal{W}(-2),\partial\mathcal{W}(-2))=\{c_{-2},c_{-2}'\}$, and $\Delta(c_{-2},c_{-2}')=6$;
\item $\mathcal{E}_k(\mathcal{W}(-4),\partial\mathcal{W}(-4))=\{c_{-4},c_{-4}'\}$, and $\Delta(c_{-4},c_{-4}')=5$.
\end{enumerate}
None of them contain two slopes with geometric intersection number $4$. Therefore, by \autoref{COR: Klein slopes}, $M$ is not homeomorphic to $\mathcal{W}(5),\mathcal{W}(-1),\mathcal{W}(-2),\mathcal{W}(-4)$, finishing the proof.
\end{proof}

\begin{lemma}\label{LEM: slopes distance 4}
Let $T$ and $T'$ be two tori.
\begin{enumerate}[label=(\arabic*), leftmargin=*]
\item\label{5.21-1} Suppose $E\subseteq \slope(T)$ is a subset such that $\max\{\Delta(c_1,c_2)\mid c_1,c_2\in E\}=4$. Then the pair $c_1,c_2\in E$ realizing the maximal geometric intersection number $\Delta(c_1,c_2)=4$ is unique.
\item\label{5.21-2} Let $c_1,c_2\in\slope(T)$ and $c_1',c_2'\in \slope(T')$ such that $\Delta(c_1,c_2)=\Delta(c_1',c_2')=4$. 
Then there exist exactly two (equivalence classes of) isomorphisms $\psi_1,\psi_2\in \mathrm{Isom}_{\Z}(\pi_1(T),\pi_1(T'))/(\psi\sim-\psi) \approx  \mathrm{PGL}(2,\Z)$ such that $\pisi_i(\{c_1,c_2\})=\{c_1',c_2'\}$.
\end{enumerate}
\end{lemma}
\begin{proof}
\ref{5.21-1}: Suppose $c_1,c_2\in E$ such that $\Delta(c_1,c_2)=4$. We can choose a basis $(\mu_0,\lambda_0)$ for $\pi_1(T)\cong \Z^2$, so that $c_1=[\mu_0]$. Then, $c_2=[a\mu_0+4\lambda_0]$ for some odd integer $a$. We can assume $a=4k+t$, where $k\in \Z$ and $t\in \{\pm1\}$. Let $\mu=t\mu_0$ and $\lambda=k\mu_0+\lambda_0$. Then $(\mu,\lambda)$ is also a basis for $\pi_1(T)$, and $c_1=[\mu]$, $c_2=[\mu+4\lambda]$. 

 We can now list all possible elements in $E$. Suppose $c=[p\mu+q\lambda]\in E$, where $(p,q)$ is a pair of coprime integers. Without loss of generality, we may assume $p\ge 0$. Then $\Delta(c_1,c)=\abs{q}\le 4$, and $\Delta(c_2,c)=\abs{4p-q}\le 4$. 
It is easy to verify that $(p,q)\in\{ (0,\pm1),(1,0),(1,1),(1,2),(1,3),(1,4)\}$. Thus, $E\subseteq \{[\mu],[\mu+\lambda],[\mu+2\lambda],[\mu+3\lambda],[\mu+4\lambda],[\lambda]\}$, and it is easy to verify that $[\mu]$ and $[\mu+4\lambda]$ is the unique pair realizing geometric intersection number $4$.

\ref{5.21-2}: According to \ref{5.21-1}, we can choose basis $(\mu,\lambda)$ for $\pi_1(T)$, and $(\mu',\lambda')$ for $\pi_1(T')$, such that $c_1=[\mu]$, $c_2=[\mu+4\lambda]$, $c_1'=[\mu']$, and $c_2'=[\mu'+4\lambda']$. Suppose $\psi:\pi_1(T)\to \pi_1(T')$ is an isomorphism such that  $\pisi(\{c_1,c_2\})=\{c_1',c_2'\}$. There are two cases to be considered.

In the first case, $\pisi(c_1)=c_1'$ and $\pisi(c_2)=c_2'$. The only possibilities are $\psi(\mu,\mu+4\lambda)=\pm (\mu',\mu'+4\lambda')$, since otherwise $4\psi(\lambda)=\pm(2\mu'+4\lambda')$, which is impossible. Thus, $\psi(\mu,\lambda)=\pm(\mu',\lambda')$.

In the second case, $\pisi(c_1)=c_2'$ and $\pisi(c_2)=c_1'$. The only possibilities are $\psi(\mu,\mu+4\lambda)=\pm(\mu'+4\lambda',\mu')$, since otherwise we again obtain $4\psi(\lambda)=\pm(2\mu'+4\lambda')$, which is impossible. Thus, $\psi(\mu,\lambda)=\pm(\mu'+4\lambda',-\lambda')$.

Therefore, there are altogether two isomorphisms $\psi_1,\psi_2\in \mathrm{Isom}_{\Z}(\pi_1(T),\pi_1(T'))/(\psi\sim-\psi)$ such that $\pisi_i(\{c_1,c_2\})=\{c_1',c_2'\}$.
\end{proof}

\begin{lemma}\label{Lem: Distinguish K}
Suppose $M$ and $N$ belong to $\mathcal{L}(n-\frac{1}{2})$ ($n\in \Z$) or $L_B(0,r)$ ($r\in \mathbb{Q}\setminus\{0,4\}$). Then $\widehat{\pi_1M}\cong \widehat{\pi_1N}$ if and only if $M\cong N$.
\end{lemma}
\begin{proof}
We first prove that $\mathcal{L}(n-\frac{1}{2})$ is not profinitely isomorphic to $L_B(0,r)$, by showing that their first homologies are not isomorphic.  
In fact, up to an appropriate choice of orientations, the two components of $\mathcal{L}$ have linking number $3$. Let $\alpha,\beta$ be two generators of $H_1(\cpl{\mathcal{L}};\Z)\cong \Z^2$ corresponding to the meridians of the two components. Then, $$H_1( \mathcal{L}(n-\frac{1}{2});\Z)\cong \frac{H_1(\cpl{\mathcal{L}};\Z)}{\left\langle (2n-1)\alpha+6\beta\right\rangle}\cong \left\{ \begin{aligned}&\Z\oplus \Z/3\Z , & 3\mid 2n-1\\& \Z, & 3\nmid 2n-1 \end{aligned}\right. .$$
Similarly,  up to an appropriate choice of orientations, the three components $L_B^{(1)},L_B^{(2)},L_B^{(3)}$ have pairwise linking number $2$. Let $\alpha,\beta,\gamma$ be three generators of $H_1(\cpl{L_B})\cong \Z^3$  corresponding to the meridians of $L_B^{(1)},L_B^{(2)},L_B^{(3)}$ respectively. For $r=\frac{p}{q}\in \mathbb{Q}$ a reduced fraction, 
\begin{equation}\label{EQU: homology}
H_1(L_B(0,r);\Z)\cong \frac{H_1(\cpl{L_B};\Z)}{\left\langle 2\beta+2\gamma, 2q\alpha+p\beta+2q\gamma\right\rangle} \cong \left\{ \begin{aligned}&\Z\oplus \Z/2\Z\oplus \Z/2\Z , & 2\mid p\\& \Z\oplus \Z/2\Z, & 2\nmid p \end{aligned}\right. .
\end{equation}
Therefore, $H_1( \mathcal{L}(n-\frac{1}{2});\Z)\not\cong H_1(L_B(0,r);\Z)$, which implies $\widehat{\pi_1\mathcal{L}(n-\frac{1}{2})}\not \cong \widehat{\pi_1L_B(0,r)}$ by \autoref{lem: abelianize}.

Now we consider the family $\mathcal{L}(n-\frac{1}{2})$ ($n\in \Z$). Let $M=\mathcal{L}(m-\frac{1}{2})$ and $N=\mathcal{L}(n-\frac{1}{2})$, and suppose $\widehat{\pi_1M}\cong \widehat{\pi_1N}$. Then according to \hyperref[PROP: Dehn filling]{Theorem A}, there exists an isomorphism $\psi:\pi_1\partial M\to \pi_1\partial N$ such that $\widehat{\pi_1M_c}\cong \widehat{\pi_1N_{\pisi(c)}}$ for all $c\in \slope(\partial M)$.  In particular, $\pisi(\mathcal{E}_k(M,\partial M))=\mathcal{E}_k(N,\partial N)$ according to \autoref{PROP: Exceptional}. We follow the $\Qinf$-parametrization for $\slope(\partial M)$ and $\slope(\partial N)$ as above. 
Recall that $\max\{\Delta(c_1,c_2)\mid c_1,c_2\in \mathcal{E}_k(M,\partial M)\}=\max\{\Delta(c_1,c_2)\mid c_1,c_2\in \mathcal{E}_k(N,\partial N)\}=4$ by \autoref{LEM: Max}. Thus, according to \autoref{COR: Klein slopes} \ref{5.18-3} and \autoref{LEM: slopes distance 4} \ref{5.21-1}, $\{-2 ,2\}$ is the unique pair of elements in $\mathcal{E}_k(M,\partial M)$, as well as in $\mathcal{E}_k(N,\partial N)$, that realizes the maximal geometric intersection number $4$. Since $\pisi: \mathcal{E}_k(M,\partial M)\xrightarrow{1:1}\mathcal{E}_k(N,\partial N)$ preserves the geometric intersection number by \autoref{LEM: preserve}, we derive that $\pisi(\{-2,2\})=\{-2,2\}$. According to \autoref{LEM: slopes distance 4} \ref{5.21-2}, there are only two possible candidates for $\psi$ in $\mathrm{Isom}_{\Z}(\pi_1\partial M,\pi_1\partial N)/(\psi\sim -\psi)$. Indeed, under the $\Qinf$-parametrization, either $\pisi(s)=s$, or $\pisi (s)=-s$ for all $s\in  \Qinf$. 

In addition, $\pisi$ sends the rational null-homologous boundary slope of $M$ to the rational null-homologous boundary slope of $N$ by \autoref{LEM: null-homologous slope} \ref{5.22-2}. Since the linking number of the two components of $\mathcal{L}$ is $3$,  the rational null-homologous boundary slope of $M$ corresponds to $\frac{3^2}{m-\frac{1}{2}}=\frac{18}{2m-1}$ in the $\Qinf$-parametrization. Similarly, the rational null-homologous boundary slope of $N$ corresponds to $\frac{3^2}{m-\frac{1}{2}}=\frac{18}{2n-1}$. Thus, $\pisi(\frac{18}{2m-1})=\frac{18}{2n-1}$. Recall that $\pisi$ is a multiplication by $\pm1$ under the $\Qinf$-parametrization. Thus, $2m-1=\pm(2n-1)$. Note that $\mathcal{L}(r)\cong \mathcal{L}(-r)$ according to (\ref{EQU: L1}). Thus, regardless of the $\pm$-sign, $M=\mathcal{L}(\frac{2m-1}{2})\cong N=\mathcal{L}(\frac{2n-1}{2})$.

Next, we consider the family $L_B(0,r)$ ($r\in \mathbb{Q}\setminus\{0,4\}$). Let $M=L_B(0,r)$ and $N=L_B(0,r')$ and suppose $\widehat{\pi_1M}\cong \widehat{\pi_1N}$. Similarly, by \hyperref[PROP: Dehn filling]{Theorem A}, there exists an isomorphism $\psi:\pi_1\partial M\to \pi_1\partial N$ such that $\widehat{\pi_1M_c}\cong \widehat{\pi_1N_{\pisi(c)}}$, and $\pisi(\mathcal{E}_k(M,\partial M))=\mathcal{E}_k(N,\partial N)$ according to \autoref{PROP: Exceptional}. 
Again by \autoref{LEM: Max}, \autoref{COR: Klein slopes} \ref{5.18-2} and \autoref{LEM: slopes distance 4} \ref{5.21-1}, $\{0 ,4\}$ is the only pair of elements in $\mathcal{E}_k(M,\partial M)$, as well as in $\mathcal{E}_k(N,\partial N)$, that realizes the maximal geometric intersection number $4$, and \autoref{LEM: preserve} implies that $\pisi(\{0,4\})=\{0,4\}$. According to \autoref{LEM: slopes distance 4} \ref{5.21-2}, there are only two possible candidates for $\psi$ up to $\pm$-sign; under the $\Qinf$-parametrization, either $\pisi(u)=u$, or $\pisi (u)=4-u$ for all $u\in  \Qinf$. 

We still consider the rational null-homologous boundary slopes. Let $r=\frac{p}{q}$ be a reduced fraction. Then,  (\ref{EQU: homology}) implies that $2q\alpha+2q\beta+(4q-p)\gamma$ is torsion in $H_1(L_B(0,r);\Z)$. Thus, the slope represented by $\frac{4q-p}{q}=4-r$ on $\partial N(L_B^{(3)})$ is the rational null-homologous boundary slope on $\partial M$. Similarly, $4-s$ is the rational null-homologous boundary slope on $\partial N$. By \autoref{LEM: null-homologous slope}~\ref{5.22-2}, $\pisi(4-r)=4-s$. Recall that either $\pisi(u)=u$, or $\pisi (u)=4-u$. Thus, either $s=r$ or $s=4-r$. Fortunately, \cite[Lemma 2.2 (2)]{Lee06} proved that $L_B(0,r)\cong L_B(0,4-r)$. Thus, in both cases, $M\cong N$. 
\end{proof}

Finally, we are ready to prove the profinite rigidity of these two families.
\begin{theorem}\label{THM: LBr}
Let $M$ be one of $\mathcal{L}(n-\frac{1}{2})$ ($n\in \Z$) or $L_B(0,r)$ ($r\in \mathbb{Q}\setminus\{0,4\}$). Then $M$ is profinitely rigid among all compact, orientable 3-manifolds.
\end{theorem}
\begin{proof}
Suppose $N$ is a compact, orientable 3-manifold such that $\widehat{\pi_1M}\cong \widehat{\pi_1N}$. Recall that $M$ is hyperbolic by \cite{MP02} and \cite[Proposition 2.7]{Lee06}. Then according to \autoref{PROP: Detect hyperbolic}, $N$ is also finite-volume hyperbolic with exactly one cusp. 

By \autoref{COR: Klein slopes}, there exist two slopes $c_1,c_2\in \mathcal{E}_k(M,\partial M)$ such that $\Delta(c_1,c_2)=4$. Then,   \autoref{PROP: Exceptional} together with \autoref{LEM: preserve} implies that there exist two slopes $c_1',c_2'\in \mathcal{E}_k(N,\partial_{j}N)$ such that $\Delta(c_1',c_2')=4$. Since $N$ has exactly one cusp, \autoref{PROP: Klein Detects} implies that $N$ is either homeomorphic to $\mathcal{L}(n-\frac{1}{2})$ for some $n\in \Z$, or to $L_B(0,r)$ for some $r\in \mathbb{Q}\setminus\{0,4\}$. It then follows from \autoref{Lem: Distinguish K} that $N\cong M$, finishing the proof.
\end{proof}

\appendix
\section{Proof of \autoref{THM: Mixed Peripheral regular}}\label{APP}

\begin{lemma}\label{LEM: irred}
Let $M$ be a compact, orientable, irreducible and boundary-incompressible 3-manifold. Suppose $N$ is a compact, orientable 3-manifold such that $\widehat{\pi_1M}\cong \widehat{\pi_1N}$. Then, 
\begin{enumerate}[label=(\arabic*), leftmargin=*]
\item\label{A11} $N$ is also irreducible and boundary-incompressible;
\item\label{A12} if $M$ is closed, then $N$ is also closed;
\item\label{A13} if $\partial M$ consists of tori, then so does $\partial N$.
\end{enumerate}
\end{lemma}
\begin{proof}
%Similar to the proof of \autoref{PROP: boundary}, we can attach compression-bodies to the compressible boundary components of $N$ and obtain a boundary-incompressible 3-manifold $N^+$ such that $\pi_1N\cong \pi_1N^+$. Then $\widehat{\pi_1M}\cong \widehat{\pi_1N^+}$.
As a standard application of the loop theorem, through attaching 2-handles to the compressible boundary components of $N$ and then capping off the emerging boundary spheres by 3-handles, we can obtain a boundary-incompressible 3-manifold $N^+$ such that $\pi_1N\cong \pi_1N^+$. Then, $\widehat{\pi_1M}\cong \widehat{\pi_1N^+}$.

\cite[Theorem 6.22]{Wil19} showed that the profinite completion of fundamental group determines the prime decomposition of a compact, orientable, boundary-incompressible 3-manifold. In particular, when $M$ is irreducible, $N^+$ is also irreducible.

Note that when $N$ is boundary-compressible,  attaching a 2-handle along a compressible boundary curve will create an essential sphere, which makes $N^+$ reducible. Thus, $N$ is boundary-incompressible and $N= N^+$. Therefore, $N$ is also irreducible, which proves  \ref{A11}. 
\ref{A12} and \ref{A13} then follow from \cite[Corollary 4.2]{Xu}.
\end{proof}

Before moving on, let us briefly recall some useful concepts related to  JSJ-decomposition.

Let $M$ be a compact, orientable, irreducible 3-manifold with empty or incompressible toroidal boundary. Then there exists a minimal collection $\mathcal{T}$ of disjoint essential tori, such that each connected component of $M\setminus \mathcal{T}$ is either a Seifert fibered manifold or a finite-volume hyperbolic manifold. In addition, $\mathcal{T}$ is unique up to isotopy, and is called the {\em JSJ-tori} of $M$.

The JSJ-decomposition endows $M$ with a structure of a graph of spaces, so that it splits $\pi_1M$ as the fundamental group of a graph of groups: $\pi_1M=\pi_1(\mathcal{G},\Gamma_M)$. Here, $\Gamma_M$ denotes the {\em JSJ-graph} of $M$, consisting of vertices $V(\Gamma_M)$ and edges $E(\Gamma_M)$; for each $v\in V(\Gamma_M)$, $\mathcal{G}_v$ is the fundamental group of a JSJ-piece (either hyperbolic or Seifert), and for each $e\in E(\Gamma_M)$, $\mathcal{G}_e=\pi_1(T^2)\cong \Z\oplus\Z$. 
\autoref{FIG: JSJ-graph of groups} shows an example for this construction and we refer the readers to \cite{SW79} for details.

\begin{figure}[h]
\subfigure[The irreducible 3-manifold $M$]{
\includegraphics[width=4.3cm]{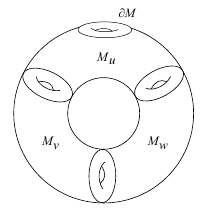}
}
\hspace{2mm}
\subfigure[The JSJ-graph $\Gamma_M$]{
\begin{tikzpicture}[x=0.75pt,y=0.75pt,yscale=-1,xscale=1]
%uncomment if require: \path (0,487); %set diagram left start at 0, and has height of 487
\path (200,220);
%Straight Lines [id:da9494050269225234] 
\draw    (239.94,106.45) -- (197.03,174.46) ;
\draw [shift={(197.03,174.46)}, rotate = 122.25] [color={rgb, 255:red, 0; green, 0; blue, 0 }  ][fill={rgb, 255:red, 0; green, 0; blue, 0 }  ][line width=0.75]      (0, 0) circle [x radius= 3.35, y radius= 3.35]   ;
%\draw [shift={(222.22,134.54)}, rotate = 122.25] [color={rgb, 255:red, 0; green, 0; blue, 0 }  ][line width=0.75]    (10.93,-3.29) .. controls (6.95,-1.4) and (3.31,-0.3) .. (0,0) .. controls (3.31,0.3) and (6.95,1.4) .. (10.93,3.29)   ;
\draw [shift={(239.94,106.45)}, rotate = 122.25] [color={rgb, 255:red, 0; green, 0; blue, 0 }  ][fill={rgb, 255:red, 0; green, 0; blue, 0 }  ][line width=0.75]      (0, 0) circle [x radius= 3.35, y radius= 3.35]   ;
%Straight Lines [id:da9360168239305302] 
\draw    (239.94,106.45) -- (282.53,174.21) ;
\draw [shift={(282.53,174.21)}, rotate = 57.85] [color={rgb, 255:red, 0; green, 0; blue, 0 }  ][fill={rgb, 255:red, 0; green, 0; blue, 0 }  ][line width=0.75]      (0, 0) circle [x radius= 3.35, y radius= 3.35]   ;
%\draw [shift={(257.51,134.4)}, rotate = 57.85] [color={rgb, 255:red, 0; green, 0; blue, 0 }  ][line width=0.75]    (10.93,-3.29) .. controls (6.95,-1.4) and (3.31,-0.3) .. (0,0) .. controls (3.31,0.3) and (6.95,1.4) .. (10.93,3.29)   ;
\draw [shift={(239.94,106.45)}, rotate = 57.85] [color={rgb, 255:red, 0; green, 0; blue, 0 }  ][fill={rgb, 255:red, 0; green, 0; blue, 0 }  ][line width=0.75]      (0, 0) circle [x radius= 3.35, y radius= 3.35]   ;
%Straight Lines [id:da06535769075298381] 
\draw    (197.03,174.46) -- (282.53,174.21) ;
\draw [shift={(282.53,174.21)}, rotate = 359.83] [color={rgb, 255:red, 0; green, 0; blue, 0 }  ][fill={rgb, 255:red, 0; green, 0; blue, 0 }  ][line width=0.75]      (0, 0) circle [x radius= 3.35, y radius= 3.35]   ;
%\draw [shift={(232.78,174.36)}, rotate = 359.83] [color={rgb, 255:red, 0; green, 0; blue, 0 }  ][line width=0.75]    (10.93,-3.29) .. controls (6.95,-1.4) and (3.31,-0.3) .. (0,0) .. controls (3.31,0.3) and (6.95,1.4) .. (10.93,3.29)   ;
\draw [shift={(197.03,174.46)}, rotate = 359.83] [color={rgb, 255:red, 0; green, 0; blue, 0 }  ][fill={rgb, 255:red, 0; green, 0; blue, 0 }  ][line width=0.75]      (0, 0) circle [x radius= 3.35, y radius= 3.35]   ;

% Text Node
\draw (235.09,90) node [anchor=north west][inner sep=0.75pt]   [align=left] {$\displaystyle u$};
% Text Node
\draw (184.87,178) node [anchor=north west][inner sep=0.75pt]   [align=left] {$\displaystyle v$};
% Text Node
\draw (284.53,177.21) node [anchor=north west][inner sep=0.75pt]   [align=left] {$\displaystyle w$};

\end{tikzpicture}
}
\hspace{2mm}
\subfigure[The graph of groups $(\mathcal{G}_M,\Gamma_M)$]{
\begin{tikzpicture}[x=0.75pt,y=0.75pt,yscale=-0.7,xscale=0.7]
%uncomment if require: \path (0,487); %set diagram left start at 0, and has height of 487
\path(200,290);
%Straight Lines [id:da22887863313280232] 
\draw    (230.16,159.79) -- (264.3,125.66) ;
\draw [shift={(265.71,124.25)}, rotate = 135] [color={rgb, 255:red, 0; green, 0; blue, 0 }  ][line width=0.75]    (10.93,-3.29) .. controls (6.95,-1.4) and (3.31,-0.3) .. (0,0) .. controls (3.31,0.3) and (6.95,1.4) .. (10.93,3.29)   ;
%Straight Lines [id:da2781552880180642] 
\draw    (197.35,196.25) -- (160.48,233.12) ;
\draw [shift={(159.07,234.54)}, rotate = 315] [color={rgb, 255:red, 0; green, 0; blue, 0 }  ][line width=0.75]    (10.93,-3.29) .. controls (6.95,-1.4) and (3.31,-0.3) .. (0,0) .. controls (3.31,0.3) and (6.95,1.4) .. (10.93,3.29)   ;
%Straight Lines [id:da8617196833864866] 
\draw    (323.23,159.79) -- (289,125.57) ;
\draw [shift={(287.59,124.15)}, rotate = 45] [color={rgb, 255:red, 0; green, 0; blue, 0 }  ][line width=0.75]    (10.93,-3.29) .. controls (6.95,-1.4) and (3.31,-0.3) .. (0,0) .. controls (3.31,0.3) and (6.95,1.4) .. (10.93,3.29)   ;
%Straight Lines [id:da3849095074213411] 
\draw    (348.2,192.97) -- (390.17,234.95) ;
\draw [shift={(391.59,236.36)}, rotate = 225] [color={rgb, 255:red, 0; green, 0; blue, 0 }  ][line width=0.75]    (10.93,-3.29) .. controls (6.95,-1.4) and (3.31,-0.3) .. (0,0) .. controls (3.31,0.3) and (6.95,1.4) .. (10.93,3.29)   ;
%Straight Lines [id:da8780829156068044] 
\draw    (256.14,250.4) -- (188.41,250.4) ;
\draw [shift={(186.41,250.4)}, rotate = 360] [color={rgb, 255:red, 0; green, 0; blue, 0 }  ][line width=0.75]    (10.93,-3.29) .. controls (6.95,-1.4) and (3.31,-0.3) .. (0,0) .. controls (3.31,0.3) and (6.95,1.4) .. (10.93,3.29)   ;
%Straight Lines [id:da2650798268721517] 
\draw    (304.45,249.49) -- (370.5,249.49) ;
\draw [shift={(372.5,249.49)}, rotate = 180] [color={rgb, 255:red, 0; green, 0; blue, 0 }  ][line width=0.75]    (10.93,-3.29) .. controls (6.95,-1.4) and (3.31,-0.3) .. (0,0) .. controls (3.31,0.3) and (6.95,1.4) .. (10.93,3.29)   ;

% Text Node
\draw (253.97,101.03) node [anchor=north west][inner sep=0.75pt]   [align=left] { $\displaystyle \pi _{1} M_{u}$};
% Text Node
\draw (191.34,167.48) node [anchor=north west][inner sep=0.75pt]   [align=left] { $\displaystyle \pi _{1} T^{2}$};
% Text Node
\draw (132.79,239.57) node [anchor=north west][inner sep=0.75pt]   [align=left] { $\displaystyle \pi _{1} M_{v}$};
% Text Node
\draw (309.84,167.48) node [anchor=north west][inner sep=0.75pt]   [align=left] { $\displaystyle \pi _{1} T^{2}$};
% Text Node
\draw (374.2,237.75) node [anchor=north west][inner sep=0.75pt]   [align=left] { $\displaystyle \pi _{1} M_{w}$};
% Text Node
\draw (257.88,237.66) node [anchor=north west][inner sep=0.75pt]   [align=left] { $\displaystyle \pi _{1} T^{2}$};

\end{tikzpicture}
}
\caption{JSJ-graph of groups}
\label{FIG: JSJ-graph of groups}
\end{figure}
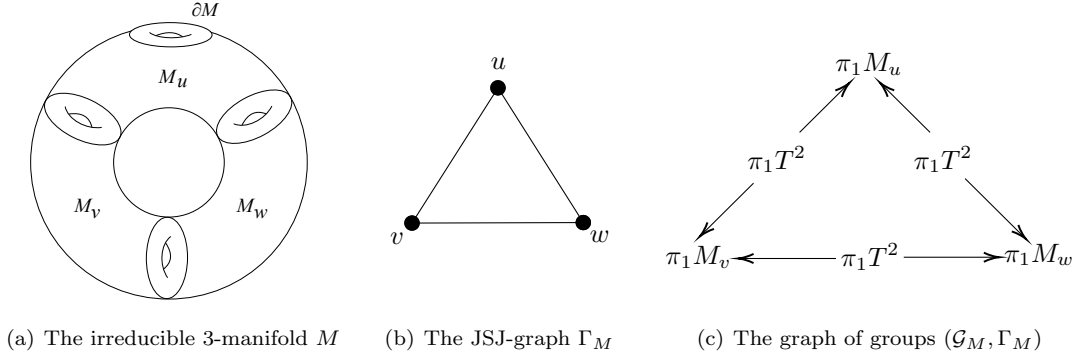

\begin{proposition}[{\cite[Theorem A]{WZ10}}]
Let $M$ be compact, orientable, irreducible 3-manifold with empty or incompressible toral boundary. For each $x\in \Gamma_M=V(\Gamma_M)\cup E(\Gamma_M)$, $\pi_1M$ induces the full profinite topology on $\mathcal{G}_x$. Thus, the homomorphism $\widehat{\mathcal{G}_x}\to \widehat{\pi_1M}$ is injective.
\end{proposition}
The proof of \cite[Theorem A]{WZ10} was stated for closed manifolds, while it applies to manifolds with toral boundary with no alteration. 

Recall that for any group $G$ and any element $g\in G$, we use $C_g(x)=gxg^{-1}$ to denote the conjugation. 

\begin{proposition}[{\cite[Theorem 4.3]{WZ19}}]\label{PROP: JSJ}
Let $M$ and $N$ be compact, orientable, irreducible 3-manifolds with empty or incompressible toral boundary, which are not closed $Sol$-manifolds. Let $(\mathcal{G},\Gamma_M)$ and $(\mathcal{H},\Gamma_N)$ denote their JSJ-graphs of groups respectively. Suppose $f:\widehat{\pi_1M}\ttt \widehat{\pi_1N}$ is an isomorphism.
\begin{enumerate}[label=(\arabic*), leftmargin=*]
\item\label{JSJ1} There is a graph isomorphism $\F: \Gamma_M\to \Gamma_N$, which preserves the hyperbolic/Seifert type of the vertices.
\item For each $x\in \Gamma_M$, $f$ induces an isomorphism $f_x: \widehat{\mathcal{G}_x}\ttt\widehat{\mathcal{H}_{\F(x)}}$.
\item\label{JSJ3} For each $x\in \Gamma_M$, there is an element $h_x\in \widehat{\pi_1N}$ such that the following diagram commutes, where $C_{h_x}$ denotes the conjugation by $h_x$.
\begin{equation*}
\begin{tikzcd}
\widehat{\mathcal{G}_x} \arrow[rr, "f_x"] \arrow[d, hook] &                                       & \widehat{\mathcal{H}_{\F(x)}} \arrow[d, hook] \\
\widehat{\pi_1M} \arrow[r, "f"]                           & \widehat{\pi_1N} \arrow[r, "C_{h_x}"] & \widehat{\pi_1N}                             
\end{tikzcd}
\end{equation*}
\item({\cite[Theorem 5.5]{Xu}})\label{JSJ4} For any $e\in E(\Gamma_M)$ whose endpoints are denoted by $u$ and $v$, there exist $g\in \widehat{\mathcal H_{\F(u)}}$ and $h\in \widehat{\mathcal H_{\F(v)}}$ such that the following diagram commutes.
\begin{equation*}
\begin{tikzcd}
\widehat{\mathcal{G}_u} \arrow[d,"C_g\circ f_u"'] & \widehat{\mathcal{G}_e}\cong \widehat{\mathbb{Z}}^2 \arrow[l] \arrow[r] \arrow[d,"f_e"] & \widehat{\mathcal{G}_v} \arrow[d,"C_h \circ f_v"]\\
\widehat{\mathcal{H}_{\F(u)}} & \widehat{\mathcal{H}_{\F(e)}}\cong \widehat{\mathbb{Z}}^2 \arrow[l] \arrow[r] & \widehat{\mathcal H_{\F(v)}}
\end{tikzcd}
\end{equation*}
\end{enumerate}
\end{proposition}
In fact, \cite{WZ19} was originally proven for closed manifolds, but the proof can be easily generalized to the bounded case, see also \cite[Theorem 5.3]{Wil18JSJ}.%  or \cite[Theorem 5.5]{Xu}.  

We are now ready to prove \autoref{THM: Mixed Peripheral regular}.

\newtheorem*{PP}{\autoref{THM: Mixed Peripheral regular}}

\begin{PP}
Let $M$ be a mixed 3-manifold, and let $\partial_1^{hyp}M,\cdots, \partial_n^{hyp}M$ be the boundary components of $M$ belonging to a hyperbolic JSJ-piece. Suppose $N$ is a compact, orientable 3-manifold, and $f:\widehat{\pi_1M}\ttt \widehat{\pi_1N}$ is an isomorphism.
\begin{enumerate}[label=(\arabic*), leftmargin=*]
\item\label{A3.3-1} $N$ is also a mixed 3-manifold.%; in particular, $N$ is irreducible and boundary-incompressible.
\item\label{A3.3-2} There are exactly $n$ boundary components of $N$ which belong to  hyperbolic JSJ-pieces, and we denote these boundary components as $\partial_1^{hyp}N,\cdots,\partial _n^{hyp} N$.
\item\label{A3.3-3} Up to a reordering, $f(\overline{\pi_1\partial_i^{hyp}M})$ is a conjugate of $\overline{\pi_1\partial_i^{hyp}N}$ in $\widehat{\pi_1N}$.
\item\label{A3.3-4} $f$ is peripheral $\Zx$-regular at each $\partial_i^{hyp}M$ ($1\le i\le n$).\iffalse:
$$C_{g_i}\circ f|_{\overline{\pi_1\partial_iM}}:\tensor \pi_1\partial_iM\xrightarrow{\lambda_i\otimes \psi_i} \tensor\pi_1\partial_iN$$
for some $g_i\in \widehat{\pi_1N}$, $\lambda_i\in \Zx$ and some isomorphism $\psi_i\in \mathrm{Hom}_{\Z}(\pi_1\partial_iM,\pi_1\partial_iN)$. %() 
\fi
\end{enumerate}
\end{PP}
\begin{proof}
 \ref{A3.3-1} By \autoref{LEM: irred}, $N$ is  irreducible and has empty or incompressible toral boundary. In addition, by \cite[Theorem 8.4]{WZ17}, $N$ is not a closed $Sol$-manifold. Therefore, $N$ satisfy the requirements for \autoref{PROP: JSJ}. In particular, \autoref{PROP: JSJ} \ref{JSJ1} implies that $N$ is also a mixed 3-manifold.

\ref{A3.3-2} We follow the notation in \autoref{PROP: JSJ}: $(\mathcal{G},\Gamma_M)$ and $(\mathcal{H},\Gamma_N)$ denote the JSJ-graphs of groups,  $\F: \Gamma_M\to \Gamma_N$ denotes the graph isomorphism, and
 $f_x: \widehat{\mathcal{G}_x}\ttt\widehat{\mathcal{H}_{\F(x)}}$ denotes the induced isomorphism for each $x\in \Gamma_M$. For each vertex $v\in V(\Gamma_M)$ corresponding to a hyperbolic piece, according to \autoref{PROP: Detect hyperbolic}, $\widehat{\pi_1M_v}\cong \widehat{\pi_1N_{\F(v)}}$ implies that $M_v$ and $N_{\F(v)}$ have the same number of boundary components. Let $\partial^{hyp}M$ and $\partial^{hyp}N$ denote the boundary components belonging to hyperbolic pieces. Together with the graph isomorphism  $\F:\Gamma_M\to \Gamma_N$, 
\begin{equation*}
n=\sharp \partial^{hyp}M=\sum_{{v\in V(\Gamma_M)}\atop{\text{hyperbolic}}}\left(\sharp \partial M_v-deg_{\Gamma_M}(v)\right)=\sum_{{w\in V(\Gamma_N)}\atop{\text{hyperbolic}}}\left(\sharp \partial N_w-deg_{\Gamma_N}(w)\right)=\sharp \partial^{hyp}N.
\end{equation*}

\ref{A3.3-3} For each hyperbolic vertex $v\in V(\Gamma_M)$, $f_v:\widehat{\pi_1M_v}\ttt \widehat{\pi_1N_{\F(v)}} $ respects the peripheral structure according to \cite[Lemma 6.3]{Xu}. In fact, according to    \autoref{PROP: JSJ}~\ref{JSJ4}, $f_v$ matches up the boundary components of $M_v$ coming from the JSJ-tori (ie corresponding to the edges  $e\in E(\Gamma_M)$ adjoining $v$) to the boundary components of $N_{\F(v)}$ coming from the JSJ-tori (ie corresponding to the edges $\F(e)$). Thus, $f_v$ matches up the remaining boundary components, ie matches up $M_v\cap\partial M $ with $N_{\F(v)}\cap \partial N$. Therefore, up to reordering, $f_v$ sends $\overline{\pi_1\partial_i^{hyp}M}$ to a conjugate of $\overline{\pi_1\partial_i^{hyp}N}$ in $\widehat{\pi_1N_{\F(v)}}$. The commutative diagram in \autoref{PROP: JSJ} \ref{JSJ3} then implies that $f$ in fact sends $\overline{\pi_1\partial_i^{hyp}M}$ to a conjugate of $\overline{\pi_1\partial_i^{hyp}N}$ in $\widehat{\pi_1N}$.

\ref{A3.3-4} Based on our proof for \ref{A3.3-3}, item \ref{A3.3-4} now  follows from the fact that $f_v$ is peripheral $\Zx$-regular at $\partial_i^{hyp}M$ (\autoref{hypMfd}), together with the commutative diagram in  \autoref{PROP: JSJ}~\ref{JSJ3}.
\end{proof}

\bibliographystyle{amsplain}
\providecommand{\bysame}{\leavevmode\hbox to3em{\hrulefill}\thinspace}
\providecommand{\MR}{\relax\ifhmode\unskip\space\fi MR }
% \MRhref is called by the amsart/book/proc definition of \MR.
\providecommand{\MRhref}[2]{%
  \href{http://www.ams.org/mathscinet-getitem?mr=#1}{#2}
}
\providecommand{\href}[2]{#2}

\end{sloppypar}
\end{document}